\documentclass[11pt]{article}

\usepackage{amsmath, amsthm}
\usepackage{amsmath, amsfonts}
\usepackage{amsmath, amssymb}
\usepackage{amsmath}
\usepackage{graphics}
\usepackage{graphicx}
\usepackage{color}
\usepackage{hyperref}
\usepackage[all]{xy}

\newtheorem{theorem}{Theorem}[subsection]
\newtheorem{definition}[theorem]{Definition}
\newtheorem{definition-lemma}[theorem]{Definition/Lemma}
\newtheorem{definition-explanation}[theorem]{Definition/Explanation}
\newtheorem{explanation-definition}[theorem]{Explanation/Definition}
\newtheorem{definition-fact}[theorem]{Definition/Fact}
\newtheorem{definition-notation}[theorem]{Definition/Notation}
\newtheorem{lemma}[theorem]{Lemma}
\newtheorem{lemma-definition}[theorem]{Lemma/Definition}
\newtheorem{proposition}[theorem]{Proposition}

\newtheorem{remark}[theorem]{\it Remark}
\newtheorem{remark-notation}[theorem]{\it Remark/Notation}

\newtheorem{convention}[theorem]{\it Convention}

\newtheorem{example-definition}[theorem]{Example/Definition}

\newtheorem{definition-prototype}[theorem]{Definition-Prototype}

\numberwithin{equation}{subsection}

\newtheorem{stheorem}{Theorem}[section]
\newtheorem{sdefinition}[stheorem]{Definition}
\newtheorem{sdefinition-lemma}[stheorem]{Definition/Lemma}
\newtheorem{sdefinition-explanation}[stheorem]{Definition/Explanation}
\newtheorem{sexplanation-definition}[stheorem]{Explanation/Definition}
\newtheorem{sdefinition-fact}[stheorem]{Definition/Fact}
\newtheorem{sdefinition-notation}[theorem]{Definition/Notation}
\newtheorem{slemma}[stheorem]{Lemma}
\newtheorem{slemma-definition}[stheorem]{Lemma/Definition}

\newtheorem{sremark-notation}[stheorem]{\it Remark/Notation}

\newtheorem{sassumption}[stheorem]{\it Assumption}

\newtheorem{sexample-definition}[stheorem]{Example/Definition}

\newtheorem{sdefinition-prototype}[stheorem]{Definition-Prototype}

\newtheorem{ssdefinition-lemma}[sstheorem]{Definition/Lemma}
\newtheorem{ssdefinition-explanation}[sstheorem]{Definition/Explanation}
\newtheorem{ssexplanation-definition}[sstheorem]{Explanation/Definition}
\newtheorem{ssdefinition-fact}[stheorem]{Definition/Fact}
\newtheorem{ssdefinition-notation}[theorem]{Definition/Notation}

\newtheorem{sslemma-definition}[sstheorem]{Lemma/Definition}

\newtheorem{ssremark-notation}[sstheorem]{\it Remark/Notation}

\newtheorem{ssexample-definition}[sstheorem]{Example/Definition}

\newtheorem{ssdefinition-prototype}[sstheorem]{Definition-Prototype}

\newcommand{\Ann}{\mbox{\it Ann}\,}

\newcommand{\Coh}{\mbox{\it Coh}\,}
 \newcommand{\CohCategory}{\mbox{\it ${\cal C}$\!oh}\,}

\newcommand{\Coker}{\mbox{\it Coker}\,}

\newcommand{\Endsheaf}{\mbox{\it ${\cal E}\!$nd}\,}
 
\newcommand{\Err}{\mbox{\it Err}\,}

\newcommand{\FM}{\mbox{${\cal F}$\!${\cal M}$}\,}
 \newcommand{\smallFM}{\mbox{\small${\cal F}$\!${\cal M}$}\,}

\newcommand{\Gr}{\mbox{\it Gr}\,}

\newcommand{\Id}{\mbox{\it Id}\,}
 \newcommand{\scriptsizeId}{\mbox{\scriptsize\it Id}}

\newcommand{\Imaginary}{\mbox{\it Im}\,}
\newcommand{\Image}{\mbox{\it Im}\,}

\newcommand{\Ker}{\mbox{\it Ker}\,}

\newcommand{\Proj}{\mbox{\it Proj}\,}

\newcommand{\Quot}{\mbox{\it Quot}\,}

\newcommand{\Real}{\mbox{\it Re}\,}

\newcommand{\Spec}{\mbox{\it Spec}\,}
 \newcommand{\boldSpec}{\mbox{\it\bf Spec}\,}

\newcommand{\Supp}{\mbox{\it Supp}\,}

\newcommand{\Tor}{\mbox{\it Tor}}

 \newcommand{\scriptsizeZss}{\mbox{\scriptsize\it $Z$-$ss$}\,}
 \newcommand{\tinyZss}{\mbox{\tiny\it $Z$-$ss$}\,}

\newcommand{\ch}{\mbox{\it ch}\,}

\newcommand{\degree}{\mbox{\it deg}\,}

\newcommand{\determinant}{\mbox{\it det}\,}
\newcommand{\dimm}{\mbox{\it dim}\,}

\newcommand{\flatscriptsize}{\mbox{\scriptsize\it flat}\,}

\newcommand{\pr}{\mbox{\it pr}}
 \newcommand{\prscriptsize}{\mbox{\scriptsize\it pr}}

\newcommand{\redscriptsize}{\mbox{\scriptsize\rm red}\,}

\newcommand{\singularscriptsize}{\mbox{\scriptsize\it singular}\,}

\newcommand{\td}{\mbox{\it td}\,}

\newcommand{\torsionscriptsize}{\mbox{\scriptsize\it torsion}\,}
\newcommand{\torsionfreescriptsize}{\mbox{\scriptsize\it torsion-free}\,}

\begin{document}

\enlargethispage{24cm}

\begin{titlepage}

$ $

\vspace{-1.5cm} % Re: -1.5cm for PC; -2.5cm for UT-Math-system

\noindent\hspace{-1cm}
\parbox{6cm}{\small August 2013}\
   \hspace{6cm}\
   \parbox[t]{6cm}{yymm.nnnn [math.AG] \\
                D(10.2): D-string ws instanton,  \\
                $\mbox{\hspace{3.8em}} $ 
				$Z$-semistable morphism
				}

\vspace{2cm}

%title
\centerline{\large\bf
  A mathematical theory of D-string world-sheet instantons,}
\vspace{1ex}
\centerline{\large\bf
 II: Moduli stack of $Z$-(semi)stable morphisms from Azumaya nodal curves }
\vspace{1ex}
\centerline{\large\bf
  with a fundamental module to a projective Calabi-Yau $3$-fold}
% end-title

%\bigskip

\vspace{3em}

%authors-'n-addresses
\centerline{\large
  Chien-Hao Liu
   \hspace{1ex} and \hspace{1ex}
  Shing-Tung Yau
}

\vspace{4em}

%abstract%
\begin{quotation}
\centerline{\bf Abstract}

\vspace{0.3cm}

\baselineskip 12pt  %13pt for [12pt] style
{\small
 In this Part II, D(10.2), of D(10), 
 we take D(10.1) (arXiv:1302.2054 [math.AG]) 
  as the foundation to define the notion of  $Z$-semistable morphisms 
   from general Azumaya nodal curves,  of genus $\ge 2$, with a fundamental module
   to a projective Calabi-Yau $3$-fold
   and show that the moduli stack of such $Z$-semistable morphisms of a fixed type is compact. 
 This gives us a counter moduli stack to D-strings 
   as the moduli stack of stable maps in Gromov-Witten theory to the fundamental string. 
 It serves and prepares for us the basis toward a new invariant of Calabi-Yau $3$-fold
   that captures soft-D-string world-sheet instanton numbers in superstring theory.       
 This note is written hand-in-hand with D(10.1) and is to be read side-by-side with ibidem.  
 } % endsmall
\end{quotation}

%\bigskip
\vspace{9em}

\baselineskip 12pt
{\footnotesize
\noindent
{\bf Key words:} \parbox[t]{14cm}{D-string world-sheet instanton, central charge;
      Azumaya nodal curve, fundamental module,\\ $Z$-semistable morphism; 
	  bubbling ${\Bbb P}^1$-tree, positivity; compactness of moduli.
 }} %end-footnotesize

\bigskip

\noindent {\small MSC number 2010: 14D20, 81T30, 14F05, 14N35, 14A22.
} % end-small

\bigskip

\baselineskip 10pt
% Re: 11pt for [11pt] style; 12pt for [12pt] style
{\scriptsize
\noindent{\bf Acknowledgements.}
For this Part II in D(10), we thank 
 D.S.~Nagaraj, C.S.~Seshadri and Jun Li, Baosen Wu
    for their works [N-S: II]  and [L-W] that influence our thought on the problem.  
C.-H.L.\ thanks in addition
 Baosen~Wu 
    for lectures/discussions on [L-W] and literature guide;   
 Montserrat Teixidor i Bigas 
    for conversation on universal vector bundles on nodal curves;
 Siu-Cheong Lau
     for discussions on symplectic aspect of D-branes and related open Gromov-Witten theory;  	
 Yng-Ing Lee, Chung-Jun Tsai
     for discussions on special Lagrangian cycles;  
 Murad Alim, Constantin Bachas, Pei-Ming Ho, Albrecht Klemm, Cumrun Vafa
     for discussions on related stringy issues and literature guide;
 Eric Sharpe for a communication; 	 
 Hai-Chau Chang, Heng-Yu Chen, Toshiaki Fujimori, Hui-Wen Lin, Ai-Nung Wang, 
 Chin-Lung Wang
     and participants of the Geometry Seminar and the String Theory Seminar
	 at National Taiwan University for conversations/discussions on this project, May 2013;
 Melody Tung Chan,  Siu-Cheong Lau, Cumrun Vafa
   for the enlightening topic courses, spring 2013;
 Yu-jong Tzeng 
   for instruction on Beamer (LaTeX) for presentation; 
 Chieh-Hsiung Kuan, Silika Prohl, Lanfang Wu
   for conversations;
 Department of Mathematics and Department of Physics at National Taiwan University  and 
 National Center for Theoretical Sciences - Mathematics Division, Taipei Office,
   for hospitality, May 2013, while part of the work is in preparation;
 Alexandra Grot
   for Bach that accompanies the typing of the notes; 
 Ling-Miao Chou for moral support.
S.-T.Y.\ thanks in addition
 Department of Mathematics at National Taiwan University
   for hospitality, fall 2013, while part of the manuscript is in preparation.
The project is supported by NSF grants DMS-9803347 and DMS-0074329.
} %endscriptsize

\end{titlepage}

\newpage

\begin{titlepage}

$ $

\vspace{12em} % \vspace{4em}

\centerline{\small\it
 Chien-Hao Liu dedicates D(10.1)and D(10.2) to}
\centerline{\small\it
 his uncle Prof.~Pin-Hsiung Liu (1925 -- 2004), aunt Ms.~Rui-Be Lin,}
\centerline{\small\it
 and cousin Master Guo-Guang Shr (Te-Ru Liu)}
\centerline{\small\it
 for their hospitality in his first year of college,}
\centerline{\small\it
 years of communications with Te-Ru,}
\centerline{\small\it
 and lots of cherished memories.}

% \vspace{36em}
%
% \baselineskip 11pt
%
%{\footnotesize
% \noindent
% $^{\dagger}$Deceased, winter  2004.  (1925 -- 2004)
% } % end-footnotesize

\end{titlepage}

%paper

\newpage
$ $

\vspace{-3em}
% \vspace{-4em}  % Re: -4cm for PC; -6cm for UT-Math-system

%short heading
\centerline{\sc
 D-string world-sheet instantons II: Stack of $Z$-Semistable Morphisms
 } %

\vspace{2em}

% \baselineskip 14pt  %Re: 14pt for [11pt] style
                      %Re: 15pt for [12pt] style.

\begin{flushleft}
{\Large\bf 0. Introduction and outline}
\end{flushleft}
 In a suitable regime of superstring theory, D-branes in a Calabi-Yau space
   and their most fundamental behaviors can be nicely described
   mathematically through morphisms from Azumaya spaces
   with a fundamental module to that Calabi-Yau space.
 In the earlier work [L-L-S-Y]
    (D(2): arXiv:0809.2121 [math.AG], with Si Li and Ruifang Song)
	 from the project,
   we explored this notion for the case of D1-branes (i.e.\ D-strings)  and
    laid down some basic ingredients toward understanding
	  the notion of D-string world-sheet instantons in this context.
  In this continuation, D(10), of D(2), we move on to construct
   a moduli stack of semistable morphisms from Azumaya nodal curves
   with a fundamental module to a projective Calabi-Yau $3$-fold $Y$.   
 In Part I of the note, D(10.1),
   we defined
              the notion of twisted central charge $Z$
			     for Fourier-Mukai transforms of dimension $1$ and width $[0]$
				  from nodal curves
            and the associated stability condition on such transforms
   and proved that for a given compact stack of nodal curves $C_{\cal M}/{\cal M}$,
     the stack $\smallFM^{1,[0];\tinyZss}_{C_{\cal M}/{\cal M}}(Y,c)$
	 of $Z$-semistable Fourier-Mukai transforms
	 of dimension $1$ and width $[0]$
	  from nodal curves in the family $C_{\cal M }/{\cal M}$
	  to $Y$ of fixed twisted central charge $c$ is compact.
	  
 In the current Part II (D(10.2)) of D(10),
  we take D(10.1) as the foundation to define the notion of  $Z$-semistable morphisms 
   from general Azumaya nodal curves, of genus $\ge 2$, with a fundamental module
   to a projective Calabi-Yau $3$-fold 
   (Definition~\ref{Zssm})
      % Definition [$Z$-(semi)stable morphism]
   and show that the moduli stack of such $Z$-semistable morphisms of a fixed type is compact
   (Theorem~4.0). 
      % Theorem [${\frak M}_{Az^{\!f}\!(g;r,\chi)}^{\scriptsizeZss}(Y;\beta,c)$ compact]
 This gives us a counter moduli stack to D-strings 
   as the moduli stack of stable maps in Gromov-Witten theory to the fundamental string. 
 It serves and prepares for us the basis toward a new invariant of Calabi-Yau $3$-fold
   that captures {\it soft-D-string world-sheet instanton numbers} in superstring theory. 
(Cf.\ Sec.~\ref{review}, Theme `{\it Issues on stability conditions for morphisms 
           from general Azumaya nodal curves with a fundamental module}'.)

\bigskip
  
\noindent{\it Remark 0.1.}{\it $[$work of Nagaraj-Seshadri and Li-Wu$]$.}	
 This is a remark for experts on the issue of moduli spaces of vector bundles 
   on nodal curves and their degenerations.
 Our notion of $Z$-semistable morphisms from Azumaya nodal curves with a fundamental module 
  to a Calabi-Yau $3$-fold is a generalization of the construction from 
  {\it D.S.~Nagaraj} and {\it C.S.~Seshadri}
  of generalized Gieseker moduli spaces in their study of degenerations of moduli spaces
  of vector bundles on curves ([N-S, II]).  
 Conditions (1), (2), and (3) in Definition~\ref{Zssm} in Sec.~3.1 are indeed 
  a generalization of an abstraction of their related study [N-S, II: Sec.~2]: 
    ``{\it Vector bundles over the curves $X_k$}", where $X_k$ is a nodal curve from 
	bubbling off a node of a stable curve by attaching a ${\Bbb P}^1$-chain of length $k$.
 In this regard, Sec.~2.1 and Sec.~2.2 of this note may be regarded 
  as a counter part to part of [N-S, II: Sec.~2].)	
 On the other hand, with the requirement of compactness of moduli space however constructed in mind,
  we need something to help demonstrate/control how far a general Fourier-Mukai transform
   that appears in the reduction is from being in our category.
 For this we learned from {\it Baosen Wu} his work with {\it Jun Li} [L-W] 
   the technique of error Hilbert polynomials in their case (turned to error charges in our case).
 In this regard, Sec.~2.3 and Sec.~4.2 are the parallel to their study of error Hilbert polynomials
   of sheaves in a degeneration in 
   [L-W: Sec.~3.3 ``{\it Numerical criterion}" 
             and part of Sec.~5 ``{\it Properness of the moduli stacks}"]. 
 We acknowledge also the works
   [Ca], [Gi2], [Kau], [K-L], [M-O-P], [Pa], [P-R], [Sch], [Sun], and [TiB]
  for related studies that have influenced us in the brewing years since spring 2008. 
%end-remark

 \bigskip

\bigskip

\noindent
{\bf Convention.}
 Standard notations, terminology, operations, facts in
  (1) stacks;$\,$
  (2) moduli spaces of sheaves;$\,$
  (3) cohomological techniques in algebraic geometry
 can be found respectively in$\,$
  (1) [L-MB];$\,$
  (2) [H-L];$\,$
  (3) [Gro2], [Ha], [EGA$_{\mbox{\scriptsize III}}$].
 \begin{itemize}
  \item[$\cdot$]
   All varieties, schemes and their products are over ${\Bbb C}$;
   a `{\it curve}' means a $1$-dimensional proper scheme over ${\Bbb C}$.
   % a `{\it stack}' means an {\it Artin stack}.

  \item[$\cdot$]
   The `{\it support}' $\Supp({\cal F})$
    of a coherent sheaf ${\cal F}$ on a scheme $Y$
    means the {\it scheme-theoretical support} of ${\cal F}$
   unless otherwise noted;
   ${\cal I}_Z$ denotes the {\it ideal sheaf} of
    a subscheme of $Z$ of a scheme $Y$;
  $l({\cal F})$ denotes the {\it length} of a coherent sheaf ${\cal F}$ of dimension $0$.

  \item[$\cdot$]
   The current note continues the study in
    [L-L-S-Y] (arXiv:0809.2121 [math.AG], D(2))
	 and [L-Y3] (arXiv:1302.2054 [math.AG] D(10.1)).
   A partial review of D-branes and Azumaya noncommutative geometry
    is given in [L-Y3] (arXiv:1003.1178 [math.SG], D(6)) and
    [Liu1] (arXiv:1112.4317 [math.AG])
	(see also [Liu2] and [Liu3]).
   Notations and conventions follow these earlier works when applicable.
 \end{itemize}

\bigskip

\bigskip
%\newpage
   
\begin{flushleft}
{\bf Outline}
\end{flushleft}
\nopagebreak
{\small
\baselineskip 12pt  %13pt
\begin{itemize}
 \item[0.]
  Introduction.
  
 \item[1.]
  D-string world-sheet instantons, 
  morphisms from Azumaya nodal curves with a fundamental module,
  Fourier-Mukai transforms, and issues on stability conditions.
  
 \item[2.]
 Preliminaries aiming for a definition of stability conditions for morphisms
                    from general Azumaya nodal curves with a fundamental module.					
  \vspace{-.6ex}
  \begin{itemize}
   \item[2.1]
    Push-forward and higher direct image of sheaves.

   \item[2.2]
   A special class of  morphisms from Azumaya curves with a fundamental module and
   positivity of fundamental modules on ${\Bbb P}^1$-trees.
   
   \item[2.3]
   Remarks on nonnegative torsion-free sheaves on a ${\Bbb P}^1$-tree.
  \end{itemize}
  
 \item[3.]
  The space of D-string world-sheet instantons:
  The moduli stack of $Z$-semistable morphisms\\ from Azumaya nodal curves
  with a fundamental module to a Calabi-Yau 3-fold.
  \vspace{-.6ex}
  \begin{itemize}
   \item[3.1]
    The moduli stack of $Z$-semistable morphisms from Azumaya nodal curves
    with a fundamental module to a Calabi-Yau 3-fold.
	
  \item[3.2]	
    A natural morphism
	  from ${\frak M}_{Az^{\!f}\!(g;r,\chi)}^{\tinyZss}(Y; \beta, c)$
	  to $\FM_g^{1,[0];\scriptsizeZss}(Y;c)$.
  \end{itemize}

 \item[4.]
  Compactness of the moduli stack
   ${\frak M}_{Az^{\!f}\!(g;r,\chi)}^{\tinyZss}(Y; \beta, c)$
   of $Z$-semistable morphisms.
   \vspace{-.6ex}
   \begin{itemize}
     \item[4.1]
	  Boundedness of
	  ${\frak M}_{Az^{\!f}\!(g;r,\chi)}^{\tinyZss}(Y; \beta, c)$.
	
     \item[4.2]
	 The error charge of a Fourier-Mukai transform.
	
     \item[4.3]
	 Completeness of
	  ${\frak M}_{Az^{\!f}\!(g;r,\chi)}^{\tinyZss}(Y; \beta, c)$.
   \end{itemize}

 %%%%%%%%%%%%%%%%%%%%%%%
 % \item[5.]
 %   Stability condition revisited.
 %%%%%%%%%%%%%%%%%%%%%%%
   
 %%%%%%%%%%%%%%%%%%%%%%%%
 %
 % \item[?.]
 %  ??????????????????????.
 %  %
 %  \vspace{-.6ex}
 %  \begin{itemize}
 %   \item[?.?]
 %    ??????????????.
 %
 %   \item[?.?]
 %    ?????????????
 %    %
 %    \begin{itemize}
 %     \item[?.?.?]
 %      ???????????????.
 %    \end{itemize}
 %
 %  \end{itemize}
 %
 %
 % \item[]\hspace{-1.7em}
 %    Appendix.
 %    ??????
 %%%%%%%%%%%%%%%%%%%%%%%%%
\end{itemize}
} %endsmall

\newpage

\section{D-string world-sheet instantons,
  morphisms from Azumaya nodal curves with a fundamental module,
  Fourier-Mukai transforms, and issues on stability conditions}
\label{review}
We review in this section related notions from
   [L-L-S-Y] (D(2)) and [L-Y3] (D(10.1))
  to bring out terminologies and notations needed.
Readers are referred to ibidem for details and references.

\bigskip

\begin{flushleft}
{\bf Holomorphic D-branes, morphisms from Azumaya schemes with a fundamental module,
          and Fourier-Mukai transforms}
\end{flushleft}
A holomorphic D-brane in superstring theory can be described
 as a morphism from a scheme with a matrix-type noncommutative structure sheaf
 (i.e. an Azumaya scheme $(X,{\cal O}_X^{A\!z})$)
   together with a fundamental ${\cal O}_X^{A\!z}$-module ${\cal E}$
 to $(Y, {\cal O}_Y)$, where ${\cal O}_Y$ is the structure sheaf of $Y$
  in either commutative or noncommutative setting; in notation/symbol,
 $$
   \varphi\; :\;  (X,{\cal O}_X^{A\!z}; {\cal E})\;
     \longrightarrow\;   (Y,{\cal O}_Y)\,,
 $$
 with a built-in isomorphism
   ${\cal O}_X^{A\!z}\simeq \Endsheaf_{{\cal O}_X}({\cal E})$.
In true contents, this means
 a contravariant gluing system of ring-homomorphisms
 $$
    {\cal O}_X^{A\!z}\;
	   \longleftarrow\; {\cal O}_Y\; :\; \varphi^{\sharp}\,,
$$
which in general {\it does not} induce any morphisms directly
   from $X$ to $Y$.
It is through $\varphi^{\sharp}$ that
 the ${\cal O}_X^{A\!z}$-module ${\cal E}$
 can be pushed forward to an ${\cal O}_Y$-module,
 in notation $\varphi_{\ast}({\cal E})$, on $Y$.
 
When the target space $Y$ is a commutative scheme
  and ${\cal E}$ is locally free ${\cal O}_X$-module,
then
 associated to a morphism
  $\varphi :(X,{\cal O}_X^{A\!z};{\cal E})
                                               \rightarrow (Y,{\cal O}_Y)$
  is the following diagram
  $$
   \xymatrix{
   {\cal O}_X^{A\!z}= \Endsheaf_{{\cal O}_X}({\cal E})\\
        {\cal A}_{\varphi}\;:=\,
         		\Image\varphi^{\sharp}\ar@{^{(}->}[u]
                   				&&& {\cal O}_Y\ar[lll]_-{\varphi^{\sharp}}\\
   {\cal O}_X\rule{0ex}{3ex}  \ar@{^{(}->}[u]								&&&&,
   }
  $$
  which defines a subscheme
    $X_{\varphi}:= \boldSpec{\cal A}_{\varphi}\subset X\times Y$
	together with a coherent sheaf $\tilde{\cal E}_{\varphi}$
	 supported on $X_{\varphi}$,
   which is simply the ${\cal O}_X^{A\!z}$-module ${\cal E}$
    regarded as an ${\cal A}_{\varphi}$-module.
 $\tilde{\cal E}_{\varphi}$ is called the {\it graph} of  the morphism $\varphi$.
It is a coherent sheaf on $X\times Y$
 that is flat over $X$, of relative dimension $0$.
Conversely, given such a coherent sheaf $\tilde{\cal E}$ on $X\times Y$,
 a morphism
   $\varphi_{\tilde{\cal E}}:(X,{\cal O}_X^{A\!z};{\cal E})\rightarrow Y$
 can be constructed from $\tilde{\cal E}$ by taking
 \begin{itemize}
  \item[$\cdot$]
    ${\cal E}=\pr_{1\,\ast}\tilde{\cal E}$,
	
  \item[$\cdot$] 	
   ${\cal O}_X^{A\!z}=\Endsheaf_{{\cal O}_X}({\cal E})$,  and

  \item[$\cdot$]
   $\varphi_{\tilde{\cal E}}^{\;\sharp}:
     {\cal O}_Y\rightarrow {\cal O}_X^{A\!z}\,$
   is defined by the composition
   $$
     {\cal O}_Y
        \xrightarrow{\hspace{1em}pr_2^{\sharp}\hspace{1em}}
      {\cal O}_{X\times Y}
        \xrightarrow{\hspace{1.2em}\iota^{\sharp}\hspace{1.2em}}
		{\cal O}_{Supp(\tilde{\cal E})}\;
		\hookrightarrow\; {\cal O}_X^{A\!z}\,.
   $$
 \end{itemize}
Here,
  $X\xleftarrow{\pr_1} X\times Y \xrightarrow{\pr_2} $
      are the projection maps,
  $\iota:\Supp(\tilde{\cal E})\rightarrow X\times Y$	
    is the embedding of the subscheme,   and
note that $\Supp(\tilde{\cal E})$ is affine over $X$.

Treating $\tilde{\cal E}$ as an object in the bounded derived category
 $D^b(\Coh(X\times Y))$ of coherent sheaves on $X\times Y$,
 $\tilde{\cal E}$ defines a Fourier-Mukai transform
 $\Phi_{\tilde{\cal F}} : D^b(\Coh(X))\rightarrow D^b(\Coh(Y))$,
 in short name, a {\it Fourier-Mukai transform from $X$ to $Y$}.
In this way, the data that specifies a morphism
 $\varphi:(X,{\cal O}_X^{A\!z};{\cal E})\rightarrow Y$
 is matched to a data that specifies a special kind of Fourier-Mukai transform.
 
\bigskip

\begin{sdefinition}\label{sdFM}
{\bf [support, dimension, width of Fourier-Mukai transform].}
{\rm ([L-Y3: Definition~1.1].)}
{\rm
For a general $\tilde{\cal F}^{\bullet}\in D^b(\Coh(X\times Y))$,
 we define
   the {\it (scheme-theoretical) support} $\Supp(\tilde{\cal F}^{\bullet})$
       of  $\tilde{\cal F}^{\bullet}$
     to be the (scheme-theoretical) support of
	   $\oplus_iH^i(\tilde{\cal F}^{\bullet})$,
   the {\it dimension} $\dimm\tilde{\cal F}^{\bullet}$
      of $\tilde{\cal F}^{\bullet}$
     to be the dimension $\dimm(\Supp({\cal F}^{\bullet}))$,  and
   the {\it width} of  $\tilde{\cal F}^{\bullet}$ to be the interval
      $[i, j]$ such that
	    $H^i(\tilde{\cal F}^{\bullet})\ne 0$,
	    $H^j(\tilde{\cal F}^{\bullet})\ne 0$,  	and
		$H^k(\tilde{\cal F}^{\bullet})= 0$,  for $k\notin [i,j]$.
 We'll denote the width $[i,i]$ by $[i]$.		
}\end{sdefinition}

\bigskip

\noindent
Thus, for $X$ fixed of pure dimension $d$,
   the stack of morphisms $(X,{\cal O}_X^{A\!z};{\cal E})\rightarrow Y$
   is embedded in the stack of Fourier-Mukai transforms from $X$ to $Y$
   of dimension $d$ and width $[0]$;
 the latter is identical to the stack of $d$-dimensional  coherent sheaves on $X\times Y$.
Similar statement holds for $X$ not fixed.
({\sc Figure} 1-1.)
\begin{figure} [htbp]
 \bigskip
 \centering
 \includegraphics[width=0.80\textwidth]{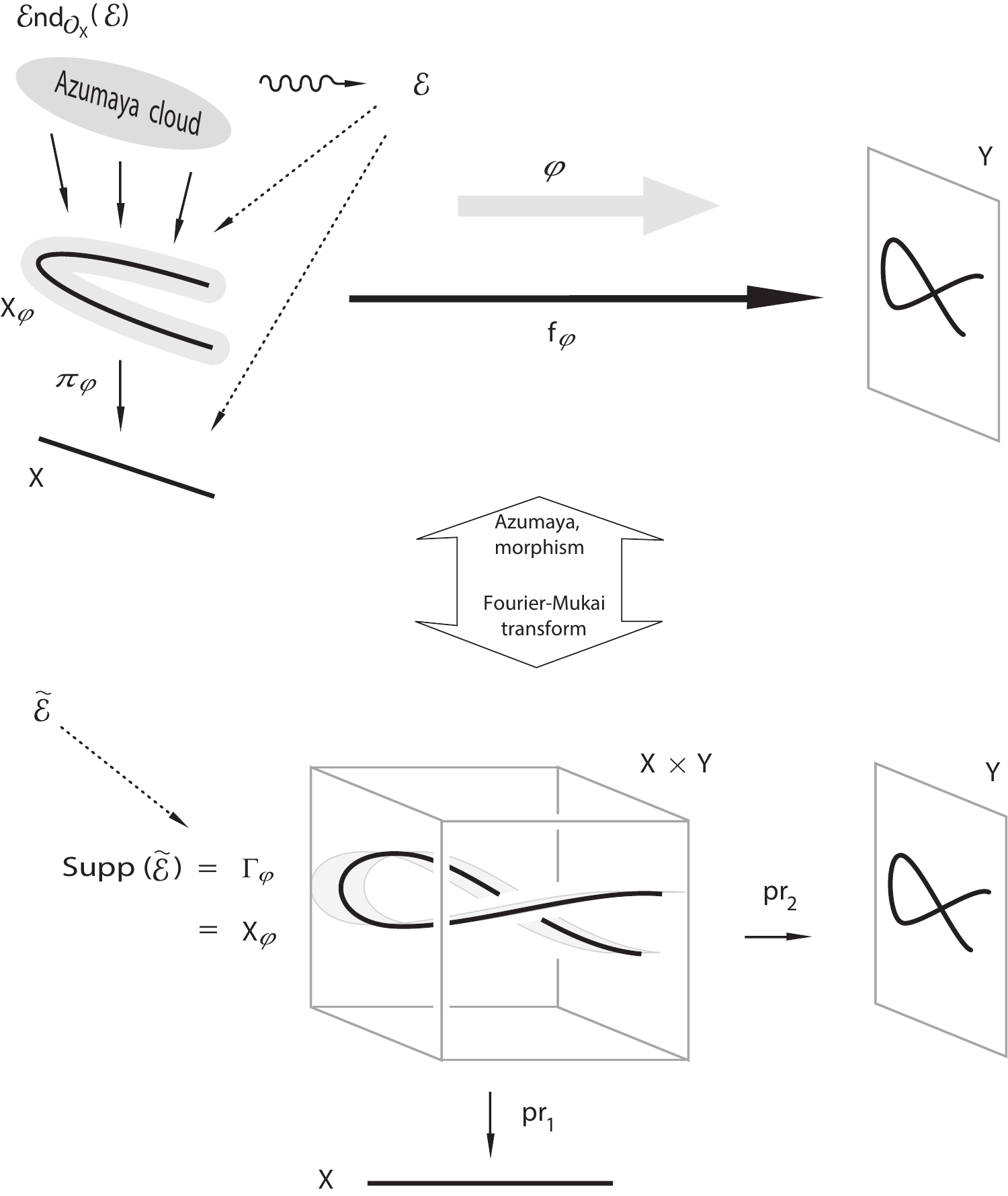}
 
 \vspace{4em}
 \centerline{\parbox{13cm}{\small\baselineskip 12pt
  {\sc Figure}~1-1.
    The equivalence between
      a morphism $\varphi$ from an Azumaya scheme with a fundamental module
	    $(X,{\cal O}_X^{Az}:=\Endsheaf_{{\cal O}_X}({\cal E}); {\cal E})$
	    to a scheme $Y$    and
	  a special kind of Fourier-Mukai transform $\tilde{\cal E}\in \CohCategory(X\times Y)$
	    from $X$ to $Y$.
      }}
 \bigskip
\end{figure}

\bigskip

We now specialize to the objects of this subseries D(10):
{\it The case of D1-branes (i.e.\ D-strings) world-sheet instantons}.

\bigskip

\begin{flushleft}
{\bf D-string world-sheet instantons, central charges, and stability conditions}
\end{flushleft}
In the context of a compactification of Type IIB superstring theory on a Calabi-Yau $3$-fold $Y$,
 an instanton in the effective $4$-dimensional, $N=1$, supersymmetric quantum field theory
 can be created by ``wrapping" a (Euclideanized/Wick-rotated) D-string world-volume
 on some cycles in $Y$; cf.\ {\sc Figure} 1-2.
\begin{figure} [htbp]
 \bigskip
 \centering
 \includegraphics[width=0.80\textwidth]{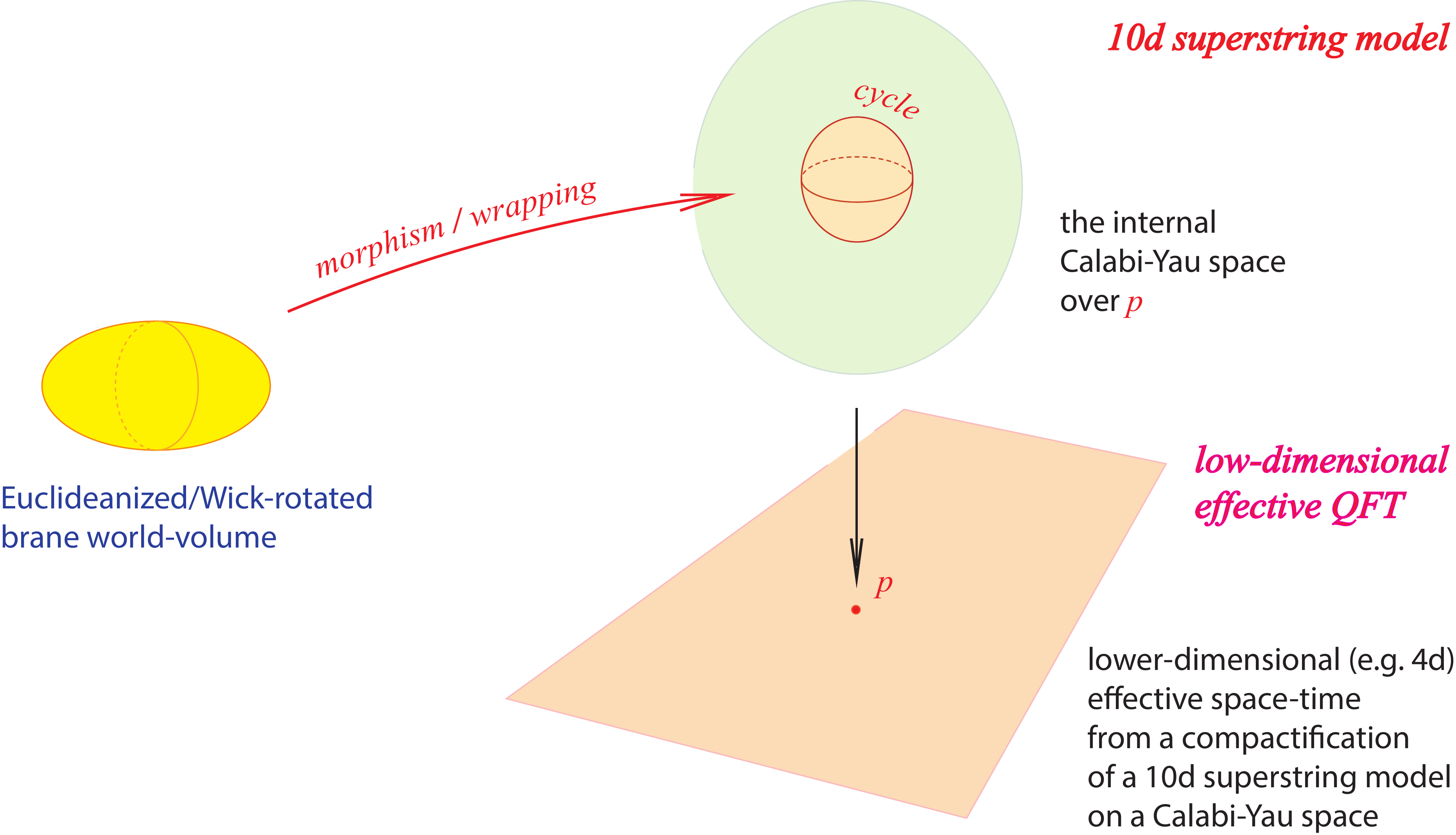}
  
 \vspace{2em}
 \centerline{\parbox{13cm}{\small\baselineskip 12pt
  {\sc Figure}~1-2.
     An instanton in the effective $4$-dimensional quantum field theory is created by
	  ``wrapping"  a Euclideanized/Wick-rotated D-brane world-volume on a cycle of
	  the internal Calabi-Yau space $Y$ in the compactification of a Type II superstring theory.
	 In the weak D-brane tension regime of superstring theory,
      such ``wrapping" can be described by a morphism $\varphi$
	  from an Azumaya space with a fundamental module $(X^{Az},{\cal E})$ to $Y$.
    In particular, for $D$-strings in the Type IIB superstring theory, this gives a description of
      {\it soft D-string world-sheet instantons}.	
    }}
 \bigskip
\end{figure}	
For such instanton to be stable (in the sense of not to decay away by fluctuations as time passes),
  it is required that this Euclidean D-string world-sheet, together with the Chan-Paton sheaf thereupon,
  satisfy some stability conditions governed by the central charge of D-branes.
In [L-Y3] (D(10.1)), these notions are rephrased/polished
 (in an appropriate large-radius limit of the Calabi-Yau space with the tension of the fundamental string
     going to infinity)
  to the following setting along the line of the D-project:

\bigskip
 
Let
  $C$ be a nodal curve with a polarization class $L$  and
  $Y$ be a projective Calabi-Yau manifold with a complexified K\"{a}hler class
    $B+\sqrt{-1}J$.
	
\bigskip
				
\begin{sdefinition}\label{tcFM}
{\bf [twisted central charge of Fourier-Mukai transform].}
{\rm ([L-Y3: Definition~2.1.1].)}
{\rm
 Let
   $\tilde{\cal F}$ be a coherent sheaf of dimension $1$ on $C\times Y$   and
   $\Phi_{\tilde{\cal F}}$ be the Fourier-Mukai transform $\tilde{\cal F}$ defines.
 Then, the {\it twisted central charge} of $\Phi_{\tilde{\cal F}}$
  associated to the data $(B+\sqrt{-1}J,L)$ is defined to be
  $$
   Z^{B+\sqrt{-1}J, L}(\Phi_{\tilde{\cal F}})\;
   :=\; Z^{B+\sqrt{-1}J, L}(\tilde{\cal F)}\;
   :=\;   \int_{C\times Y}\,
                \pr_2^{\ast}\left(
                    \frac{e^{-(B+\sqrt{-1}J)}}{\sqrt{td(T_Y)}}
                                          \right)\,
                \cdot\, \pr_1^{\ast}\, e^{-\sqrt{-1}L}\,
                \cdot\, \tau_{C\times Y}(\tilde{\cal F})\,,
  $$
  where  $\tau_{C\times Y}(\tilde{\cal F}):= \ch(\tilde{\cal F })
                    \cdot \td(T_{C\times Y})$ is the $\tau$-class of $\tilde{\cal F}$.
} \end{sdefinition}		

\bigskip
 
\begin{slemma}\label{tccef}
 {\bf [twisted central charge: explicit form].}
 {\rm ([L-Y3: Lemma~2.1.2].)}
  Continuing the above notation.
  Let
   $$
     \tilde{\beta}(\tilde{\cal F})\; :=\;  \sum_i d_i[\zeta_i]\; \in\;  A_1(C\times Y)\,,
   $$
    where
	  $\zeta_i$ runs through the generic points of $\Supp(\tilde{\cal F})$   and
      $d_i$  is the dimension of $\tilde{\cal F}|_{\zeta_i}$
		   as a $k_{\zeta_i}$-vector space.
  Then,
   $$
     Z^{B+\sqrt{-1}J,L}(\tilde{\cal F})\;
      =\;  \left(\chi(\tilde{\cal F})-B\cdot\tilde{\beta}(\tilde{\cal F})\right)\,
 	        -\,\sqrt{-1}\,\left((J+L)\cdot \tilde{\beta}(\tilde{\cal F})\right)\,.
   $$
  In particular, for non-zero coherent sheaves on $C\times Y$ of dimension $\le 1$,
   $Z^{B+\sqrt{-1}J,L}$ takes its values
    in the partially completed lower-half complex plane
   $$
     \hat{\Bbb H}_-\;
	 :=\; \{ z\in {\Bbb C}\,|\,
                   \mbox{either $\Imaginary z <0$
		                          or $\Imaginary z=0$ with $\Real z >0$}\, \}\,.
   $$
\end{slemma}
  
\bigskip

\begin{sdefinition}\label{slope}
 {\bf [$Z$-slope $\mu^Z$].}
 {\rm ([L-Y3: Definition~2.1.4].)}
 {\rm
 Continuing Definition~\ref{tcFM}.
                % Definition [twisted central charge of Fourier-Mukai transform]
 We define the {\it $Z$-slope} for a non-zero coherent sheaf $\tilde{\cal F}$ on $C\times Y$
  of dimension $\le 1$ to be
  $$
   \begin{array}{crl}
      \mu^Z(\tilde{\cal F})
       &  :=
	   &  - \left.
	          \Real \left(Z^{B+\sqrt{-1}J,L}(\tilde{\cal F})\right)
	          \right /
		      \Imaginary\left(Z^{B+\sqrt{-1}J,L}(\tilde{\cal F})\right)   \\[1.6ex]
     & = 	
	   & \left.
      	     \left(\chi(\tilde{\cal F})-B\cdot\tilde{\beta}(\tilde{\cal F})\right)
		     \right/
 	         \left((J+L)\cdot \tilde{\beta}(\tilde{\cal F})\right)\,.
   \end{array}
  $$
}\end{sdefinition}

\bigskip

\begin{sdefinition}\label{Zss}
{\bf [$Z$-semistable, $Z$-stable, $Z$-unstable, strictly $Z$-semistable].}
{\rm ([L-Y3: Definition~2.2.1].)}
{\rm
 Continuing the discussion.
 \begin{itemize}
  \item[(1)]
   A $1$-dimensional coherent sheaf $\tilde{\cal F}$ on $C\times Y$
     is said to be {\it $Z$-semistable} (resp.\ $Z$-stable)
    if $\tilde{\cal F}$ is pure and
       $\mu^Z(\tilde{\cal F}^{\prime})\le $ (resp.\ $<$)
       $\mu^Z(\tilde{\cal F})$
	   for any nonzero proper subsheaf
	     $\tilde{\cal F }^{\prime}\subset \tilde{\cal F}$.
   Such $\tilde{\cal F}$ is called $Z$-{\it unstable} if it is not $Z$-semistable,   and
   is called {\it strictly $Z$-semistable} if it is $Z$-semistable but not $Z$-stable.
  
  \item[(2)]
    A morphism
	  $\varphi:
	     (C,{\cal O}_C^{Az}:= \Endsheaf_{{\cal O}_C}({\cal E}); {\cal E})
		 \rightarrow Y$
       is said to be {\it $Z$-semistable}
	      (resp.\ {\it $Z$-stable}, {\it $Z$-unstable}, {\it strictly $Z$-semistable})
    if its graph $\tilde{\cal E}_{\varphi}\in \CohCategory_1(C\times Y)$
      is $Z$-semistable	
	   (resp.\ $Z$-stable, $Z$-unstable, strictly $Z$-semistable).
  \end{itemize}
 When the central charge functional  $Z$ is known and fixed either explicitly or implicitly,
  we may use the terminology: {\it semistable, stable, unstable, strictly semistable},
  for simplicity.
}\end{sdefinition}

\bigskip

\begin{flushleft}
{\bf Compactness of the moduli stack of Fourier-Mukai transforms from stable curves}
\end{flushleft}
\begin{sassumption}{$[g\ge 2]$.}  {\rm
{\it For the rest of the notes, we assume that $g\ge 2$}
 and leave the special cases of $g=0$ and $g=1$ to separate notes.
} \end{sassumption}
   
\bigskip

\begin{stheorem}\label{FM-compact}
 {\bf [$\FM_g^{1,[0];\scriptsizeZss}(Y;c)$ compact].}
   {\rm (Cf.\ [L-Y3: Theorem 3.1]).}
   Let
    $(Y, B+\sqrt{-1}J)$ be a projective Calabi-Yau $3$-fold
       with a fixed complexified K\"{a}hler class,
    $\overline{\cal M}_g$  be the moduli stack of stable curves of genus $g$,
    $C_{\overline{\cal M}_g}/\overline{\cal M}_g$
	  be the associated universal curve over $\overline{\cal M}_g$
	   with a fixed relative polarization class $L$.
   Then
     the moduli stack
	    $\FM_g^{1,[0];\scriptsizeZss}(Y;c)$
      of $Z^{B+\sqrt{-1}J,L}$-semi-stable Fourier-Mukai transforms
	     of dimension $1$,  width $[0]$, and central charge $c\in \hat{\Bbb H}_-$
     from stable curves of genus $g$  to $Y$
	 is compact.
\end{stheorem}

\bigskip

\begin{flushleft}
{\bf Issues on stability conditions for morphisms from general Azumaya nodal curves
          with a fundamental module}
\end{flushleft}
%
% {\bf Goal.}
As in Gromov-Witten theory for stable maps,
   the domain Azumaya nodal curves with a fundamental module of
    our intended $Z$-semistable morphisms are not fixed.
 To have a good mathematical theory of D-string world-sheet instantons,
  we want our moduli stack of intended $Z$-semistable morphisms
     from Azumaya nodal curves with a fundamental module to $Y$ to be compact.
 This is a minimal requirement to have a well-defined intersection theory on the moduli stack.
 Only then may one have a chance for a good-enough tangent-obstruction theory,
 at least in some important cases,  to  define {\it D-string world-sheet instanton numbers}
  from a purely D-string world-sheet aspect.
 If achieved, this would give us a counter theory to D-strings as
   Gromov-Witten theory to the fundamental string.
			 	
%{\bf Problems from ${\Bbb P}^1$-tree bubbling.}
 {\it However},
 also as in Gromov-Witten theory,
  degenerations of a naive $Z$-semistable morphism $\varphi$
    in our problem may give rise to objects not in our category.
 For example	and in terms of the graph
     $\tilde{\cal E}_{\varphi}\in \CohCategory_1(C\times Y)$ of $\varphi$,
   $\Supp(\widetilde{\cal E}_{\varphi})$ may turn to have a vertical component
    with respect to the projection map $\pr_C:C\times Y\rightarrow C$ and/or
  ${\cal E}={\pr_C}_{\ast}(\tilde{\cal E}_{\varphi})$
    may turn to a non-locally-free sheaf on a deformation of $C$.
Just considering morphisms alone,
  all such bad degenerations/irregularities of morphisms under a deformation
  can be corrected/absorbed by ${\Bbb P}^1$-bubbling trees added to nodal curves,
  as in Gromov-Witten theory.
However, as there is no known universal estimate to relate the complexity of
  degenerations of intended $Z$-semistable morphisms in our category
  to a bound on the complexity of the ${\Bbb P }^1$-trees needed to absorb
   the bad degeneration,
 there is no way to select beforehand a large enough polarization class $L$
 on nodal curves to guarantee that we never use up its positivity condition
 to define $Z$-semistability condition on morphisms
 when extending  the polarization class on a nodal curve to its cousin with
  ${\Bbb P}^1$-bubbling trees.
This is the same issue mathematicians already ran into
  when studying the Gieseker-type compactifications of moduli spaces
  of vector bundles on stable curves.
		
%{\bf The way out.}
Following the lesson learned from these previous studies on moduli spaces
 of vector bundles on nodal curves
 ([Gi2], [N-S,II], and e.g.\ [Kau], [K-L], [Sch], [Sun], [TiB]),
the way out of such difficulty in our situation is,
  in conceptual, nontechnical, and slightly imprecise words, that  		
   one needs to separate the underlying domain nodal curve $C$ of a morphism
 $\varphi:(C,{\cal O}_C^{Az};{\cal E})\rightarrow Y$ in our problem
 	  into the union of a major nodal subcurve $C_0$
	     and a minor subcurve $C_1$.
	     $C_0$ in general contributes to the imaginary part of the twisted central charge
		  of $\varphi$
	    and is where the $Z$-semistability of $\varphi$  should be imposed
		 in the standard way, as given in Definition~\ref{Zss}
		    now applied only to $\tilde{\cal E}_{\varphi}|_{C_0}$,		
	 while $C_1$ consists only of ${\Bbb P}^1$-bubbling trees,
	   whose contribution to the $Z$-semistability of $\varphi$
	   has to be imposed by hand (in an as-natural-as-possible way).

 %{\bf An auxiliary moduli stack.}
 The moduli stack
     $\FM_g^{1,[0];\scriptsizeZss}(Y;c)$
        of $Z^{B+\sqrt{-1}J,L}$-semi-stable Fourier-Mukai transforms
	    of dimension $1$,  width $[0]$, and central charge $c\in \hat{\Bbb H}_-$
        from stable curves of genus $g$  to $Y$
     in [L-Y3] D(10.1)	
  was constructed exactly to make such separation of morphisms from over nodal curves
   to morphisms from over main-versus-minor subcurve of nodal curves precise
  and, hence, ease the process toward a complete notion/definition of stability conditions
   for $\varphi$.
 The compactness of  $\FM_g^{1,[0];\scriptsizeZss}(Y;c)$,
    cf.\ Theorem~\ref{FM-compact}  above from [L-Y3: Theorem 3.1],
    justifies that this is a reasonable stack to begin with.
 What remains to be done
  constitutes the core of the current note D(10.2): Sec.~2 -- Sec.~4.

\bigskip

\section{Preliminaries aiming for a definition of stability conditions for morphisms
                    from general Azumaya nodal curves with a fundamental module}
\label{preliminary}	
To extend the notion of semistable morphisms
  from Azumaya nodal curves with a fundamental module
 beyond the domains reviewed in Sec.~\ref{review} in such a way
   that it is natural  and
   that it gives rise to a compact moduli stack,
there are two sets of technical issues one has to understand beforehand
 -- one on the behavior of cohomologies under push-pull  and
     the other on the behavior over bubbling ${\Bbb P}^1$-trees.
In this section, we explain what paves our path toward Definition~\ref{Zssm}
 of general semistable morphisms in Sec.~\ref{msssm}.
It is also through this section one can see better why Definition~\ref{Zssm},
  though conceivably not the only possibility, is very natural.

\bigskip

\subsection{Push-forward and higher direct image of sheaves}

We address in this subsection
  related cohomological issues toward the notion of semistable morphisms in our problem.

\bigskip

\begin{lemma}\label{cpf}
{\bf [criterion for pushing forward a flat family to another flat family].}
 Let
   $S$ be a separated Noetherian base scheme;
   $f_S:X^{\prime}_S \rightarrow X_S$ be an $S$-morphism of  separated Noetherian schemes
      of finite type over $S$;  and
	${\cal F}^{\prime}_S$ be a quasi-coherent sheaf on $X^{\prime}_S$ that is flat over $S$.
 Assume that $R^i\!{f_S}_{\ast}({\cal F}^{\prime}_S)=0$ for all $i>0$,
 then ${f_S}_{\ast}({\cal F}^{\prime}_S)$ is a quasi-coherent sheaf on $X_S$
     that is flat over $S$.
\end{lemma}

\bigskip

We remark that in our application, $f_S$ is projective and the same statement holds with
 `quasi-coherent'  replaced by `coherent'.

\bigskip

\begin{proof}
 As the statement is local, without loss of generality, one may assume that
  both $S$ and $X$ are affine with $S=\Spec A_0$, $X=\Spec A$, and $A$ is an $A_0$-algebra.
 In this case,
  ${f_S}_{\ast}({\cal F}^{\prime}_S)
      = H^0(X^{\prime}_S, {\cal F}^{\prime}_S)^{\sim}$,
   the quasi-coherent sheaf on $X_S$ associated to the $A$-module 	
   $H^0(X^{\prime}_S, {\cal F}^{\prime}_S)$.
 Similarly,
  $R^i\!{f_S}_{\ast}({\cal F}^{\prime}_S)
     = H^i(X^{\prime}_S, {\cal F}^{\prime}_S)^{\sim}$
   for   $i>0$.
 The statement becomes:
  \begin{itemize}
   \item[{\tiny $\bullet$}]
  {\it  Let ${\cal F}^{\prime}_S$ be a quasi-coherent sheaf on $X^{\prime}_S$
             that is flat over $S=\Spec A_0$  with
           $H^i(X^{\prime}_S,{\cal F}^{\prime}_S)=0$ for all $i>0$.
           Then, $H^0(X^{\prime}_S, {\cal F}^{\prime}_S)$	 is a flat $A_0$-module.	}
  \end{itemize}
  Which we will prove.

 Let ${\cal U}=\{U_{\alpha}\}_{\alpha}$
   be a finite affine open cover of $X^{\prime}_S$,
 then $\{H^i(X^{\prime}_S,{\cal F}^{\prime}_S)\}_{i\ge 0}$
   is the cohomology associated to the \v{C}ech complex
   $$
    (\, C^{\bullet}({\cal U},{\cal F}^{\prime}_S)\,,\, d\,)\; : \hspace{1em}
     C^0({\cal U},{\cal  F}^{\prime}_S)\; \stackrel{d^0}{\longrightarrow}\;
	 C^1({\cal U},{\cal  F}^{\prime}_S)\; \stackrel{d^1}{\longrightarrow}\;
	 C^2({\cal U},{\cal  F}^{\prime}_S)\; \stackrel{d^2}{\longrightarrow}\;  \cdots\,.
   $$
  Note that  since ${\cal F}^{\prime}_S$ is flat over $S$,
   $C^p({\cal U},{\cal F}^{\prime}_S)
	    := \prod_{i_0<\,\cdots\, <i_p}{\cal F}^{\prime}_S(U_{i_0\,\cdots\,i_p})$
    is a flat $A_0$-module, for all $p\ge 0$.		
  Consider now the following sequence of short exact sequences
  $${\tiny
   \begin{array}{ccccccccccccccccccc}
    0 & \rightarrow  & H^0(X^{\prime}_S,{\cal F}^{\prime}_S)
	   & \rightarrow  & C^0({\cal U},{\cal F}^{\prime}_S)
	   & \rightarrow  & \Image d^0    & \rightarrow  & 0                   \\[.6ex]
	&& \|       &&&&  \|                                                                                                 \\[.6ex]
	&& \Ker d^0             &&   0   & \rightarrow        & \Ker d^1  & \rightarrow
	   & C^1({\cal U},{\cal F}^{\prime}_S) & \rightarrow
	   & \Image d^1       & \rightarrow                      & 0                                         \\[.6ex]
    &&&&&&&&&&  \|	                                                                                             \\[.6ex]
	&&&&&&&&  0      & \rightarrow      &  \Ker d^2                  & \rightarrow
	   & C^2({\cal U},{\cal F}^{\prime}_S) & \rightarrow
	   & \Image d^2      & \rightarrow                               &  0                                \\[.6ex]
	&&&&&&&&&&&&&&  \|                                                                                      \\[.6ex]
	&&&&&&&&&&&&&&  \cdots\cdots  &&&.
   \end{array} } % end-tiny
  $$
  Here, $\Image d^i = \Ker d^{i+1}$, for all $i\ge 0$,
   since $H^i(X^{\prime}_S,{\cal F}^{\prime}_S)=0$, for all $i>0$, by assumption.
  Let $I_0$ be any ideal of $A_0$.  Then,
   $$
       \Tor_i^{A_0}(I_0, C^p({\cal U},{\cal F}^{\prime}_S))\; =\; 0\,,
   $$
   for all $i>0$ and $p\ge 0$,
   since $C^p({\cal U},{\cal F}^{\prime}_S)$ is a flat $A_0$-module.
  It follows that
  $$
   \begin{array}{cccccccccccccc}
     \Tor_1^{A_0}(I_0,H^0(X^{\prime}_S,{\cal F}^{\prime}_S))
	   & \simeq &    \Tor_2^{A_0}(I_0,  \Image d^{\,0})                               \\[.6ex]
    && \mbox{\scriptsize $\mid$}\wr                                                                          \\[.6ex]
	&&  \Tor_2^{A_0}(I_0,  \Ker d^1)
       & \simeq &   \Tor_3^{A_0}(I_0,  \Image d^1)                                          \\[.6ex]
    &&&& \mbox{\scriptsize $\mid$}\wr	                                                                 \\[.6ex]
    &&&& \cdots\cdots     & .	
   \end{array}	
  $$	
  Since
   $$
      \Tor_p^{A_0}(I_0, \Image d^{p-2})\;
        \simeq\;  \Tor_p^{A_0}(I_0, \Ker d^{p-1})\; \simeq\;  0
   $$
   for $p$ large enough,
   $$
      \Tor_1^{A_0}(I_0, H^0(X^{\prime}_S, {\cal F}^{\prime}_S))\; =\; 0
   $$
  for all ideal $I_0$ of $A_0$.
 This shows that $H^0(X^{\prime}_S, {\cal F}^{\prime}_S)$ is a flat  $A_0$-module
 and hence proves the lemma.
 
\end{proof}

\bigskip

\begin{lemma}\label{rpf}
 {\bf [relation on push-forward as a closed condition].}
 Let
   $(B, b_0)$ be a smooth pointed curve,
   $f_B: X^{\prime}_B\rightarrow X_B$
     be a birational projective morphism of Noetherian $B$-schemes
	 whose exceptional locus is of relative dimension $1$ under the restriction of $f_B$,
   ${\cal F}_B$ be a coherent ${\cal O}_{X_B}$-module
        that gives a flat family of pure $1$-dimensional coherent  sheaves on fibers of $X_B$ over $B$,
	 and
   $f_B^{\ast}({\cal F}_B)\rightarrow {\cal F}_B^{\prime}\rightarrow 0$
      be a quotient ${\cal O}_{X^{\prime}_B}$-module
	  that gives a flat family of $1$-dimensional coherent sheaf on fibers of $X^{\prime}_B$ over $B$.
  Suppose that
   \begin{itemize}
    \item[(1)]
      $R^1\!f_{B\ast}({\cal F}^{\prime}_B)=\; 0\; =
	    R^1\!f_{b_0\ast}({\cal F}^{\prime}_{b_0})$
	    \hspace{1em} and \hspace{1em}
      $(f_{B-\{b_0\}})_{\ast}({\cal F}^{\prime}_{B-\{b_0\}})\;
	       \simeq\; {\cal F}_{B-\{b_0\}}\,$;
		
    \item[(2)]
      $f_B$ induces an isomorphism
	  ${\cal F}_B \rightarrow
	    {f_B}_{\ast}({\cal F}^{\prime}_B)$
	  over an open dense complement in $\Supp({\cal F}_B)$ of a codimension-$2$ subset.
    \end{itemize}
 Then,
   $$
    {f_B}_{\ast}({\cal F}^{\prime}_B)\; \simeq\; {\cal F}_B
      \hspace{2em}\mbox{on $X_B$}
    $$
   with
   $$
    {f_{b_0}}_{\ast}({\cal F}^{\prime}_{b_0})\; \simeq\; {\cal F}_{b_0}
      \hspace{2em}\mbox{on $X_{b_0}$.}
   $$
 Here, $(\,\bullet\,)_{B-\{b_0\}}$ and $(\,\,\bullet\,\,)_{b_0}$	
   mean the restriction of $(\,\bullet\,)$ to over $B-\{b_0\}$ and $b_0$ respectively.
\end{lemma}

\begin{proof}
 The pre-composition with a natural homomorphism
   $$
      {\cal F}_B\; \longrightarrow\;
	  {f_B}_{\ast}f_B^{\ast}({\cal F}_B)\;\longrightarrow\;
	    {f_B}_{\ast}({\cal F}^{\prime}_B)	
   $$
  together with the assumption
    that ${\cal F}_B$ is flat over $B$ and
    that its restriction to over $B-\{b_0\}$ is an isomorphism
	   ${\cal F}_{B-\{b_0\}}\stackrel{\sim}{\rightarrow}
	       (f_{B-\{b_0\}})_{\ast}({\cal F}^{\prime}_{B-\{b_0\}})$
   gives an inclusion, which extends to an exact sequence of ${\cal O}_{X_B}$-modules
   $$
      0\;\longrightarrow\; {\cal F}_B\; \longrightarrow\;
	  {f_B}_{\ast}({\cal F}^{\prime}_B)\;\longrightarrow\;
	  {\cal Q}\;\longrightarrow\; 0\,,
   $$
   where ${\cal Q}$ is supported in an infinitesimal neighborhood of $X_{b_0}$ in $X_B$.
 This induces a long exact sequence
  $$
   \begin{array}{l}
     \cdots\; \longrightarrow\;
 	 \Tor_1^{X_B}
    ({\cal O}_{X_{b_0}},  {f_B}_{\ast}({\cal F}^{\prime}_B))\;
	 \longrightarrow\;    \Tor_1^{X_B}({\cal O}_{X_{b_0}},{\cal Q})
	    \mbox{$\hspace{4em}$} \\[1.6ex]
	 \hspace{16em}	
	 \longrightarrow\;  {\cal F}_{b_0}\;\longrightarrow\;
            \left.\left({f_B}_{\ast}({\cal F}^{\prime}_B)\right)	\right|_{b_0}\;
	        \longrightarrow\; {\cal Q}_{b_0}\; \longrightarrow\; 0\,.
   \end{array}
  $$
  This gives the exact sequence
  $$
    0\; \longrightarrow\;  {\cal Q}_{b_0}\;\longrightarrow\;
	  {\cal F}_{b_0}\;\longrightarrow\;
            \left.\left({f_B}_{\ast}({\cal F}^{\prime}_B)\right)	\right|_{b_0}\;
	        \longrightarrow\; {\cal Q}_{b_0}\; \longrightarrow\; 0\,,
  $$
 since
   	$\Tor_1^{X_B}
    ({\cal O}_{X_{b_0}},  {f_B}_{\ast}({\cal F}^{\prime}_B))=0$
	  from the flatness of ${f_B}_{\ast}({\cal F}^{\prime}_B)$ over $B$
	  and
	$\Tor_1^{X_B}({\cal O}_{X_{b_0}},{\cal Q})\simeq {\cal Q}_{b_0}$
	  via  the two-term line-bundle resolution
	  $0\rightarrow {\cal I}_{X_{b_0}}\rightarrow {\cal O}_{X_B}
	      \rightarrow {\cal O}_{X_{b_0}}\rightarrow 0$
	  of ${\cal O}_{X_{b_0}}$.
  Condition(2) on ${\cal F}_{b_0}$ implies that ${\cal Q}_{b_0}$
     is $0$-dimensional; it follows then from the purity of ${\cal F}_{b_0}$ that
     ${\cal Q}_{b_0}$ (and, hence, ${\cal Q}$) $=0$.
  This proves that
    $$
	  {\cal F}_B \stackrel{\sim}{\rightarrow}
	   {f_B}_{\ast}({\cal F}^{\prime}_B)
	$$	
    on $X_B$.	
  In particular,
   ${\cal F}_{b_0} \stackrel{\sim}{\rightarrow}
	   ({f_B}_{\ast}({\cal F}^{\prime}_B))|_{b_0}$.	
	
 It remains to show that the natural homomorphism
   $$
	 ({f_B}_{\ast}({\cal F}^{\prime}_B))|_{b_0}\;
	   \stackrel{\alpha}{\longrightarrow}\;
	    {f_{b_0}}_{\ast}({\cal F}^{\prime}_{b_0})	
   $$
	is an isomorphism.
 First note that, by the assumptions in the lemma,
  $\Ker(\alpha)$ is $0$-dimensional on $X_{b_0}$,
  which must then vanish
   since
      $({f_B}_{\ast}({\cal F}^{\prime}_B))|_{b_0}
                                    \simeq	{\cal F}_{b_0}$
    is pure.
 Consider now the exact sequence of coherent ${\cal O}_{X^{\prime}_B}$-modules
  $$
     0\; \longrightarrow\; {\cal H}\; \longrightarrow {\cal F}^{\prime}_B\;
	       \longrightarrow\;   {\cal F}^{\prime}_{b_0} \; \longrightarrow\; 0\,,
  $$
 which gives the long exact sequence of coherent ${\cal O}_{X_B}$-modules
  $$
   0\;\longrightarrow\; {f_B}_{\ast}({\cal H})\; \longrightarrow\;
     {f_B}_{\ast}({\cal F}^{\prime}_B)\;\longrightarrow\;
	 {f_B}_{\ast}({\cal F}^{\prime}_{b_0})\; \longrightarrow\;
	  R^1\!{f_B}_{\ast}({\cal H})\;\longrightarrow\;
	  R^1\!{f_B}_{\ast}({\cal F}^{\prime}_B)\;\longrightarrow\; \cdots\,.
  $$
  Assumptions (1) and (2)  and
  the fact that ${\cal F}^{\prime}_{b_0}$ is supported only on $X^{\prime}_{b_0}$
  imply further that
   the above long exact sequence can be truncated
   to give a short exact sequence of coherent ${\cal O}_{X_{b_0}}$-modules
  $$
    0\;\longrightarrow\;
    ({f_B}_{\ast}({\cal F}^{\prime}_B))|_{b_0}\;
	  \stackrel{\alpha}{\longrightarrow}\;
	 {f_{b_0}}_{\ast}({\cal F}^{\prime}_{b_0}) \; \longrightarrow\;
	 R^1\!{f_B}_{\ast}({\cal H})\;\longrightarrow\; 0
  $$
  with $R^1\!{f_B}_{\ast}({\cal H})$
   a $0$-dimensional coherent sheaf on $X_{b_0}$.
	
On the other hand,
    $$
	 {\cal  H}\; \simeq\;
	 {\cal F}^{\prime}_B
	        \otimes_{{\cal O}_{X^{\prime}_B}}
			  {\cal O}_{X^{\prime}_B}(X_{b_0})
	$$
   and, hence, is isomorphic to ${\cal F}^{\prime}_B$ over an affine neighborhood
   of $b_0\in B$.
  Since  $R^1\!{f_B}_{\ast}({\cal H})$ is supported only over $b_0\in B$,
   $$
     R^1\!{f_B}_{\ast}({\cal H})\; \simeq\;
	  R^1\!{f_B}_{\ast}({\cal F}^{\prime}_B )\; \simeq\;  0\,,
   $$
   by Assumption (1).
  It follows that  $\alpha$ is an isomorphism.
  This concludes the proof.
  
\end{proof}

\bigskip	
	
As we only need to deal with coherent sheaves of relative dimension $\le 1$,
we now turn to some criteria naturally occurring in our setting
 that force the  first (and hence all higher) direct image of sheaves to vanish.
A basic example is given in the lemma below
 whose repeating use, together with base changes, is enough to deal with
  the case of ${\Bbb P}^1$-bubbling from resolving $A_n$-singularities
  on a complex surface from the total space of a complex $1$-parameter family of nodal curves.
However, what we will actually use for this note is the more powerful/encompassing
  Lemma~\ref{vR1sm} in Sec.~3.2.
          % Lemma  [vanishing of $R^1\!(\rho\times \Id_Y)_{\ast}(\,\mbox{graph}\,)$
          %                 for semistable morphism].}

\bigskip

\begin{lemma}\label{cvfdis}
 {\bf [criterion for vanishing first direct-image sheaf].}
  Let
   $f:X^{\prime}\rightarrow X$ be a projective morphism of Noetherian schemes
     that fits into the following commutative diagram
    $$
     \xymatrix{
        X^{\prime}\; \ar[rrd]_f   \ar@{^{(}->}[rr]
		&&   \;{\Bbb P}^1_X \ar[d]^-{\prscriptsize_X} \\
        && X	
      }
    $$
    and
  ${\cal F}^{\prime}$ be a coherent sheaf on $X^{\prime}$
    that fits into an exact sequence of the form
	$$
	 {\cal O}_{X^{\prime}}^{\; \oplus k}\;
	   \longrightarrow\; {\cal F}^{\prime}\;  \longrightarrow\; 0\,,
	$$
	for some $k\in {\Bbb Z}_{>0}$.
 Then,
  $$
     R^1\!f_{\ast}({\cal F}^{\prime})\;=\; 0\,.
  $$
\end{lemma}

\begin{proof}
 Since one has the quotient homomorphism
    ${\cal O}_{{\Bbb P}^1_X}\rightarrow {\cal O}_{X^{\prime}}$
	via the embedding $X^{\prime}\hookrightarrow {\Bbb P}^1_X$,
  one can promote the quotient ${\cal O}_{X^{\prime}}$-module
     ${\cal O}_{X^{\prime}}^{\;\oplus k}\rightarrow {\cal F}^{\prime}$
   to a short exact sequence of coherent ${\cal O}_{{\Bbb P}^1_X}$-modules:	
     $$
	   0\; \longrightarrow\; {\cal H}\; \longrightarrow\;
	    {\cal  O}_{{\Bbb P}^1_X}^{\; \oplus k}\; \longrightarrow\;
	      {\cal F}^{\prime}\;\longrightarrow\; 0\,.
     $$
  This gives a long exact sequence of coherent ${\cal O}_X$-modules
    $$
	  \cdots\; \longrightarrow\;
	   R^1\!\pr_{X\ast}({\cal O}_{{\Bbb P}^1_X}^{\;\oplus k})\; \longrightarrow\;
	   R^1\!\pr_{X\ast}({\cal F}^{\prime})\; \longrightarrow\;
	   R^2\!\pr_{X\ast}({\cal H})\; \longrightarrow \; \cdots\,.	
	$$
  Since
      both $R^1\!\pr_{X\ast}({\cal O}_{{\Bbb P}^1_X}^{\;\oplus k})$ and
	   $R^2\!\pr_{X\ast}({\cal H})$ vanish
	   from the known cohomology of projective spaces
	    and the fact that $\pr_X$ has only relative dimension $1$,
   the lemma follows. 	

\end{proof}

\bigskip

\begin{lemma}\label{vfdis}
{\bf [vanishing first direct-image sheaf under composition]. }\\
 Let
   $f=h_l\circ\; \cdots\;\circ h_1$ be a composition of morphisms of Noetherian schemes
	  $$
	    X_l\;\stackrel{h_l}{\longrightarrow}\;
		X_{l-1}\;\stackrel{h_{l-1}}{\longrightarrow}\; \cdots\;
     		 \stackrel{h_2}{\longrightarrow}\;
		 X_1\;\stackrel{h_1}{\longrightarrow}\; X_0
	  $$
	and
	${\cal F}$ be a coherent sheaf on $X_l$.
  Then,
    $R^1\!f_{\ast}({\cal F})=0$
	 if and only if
	$R^1\!h_{i\ast}( (h_l\circ\,\cdots\,\circ h_{i+1})_{\ast}({\cal F})     )=0$
      for all $i=1,\,\ldots\,,l$.
\end{lemma}

\begin{proof}
  This is an immediate consequence of the spectral sequence associated to a composition
   $f= g\circ h$ of morphisms of Noetherien schemes, which gives the exact sequence
   $$
      0\; \longrightarrow\;
	   R^1\!g_{\ast}(h_{\ast}({\cal F}))\;\longrightarrow\;
	   R^1\!f_{\ast}({\cal F})\;\longrightarrow\;
	   g_{\ast}( R^1\!h_{\ast}({\cal F})  )\; \longrightarrow\; 0\,.
   $$
\end{proof}

%%%%%%%%%%%%%%%%%%%%
% \bigskip
%
% A situation we will encounter repeatedly is illuminated in the following example.
%
% \bigskip
%
% \begin{example}{\bf [blow-up of $A_n$-singularities on a surface].} {\rm
%   ????????????.
% }\end{example}
%
% \bigskip
%
% \noindent $\bullet$
%   ???????????????.
%
%%%%%%%%%%%%%%%%%%%%%

\bigskip

\subsection{A special class of  morphisms from Azumaya curves with a fundamental module
                          and positivity of fundamental modules on ${\Bbb P}^1$-trees}
						
Anticipating the issue of bubbling ${\Bbb P}^1$-trees from removing irregularities of a morphism
 to haunt us, as in the case of Gromov-Witten theory,
we study in this subsection related algebraic geometry for a special class of morphisms
 from Azumaya ${\Bbb P}^1$-trees with a fundamental module.

\bigskip						

\begin{flushleft}
{\bf A special class of  morphisms from Azumaya curves with a fundamental module}
\end{flushleft}
Consider a morphism
  $$
     \varphi\;  :\;
     (C, {\cal O}_C^{Az}:=\Endsheaf_{{\cal O}_X}({\cal E});{\cal E} )\;
    \longrightarrow\;  Y
  $$
   from an Azumaya nodal curve with a fundamental module of rank $r$
   to a projective variety $(Y, {\cal O}_Y(1))$.
Let
  $\tilde{\cal E}_{\varphi}$, the graph of $\varphi$,
  be the associated coherent ${\cal O}_{C\times Y}$-module on $C\times Y$
  that is flat over $C$  and of relative dimension $1$ (and relative length $r$).
Suppose $\varphi$ is a morphism such that  the following three special properties hold
 for $\tilde{\cal E}_{\varphi}\,$:
 \begin{itemize}
   \item[(1)]
     $\tilde{\cal E}_{\varphi}$  is realizable as a quotient of a trivial bundle on $C\times Y$
     $$
 	 {\cal O}_{C\times Y}^{\;\oplus k}\;
 	    \stackrel{\alpha}{\longrightarrow}\;  \tilde{\cal E}_{\varphi}\,;
     $$
	
   \item[(2)]	
     its push-forward ${\pr_Y}_{\ast}(\tilde{\cal E}_{\varphi})$
      is $0$-dimensional,
     where $\pr_Y:C\times Y\rightarrow Y$ is the projective map;
   
  \item[(3)]
  on each irreducible component of $C$,
   $\tilde{\cal E}_{\varphi}$ is not the pull-back $\pr_Y^{\ast}({\cal F})$
   of some $0$-dimensional coherent ${\cal O}_Y$-module;
   i.e., $\varphi$ is not constant on each irreducible component of $C$.
 \end{itemize}	
 %%%%%%%%%%%%%%%%%
 % %
 % Cf. {\sc Figure}~???.
 % %
 % \marginpar{\raggedright\tiny $\bullet$ {\sc Figure}}
 % %	
 %%%%%%%%%%%%%%%%%
The decoration/marking $\alpha$ in Property (1) specifies a morphism
 $$
    f_{(\varphi, \alpha)}\; :\;  C\;
    	\longrightarrow\;  \Quot_Y({\cal O}_Y^{\;\oplus k}, r)
 $$
 from $C$ to the Quot-scheme $\Quot_Y({\cal O}_Y^{\;\oplus k}, r)$
 of $0$-dimensional quotients of ${\cal O}_Y^{\,\oplus k}$ of length $r$.
To see the implication of Property (2) and Property (3), let us recall
  how  $\Quot_Y({\cal O}_Y^{\,\oplus k}, r)$ is embedded in a Grassmannian variety,
  which in turn embeds into a projective space, (e.g., [F-G-I-K-N-V] or [H-L] ):
\begin{itemize}
 \item[{\Large $\cdot$}]
   Let ${\cal O}_{(Quot_Y({\cal O}_Y^{\,\oplus k},r)\times Y)
                                                             /Quot_Y({\cal O}_Y^{\,\oplus k},r)}(1)$
     be the relative ample line bundle on\\
	 $(\Quot_Y({\cal O}_Y^{\,\oplus k},r)\times Y)
	                             / \Quot_Y({\cal O}_Y^{\,\oplus k},r)$
	 associated to the ample line bundle ${\cal O}_Y(1)$ on $Y$,
     $$
	    0\;  \longrightarrow\;     \tilde{\cal H}\; \longrightarrow\;
      {\cal O}_{Quot_Y({\cal O}_Y^{\,\oplus k}, r)\times Y}^{\;\oplus k}\;
    	\longrightarrow \;     \tilde{\cal Q}\;    \longrightarrow\; 0
	$$
       be the short sequence of
       ${\cal O}_{Quot_Y({\cal O}_Y^{\,\oplus k},r)\times Y}$-modules
	   that extends the universal quotient sheaf
      on $\Quot_Y({\cal O}_Y^{\,\oplus k}, r) \times Y$,    and
	$$
	   \pr_1\, :\,  \Quot_Y({\cal O}_Y^{\,\oplus k},r) \times Y\,
	       \longrightarrow\,  \Quot_Y({\cal O}_Y^{\,\oplus k},r)\;,    \hspace{2em}
	   \pr_2\, :\,  \Quot_Y({\cal O}_Y^{\,\oplus k},r) \times Y\,
	       \longrightarrow\,  Y				
	$$
      be the projection maps.
  Then, there exists an $m>0$ such that
   \begin{itemize}
     \item[(a)]
      for all $m^{\prime}\ge m$,
      $$
        R^i\!{\pr_1}_{\ast}(\tilde{\cal H}(m^{\prime}))\;,\hspace{2em}
	    R^i\!{\pr_1}_{\ast}(
	       {\cal O}_{Quot_Y({\cal O}_Y^{\,\oplus k},r)\times Y}(m^{\prime}))\;,
	     \hspace{2em}
	    R^i\!{\pr_1}_{\ast}(\tilde{\cal Q}(m^{\prime}))	
      $$	
	  vanish, for all $i>0$,
	
	 \item[(b)]
	  the push-forward
       $$
        0\; \longrightarrow\; 		
		{\pr_1}_{\ast}(\tilde{\cal H}(m))\; \longrightarrow\;
	    {\pr_1}_{\ast}(
	       {\cal O}_{Quot_Y({\cal O}_Y^{\,\oplus k},r)\times Y}(m))\;
		   \longrightarrow\;
	       {\pr_1}_{\ast}(\tilde{\cal Q}(m))\; \longrightarrow\; 0	
      $$	 	
	  is an exact sequence of locally free sheaves
	  on $\Quot_Y({\cal O }_Y^{\,\oplus k},r)$,
	
	 \item[(c)]
	 for all $m^{\prime}\ge m$,  the sequence
	  $$
	   \begin{array}{l}
		{\pr_1}_{\ast}(\tilde{\cal H}(m))
		     \otimes
	       {\pr_1}_{\ast}(
			   {\cal O}_{Quot_Y({\cal O}_Y^{\,\oplus k},r)\times Y}(m^{\prime}-m))
			                                                  \\[1.2ex]
															
		 \hspace{6em}\longrightarrow\;
	    {\pr_1}_{\ast}(
	       {\cal O}_{Quot_Y({\cal O}_Y^{\,\oplus k},r)\times Y}(m^{\prime}))\;
		   \longrightarrow\;
	       {\pr_1}_{\ast}(\tilde{\cal Q}(m^{\prime}))\; \longrightarrow\; 0	    	
	   \end{array}	
	  $$
	  where the first homomorphism is given by multiplication of global sections,
 	  is exact.	
	\end{itemize}
  
 \item[{\Large $\cdot$}]
  The quotient sequent sequence
   $$
	\begin{array}{ccccc}
  {\pr_1}_{\ast}(
	      {\cal O}_{Quot_Y({\cal O}_Y^{\,\oplus k},r)\times Y}(m))
		&  \longrightarrow
		&  {\pr_1}_{\ast}(\tilde{\cal Q}(m))     &  \longrightarrow     &  0	  \\[1.2ex]
	  \|                                                                                                                                                 \\[.6ex]
     {\cal O}_{Quot_Y({\cal O}_Y^{\,\oplus k},r)}
		  \otimes  H^0(Y, {\cal O}_Y(m)) 	  		
    \end{array}		
   $$	 	
   specifies then an embedding
   $$
      \iota\; :\;  \Quot_Y({\cal O}_Y^{\,\oplus k},r)\;
	    \hookrightarrow \Gr_{\Bbb C}(H^0(Y,{\cal O}_Y(m)),r)\,.
   $$
  By construction,
    ${\cal O}_{Quot_Y({\cal O}_Y^{\,\oplus k},r)}
		                                                            \otimes  H^0(Y, {\cal O}_Y(m)) 	  	
	    \rightarrow    {\pr_1}_{\ast}(\tilde{\cal Q}(m))
        \rightarrow     0$
	on $\Quot_Y({\cal O}_Y^{\,\oplus k},r)$	
	 is the pull-back of  the universal quotient sequence
    $$
	  {\cal O}_{Gr_{\Bbb C}(H^0(Y,{\cal O}_Y(m)))}
	      \otimes H^0(Y,{\cal O}_Y(m))\; \longrightarrow\;  \hat{\cal Q}\;
		  \longrightarrow\; 0
	$$
	on $\Gr_{\Bbb C}(H^0(Y,{\cal O}_Y(m)),r)$ by $\iota$.
   
 \item[{\Large -}]
  Note that $\hat{\cal Q}$ is a locally free sheaf on
   $\Gr_{\Bbb C}(H^0(Y,{\cal O}_Y(m)),r)$ of rank $r$;
  the quotient homomorphism
   $$
    {\cal O}_{Gr_{\Bbb C}(H^0(Y,{\cal O}_Y(m)),r)}
            \otimes \bigwedge^rH^0(Y,{\cal O}_Y(m))\;
	   \longrightarrow\;
 	   \bigwedge^r\hat{\cal Q}\;\; (\,=: \determinant \hat{\cal Q}\,)\;
	   \longrightarrow\;  0
   $$
   defines the Pl\"{u}cker embedding of $\Gr_{\Bbb C}(H^0(Y,{\cal O}_Y(m)),r)$
   into the projective space
     ${\Bbb P}^{\,(\!\!\!\mbox{\tiny
                               $\begin{array}{c}N \\[-.6ex] r \end{array}$}\!\!\!)}$,
    where $N=\dimm H^0(Y,{\cal O}_Y(m))$.	
  In particular, $\determinant \hat{Q}$	is a very ample line bundle on
   $\Gr_{\Bbb C}(H^0(Y,{\cal O}_Y(m)),r)$.
\end{itemize}
  
Back to the study of $\varphi$ or, equivalently, $\tilde{\cal E}_{\varphi}$.
Property (1) of $\tilde{\cal E}_{\varphi}$ initiates the above construction.
From the above review,
 if letting ${\cal O}_{(C\times Y)/C}(1)$
  be the relative ample line bundle on $(C\times Y)/C$ associated to the ample line bundle
  ${\cal O}_Y(1)$ on $Y$  ,
 then
  $$
   (\iota\circ f_{(\varphi,\alpha)})^{\ast}(\hat{\cal Q})\; \simeq\;
	f_{(\varphi,\alpha)}^{\ast}({\pr_1}_{\ast}(\tilde{\cal Q}(m)))\; 	
	\simeq\; {\pr_C}_{\ast}(\tilde{\cal E}_{\varphi}(m))\,.	
  $$
 Here,  $\pr_C:C\times Y\rightarrow C$ is the projection map.
In general,
  ${\pr_C}_{\ast}(\tilde{\cal E}_{\varphi}(m))$
   is not isomorphic to ${\pr_C}_{\ast}(\tilde{\cal E}_{\varphi})={\cal E}$.
However,
 when Property (2) holds for $\tilde{\cal E}_{\varphi}$,
 $\tilde{\cal E}_{\varphi}(m^{\prime})\simeq \tilde{\cal E}$ for all $m^{\prime}$.
Thus, in this case,
 $$
   (\iota\circ f_{(\varphi,\alpha)})^{\ast}(\hat{\cal Q})\; \simeq\;
	f_{(\varphi,\alpha)}^{\ast}({\pr_1}_{\ast}(\tilde{\cal Q}(m)))\; 	
	\simeq\; {\cal E}	
 $$
 and the universal quotient bundle on $\Gr_{\Bbb C}(H^0(Y,{\cal O}_Y(m)),r)$
 is pulled back to a quotient sequence of ${\cal O}_C$-modules
 $$
   {\cal O}_C\otimes H^0(Y,{\cal O}_Y(m))\; \longrightarrow\;
    {\cal E}\; \longrightarrow\; 0\,.
 $$
If, in addition, Property (3) also holds, then
$f_{(\varphi,\alpha)}$ is not constant on each irreducible component of $C$.
Since now
  $\determinant{\cal E}
      \simeq (\iota\circ f_{(\varphi,\alpha)})^{\ast}(\determinant \hat{\cal Q})$
   and $\determinant\hat{Q}$ is very ample,
 ${\cal E}$ has positive degree on each irreducible component of $C$.
In summary:
 
\bigskip

\begin{lemma}\label{ggEpdE}
  {\bf [global generation of ${\cal E}$ and positivity of $\degree({\cal E})$].}
 Let $\varphi:(C,{\cal O}_C^{Az};{\cal E})\rightarrow Y$ be a morphism as above,
   whose graph $\tilde{\cal E}_{\varphi}\in \CohCategory(C\times Y)$
   satisfies Properties (1), (2), and (3).
 Then
   ${\cal E}$ is globally generated and has positive degree on each irreducible component of $C$.
\end{lemma}

\bigskip

\begin{flushleft}
{\bf When $C$ is a  ${\Bbb P}^1$-tree: positivity of fundamental modules}
\end{flushleft}
\begin{definition}\label{PtPc}
  {\bf [${\Bbb P}^1$-tree and ${\Bbb P}^1$-chain].}  {\rm
  A compact (not necessarily connected) nodal curve $C$ that is simply-connected
    is called a {\it ${\Bbb P}^1$-tree}
  since in this case all its irreducible components are ${\Bbb P}^1$'s
  and the dual graph to $C$ is a tree
  (i.e.\ a simply-connected graph with possibly more-than-one connected components ).
  A ${\Bbb P}^1$-tree that is connected
        and such that each ${\Bbb P}^1$-component intersects with
	    at most two other ${\Bbb P}^1$-components
	is called  a {\it ${\Bbb P}^1$-chain}.
  The number of ${\Bbb P}^1$-components in the chain is called the {\it length} of the chain.
  For a ${\Bbb P}^1$-chain of length $>1$, there are exactly two ${\Bbb P}^1$-components,
   each of which intersects with only one other ${\Bbb P}^1$-component.
  Each is called a {\it ${\Bbb P}^1$ at an end} of the chain.
  Cf.\ {\sc Figure}~2-2-1.
 \begin{figure} [htbp]
  \bigskip
  \centering
  \includegraphics[width=0.80\textwidth]{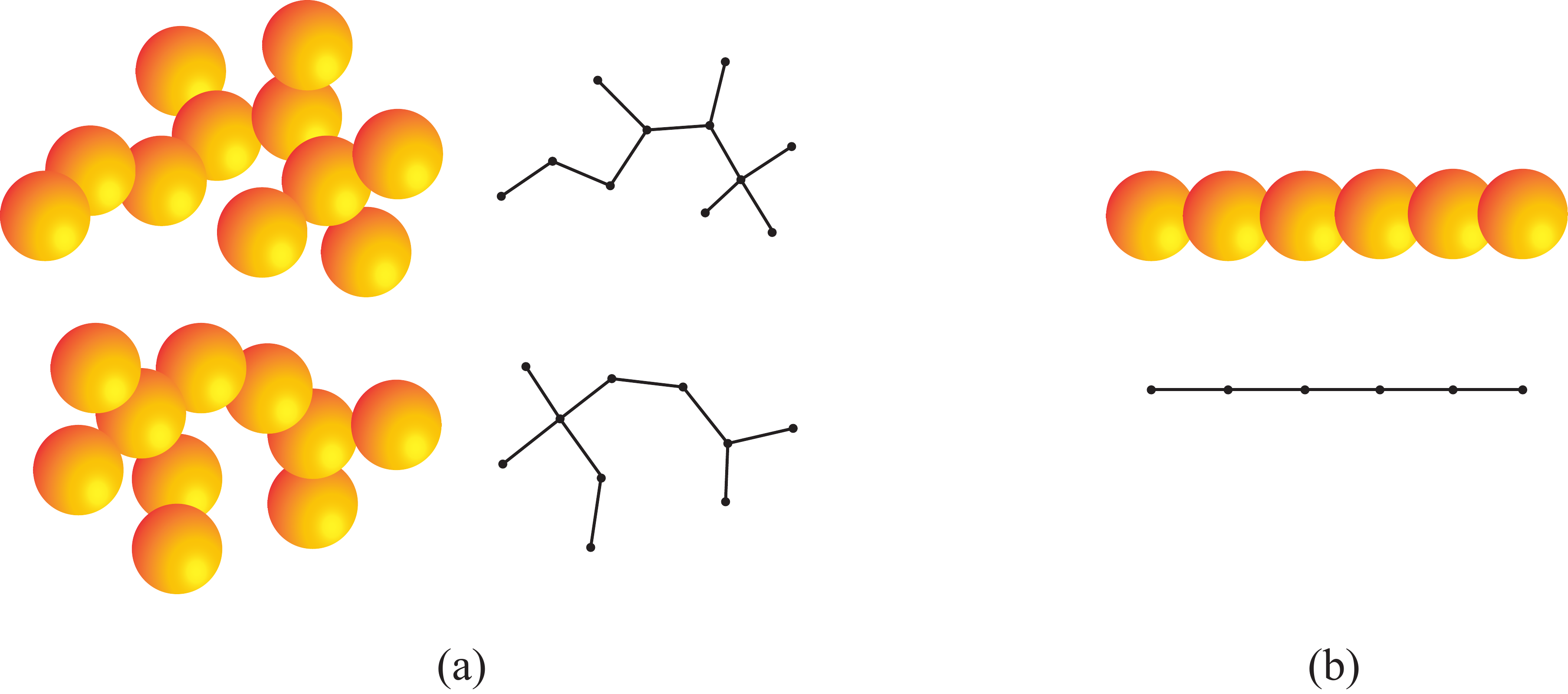}
 
  \vspace{2em}
  \centerline{\parbox{13cm}{\small\baselineskip 12pt
  {\sc Figure}~2-2-1.
    Two examples of ${\Bbb P }^1$-trees and their dual graph are illustrated.
	  In (a), the ${\Bbb P}^1$-tree has two connected components.
	  In (b), a ${\Bbb P }^1$-chain of length $6$ is illustrated.
	Note that bubbling ${\Bbb P}^1$-trees that arise from resolving/absorving irregularities
    	of morphisms from a filling (cf.\ Sec.~4.3) can be complicated. 	
      }}
  \bigskip
\end{figure}	
}\end{definition}

\bigskip

We now consider the case when $C$ is a ${\Bbb P}^1$-tree.
As in the Gromov-Witten theory for world-sheet instantons created by fundamental strings,
 such ${\Bbb P}^1$-trees occur from bubbling off of domains to keep the regularity
 of morphisms in our problem.

Recall the structure theorem of Grothendieck that
 a locally free sheaf of rank $r$ on ${\Bbb P}^1$ must be of the form
 $\oplus_{i=1}^r{\cal O}_{{\Bbb P}^1}(a_i)$.
 
\bigskip

\begin{definition}\label{nnpsp}
 {\bf [nonnegative/positive/strictly-positive locally-free sheaf on ${\Bbb P}^1$-tree].} {\rm
 A locally free sheaf ${\cal E}$ of rank $r$ on a ${\Bbb P }^1$-tree $C$
  is called
   \begin{itemize}
    \item[{\Large $\cdot$}]
    {\it nonnegative}
      if for each ${\Bbb P}^1$-component of $C$,
        ${\cal E}|_{{\Bbb P}^1}\simeq
          \oplus_{i=1}^r{\cal O}_{{\Bbb P}^1}(a_i)$
       for some non-negative integers $0\le a_1\le\,\cdots\,\le a_r $;
	
    \item[{\Large $\cdot$}]	
  {\it positive}	
	 if it is nonnegative  and
	 each connected component of $C$ has at least one ${\Bbb P}^1$-component such that
         ${\cal E}|_{{\Bbb P}^1}\simeq
               \oplus_{i=1}^r{\cal O}_{{\Bbb P}^1}(a_i)$
      for some non-negative integers $0\le a_1\le\,\cdots\,\le a_r $ with $a_r>0$;
	
    \item[{\Large $\cdot$}]	
  {\it strictly positive}
    if for each ${\Bbb P}^1$-component of $C$,
      ${\cal E}|_{{\Bbb P}^1}\simeq
         \oplus_{i=1}^r{\cal O}_{{\Bbb P}^1}(a_i)$
       for some non-negative integers $0\le a_1\le\,\cdots\,\le a_r $ with $a_r>0$.	
 \end{itemize}	
 By definition,
 $\;\; \mbox{\it strictly positive}\;\Rightarrow\; \mbox{\it positive}\; \Rightarrow\;
         \mbox{\it nonnegative}\,$.
}\end{definition}
 
\bigskip

\begin{lemma}\label{csp}
 {\bf [criterion of strict positivity].}
 Let ${\cal E}$ be a locally free sheaf on a ${\Bbb P}^1$-tree $C$.
 Then ${\cal E}$ is strictly positive if and only if
 ${\cal E}$ is globally generated and of positive degree on each irreducible component.
\end{lemma}

\begin{proof}
 That global generatedness and positive degree together imply strict positivity
  follows immediately from the structure theorem of Grothendieck.
 For the converse, we postpone till the end of the next theme
  after the space $H^0(C,{\cal E})$ of global sections of ${\cal E}$ is studied.
  
\end{proof}

\bigskip

%%%%%%%%%%%%%%%%%%%%%%%%%%%
% \begin{lemma} {\bf [criterion of positivity].}
%  Let ${\cal E}$ be a locally free sheaf on a ${\Bbb P}^1$-tree $C$
%    that is realizable in an exact sequence
%   $$
%    {\cal O}_C^{\,\oplus k}\; \longrightarrow\; {\cal E}\; \longrightarrow\;
%    {\cal G}\;  \longrightarrow\; 0\,,
%   $$
%   where $k>0$ and ${\cal G}$ is $0$-dimensional,
%   non-zero on each irreducible component of $C$.
%  Then ${\cal E}$ is positive.
% \end{lemma}
%
% \bigskip
%
%%%%%%%%%%%%%%%%%%%%

Continuing the discussion from the previous theme.
Thus, as a consequence of Lemma~\ref{ggEpdE} and Lemma~\ref{csp},
       % Lemma [global generation of ${\cal E}$ and positivity of $\degree({\cal E})$]
	   % Lemma [criterion of strict positivity]
  one has the following proposition:
	
\bigskip

\begin{proposition}\label{ncPEasm}
{\bf [necessary condition for $({\Bbb P}^1\mbox{-tree}, {\cal E})$
           to admit special morphism].}	
   Let $C$ be a ${\Bbb P}^1$-tree   and
    $(C,{\cal O}_C^{Az};{\cal E})$ be an Azumaya ${\Bbb P}^1$-tree
        with a fundamental module
     that admits a special morphism $\varphi$ to $Y$
	   whose graph $\tilde{\cal E}_{\varphi}$ satisfies Properties (1), (2), and (3).
  Then ${\cal E}$ must be strictly positive on $C$.
\end{proposition}
  
\bigskip

\noindent
This gives us a guide to a part of the stability condition in Sec.~3.1  and
 is the main reason why Definition~\ref{Zssm} there gives rise to  a bounded moduli stack,
 taking into account the preliminary compactness result in Part I [L-Y3]  (D(10.1)) of the work.

\vspace{3em}
\bigskip

\begin{flushleft}
{\bf $H^0(C, {\cal E})$
          and sections of ${\cal E}$ that vanish at specified points on $C$}
\end{flushleft}
Let
  $C$ be a ${\Bbb P}^1$-tree  and
  ${\cal E}$ be a nonnegative locally free sheaf of rank $r$ on $C$.
  
\bigskip

\begin{lemma}\label{h0E}
 {\bf [$h^0({\cal E})$].}
 Let  $\{C_{ij}\simeq {\Bbb P}^1\,:\, i,j\}$ be the set of irreducible components of $C$,
    where
	  $i$ labels the connected components of $C$   and
      $j$ labels the irreducible components in the $i$-th connected component of $C$.
  Suppose that 		
    ${\cal E}|_{C_{ij}}
	      \simeq  \oplus_{k=1}^r{\cal O}_{{\Bbb P}^1}(a_{ijk})$,
	$0\le a_{ij1}\le\, \cdots\, \le a_{ijr}$.	
	Then
   $$
     h^0({\cal E})\;
  	 :=\;   \dimm H^0(C, {\cal E})\;
	  =\;   r\cdot  |\pi_0(C)|\;   +\;   \sum_{i,j,k}a_{ijk}\,,	
   $$
   where $|\pi_0(C)|$ is the number of connected components of $C$.
  In particular,  except the number of connected components of $C$,
  it is independent of how this collection of nonnegative locally free sheaves on ${\Bbb P}^1$'s
   are  glued to give a nonnegative locally free sheaf ${\cal E}$ on a ${\Bbb P}^1$-tree $C\;$
  (i.e., independent of the isomorphism class of the pair $(C,{\cal E})$ from gluing
            except $|\pi_0(C)|$).
\end{lemma}

\begin{proof}
 Recall that $h^0({\cal O}_{{\Bbb P}^1}(a))=1+a$ for $a\ge 0$.
 As a consequence of the structure theorem of Grothendieck,
  any element in a fiber of a nonnegative locally free sheaf
   ${\cal F}\simeq \oplus_{k=1}^r{\cal O}(a_k)$   on ${\Bbb P}^1$
  extends to a global section of ${\cal F}$.
 The dimension of the space of such extensions is given by the sum $\sum_{k=1}^ra_k$.
    
 The fact that a tree is simply connected implies that there is no constraint on extending sections
   to over a new ${\Bbb P}^1$-component
  when building a nonnegative locally free sheaf on $C$
   by gluing a nonnegative locally free sheaf on ${\Bbb P}^1$ one at a time
   through an isomorphism between paired fibers.
  For each connected component,  the extension-of-global-sections is indifferent
    to the isomorphism class of the whole sheaf from gluing as well.
  It follows that, for the $i$-connected component of $C$,
   the dimension $h^0$ of the space of global sections is given by $r+\sum_{j,k}a_{ijk}$.
 The lemma follows.

\end{proof}
 
\bigskip

\begin{remark}\label{cstcpt}
 {$[$constrained sections and torsions under collapsing ${\Bbb P}^1$-tree$]$.} {\rm
 Let $A\subset C$ be a finite set of $a$-many distinct points on the ${\Bbb P}^1$-tree $C$.
 Then, the linear space
  $$
    \{s\in H^0(C,{\cal E})\,|\,
	     \mbox{$s$ vanishes at all the points in $A\subset C$ }\}
  $$
  has dimension $\ge h^0({\cal E})-ra= r(|\pi_0(C)|  - a)+\sum_{i,j,k}a_{ijk}$.
 When $A$ is the set of attached points of a ${\Bbb P}^1$-tree to the rest of a nodal curve,
 such sections generate torsions when pushing forward a locally free sheaf on the nodal curve
  by a collapsing morphism that contracts the ${\Bbb P}^1$-tree.
}\end{remark}	

\bigskip

\noindent
{\it Finishing the proof of Lemma~\ref{csp}.}
                                         % Lemma [criterion of strict positivity]
  Let
    $C$ be a ${\Bbb P}^1$-tree  and
	${\cal E}$ be a strictly positive locally-free sheaf on $C$.
   It's clear from the definition of strict positivity of ${\cal E}$ that $\determinant{\cal E}$
   has positive degree on each irreducible component of $C$.
  Furthermore, as in the proof of Lemma~\ref{h0E},
                                                   % Lemma [$h^0({\cal E})$]
   any element in a fiber of ${\cal E}$ extends to a global section of ${\cal E}$.
  It follows that the natural homomorphism of ${\cal O}_C$-modules
    $$
      {\cal O}_C\otimes H^0(C,{\cal E})\;\longrightarrow\; {\cal E}
    $$	
	is surjective.
  This shows the other direction of the lemma. 	

\noindent\hspace{40.76em}$\square$

\bigskip

\subsection{Remarks on nonnegative torsion-free sheaves on a ${\Bbb P}^1$-tree.}

As our main interest is semistable morphisms from Azumaya  nodal curves
  with a fundamental module, the objects in Sec.~2.2 are well anticipated.
However, somewhere along the path toward our goal
   we need also to understand the basics of nonnegative torsion-free coherent sheaves
   on a ${\Bbb P}^1$-tree.
We devote this subsection to this, with details that are similar to  Sec.~2.2 omitted.

\bigskip

\begin{convention}\label{csrr}
{$[$coherent sheaf of rank $r$$]$.} {\rm
By convention, a {\it (coherent) sheaf of rank $r$} on a curve $C$ means
 a coherent ${\cal O}_C$-module that has rank $r$ on every irreducible component of $C$.
}\end{convention}

\bigskip

\begin{flushleft}
{\bf Discrepancy to flatness and positive sheaves on a ${\Bbb P}^1$-tree}
\end{flushleft}
\begin{definition}\label{dtf}
 {\bf [discrepancy to flatness]. } {\rm
 Let ${\cal F}$ be a torsion-free sheaf of rank $r$ on a nodal curve $C$.
 For $p\in C$, define the {\it discrepancy to flatness of ${\cal F}$ at $p$} to be
  $$
   \delta_{\flatscriptsize}({\cal F};p)\;
     :=\;   \dimm({\cal F}|_p) \, -\, r\,.
  $$
 Note that ${\cal F}$ is flat at smooth points of $C$.
 Thus, $\delta_{\flatscriptsize}({\cal F};\,\bullet\,)=0$  except possibly at nodes of $C$
   and
  ${\cal F}$ is flat
      if and only of $\delta_{\flatscriptsize}({\cal F};\,\bullet\,)$ is identically $0$.
 Also, it follows from [Se1: Chapter~8] that
  $0\; \le \delta_{\flatscriptsize}({\cal F};p)\le r$ for all $p$.	
 Thus, one may define the {\it discrepancy to flatness of ${\cal F}$ on $C$} to be
  $$
   \delta_{\flatscriptsize}({\cal F})\;
     :=\;  \sum_{p\in C} \delta_{\flatscriptsize}({\cal F};p)\,.
  $$
} \end{definition}

\bigskip

It's clear that
  a torsion-free sheaf ${\cal F}$ of a fixed rank on $C$ is locally free
  if and only if $\delta_{\flatscriptsize}({\cal F})=0$.
This justifies the name.
  
\bigskip
 
\begin{definition}\label{nnpsptfsPt}
 {\bf [nonnegative/positive/strictly-positive torsion-free sheaf on ${\Bbb P}^1$-tree].}
{\rm {(Cf.\ Definition~\ref{nnpsp}.)}
     % Definition [nonnegative/positive/strictly-positive locally-free sheaf on ${\Bbb P}^1$-tree]
 A torsion-free sheaf ${\cal F}$ of rank $r$ on a ${\Bbb P }^1$-tree $C$
  is called
  \begin{itemize}
   \item[{\Large $\cdot$}]
  {\it nonnegative}
    if for each ${\Bbb P}^1$-component of $C$,
    $$
      ({\cal F}|_{{\Bbb P}^1})_{\torsionfreescriptsize}\;
	    :=\;   {\cal F}|_{{\Bbb P}^1}/
	                ({\cal F}|_{{\Bbb P}^1})_{\torsionscriptsize}\;
        \simeq\;     \oplus_{i=1}^r{\cal O}_{{\Bbb P}^1}(a_i)
    $$
     for some non-negative integers $0\le a_1\le\,\cdots\,\le a_r $,
	 where
    	 $({\cal F}|_{{\Bbb P}^1})_{\torsionscriptsize}$
		 is the torsion subsheaf of ${\cal F}|_{{\Bbb P}^1}$ involved;
	
    \item[{\Large $\cdot$}]
   {\it positive}
    if it is nonnegative  and
	each connected component of $C$ has at least one ${\Bbb P}^1$-component such that
    $({\cal F}|_{{\Bbb P}^1})_{\torsionfreescriptsize}
        \simeq     \oplus_{i=1}^r{\cal O}_{{\Bbb P}^1}(a_i)$
     for some non-negative integers $0\le a_1\le\,\cdots\,\le a_r $ with $a_r>0$;
	
   \item[{\Large $\cdot$}]
  {\it strictly positive}
    if for each ${\Bbb P}^1$-component of $C$,
    $({\cal F}|_{{\Bbb P}^1})_{\torsionfreescriptsize}
        \simeq     \oplus_{i=1}^r{\cal O}_{{\Bbb P}^1}(a_i)$
     for some non-negative integers $0\le a_1\le\,\cdots\,\le a_r $ with $a_r>0$.	
  \end{itemize}	
  By definition,
    $\;\; \mbox{\it strictly positive}\;\Rightarrow\; \mbox{\it positive}\; \Rightarrow\;
           \mbox{\it nonnegative}\,$. 	
}\end{definition}

\bigskip

\begin{definition}\label{wgCF}
 {\bf [weighted graph associated to $(C,{\cal F})$].} {\rm
 Let
  ${\cal F}$ be a torsion-free sheaf on a ${\Bbb P}^1$-tree curve $C$  and
  $\{C_{ij}\simeq {\Bbb P}^1\,:\, i,j\}$ be the set of irreducible components of $C$,
    where
	  $i$ labels the connected components of $C$   and
      $j$ labels the irreducible components in the $i$-th connected component of $C$.
  Suppose that 		
    $({\cal F}|_{C_{ij}})_{\torsionfreescriptsize}
	      \simeq  \oplus_{k=1}^r{\cal O}_{{\Bbb P}^1}(a_{ijk})$,
	$a_{ij1}\le\, \cdots\, \le a_{ijr}$.
  Define the {\it weighted graph $\Gamma_{(C,{\cal F})}$
    associated to $(C,{\cal F})$} to be the (unoriented) graph
	that associates
      to each $C_{ij}$  a vertex $v_{ij}$,
	     with weight $\{a_{ij1},\,\ldots\,a_{ijr}\}$,      and	
	  to each node $p$, say connecting $C_{ij}$ and $C_{ij^{\prime}}$,	
       an edge $e_{i,jj^{\prime}}$ connecting $v_{ij}$ and $v_{ij^{\prime}}$,
	     with weight $\delta_{\flatscriptsize}({\cal F};p)$.
  (By convention, $e_{i,jj^{\prime}}=e_{i,j^{\prime}j}$.)		
}\end{definition}

\bigskip

\noindent
Note that, forgetting the weights,
  $\Gamma_{(C,{\cal F})}$ is simply the dual graph/tree to $C$.

\bigskip

\begin{definition}\label{degree}
 {\bf [degree].} {\rm
 Continuing the setting in Definition~\ref{wgCF}.
                                       % Definition [weighted graph associated to $(C,{\cal F})$]  									   
  The {\it degree} $\degree({\cal F})$   of a torsion-free sheaf ${\cal F}$
     on a ${\Bbb P}^1$-tree curve $C$ is defined to be $\sum_{i,j,k}a_{ijk}$.
}\end{definition}
  
\bigskip

\begin{lemma}\label{h0F-2}
 {\bf [$h^0({\cal F})$].}
  {\rm (Cf.\ Lemma~\ref{h0E}.)}
                    % Lemma [$h^-0({\cal E})$]
 Continuing the setting in Definition~\ref{wgCF}
                                       % Definition [weighted graph associated to $(C,{\cal F}$]	
  with the additional assumption that ${\cal F}$ is nonnegative.
 Then
   $$
     h^0({\cal F})\;
  	 :=\;   \dimm H^0(C, {\cal F})\;
	  =\;   r\cdot  |\pi_0(C)|\;   +\;   \degree({\cal F})\;
	         +\; \delta_{\flatscriptsize}({\cal F})\,.
   $$
   where $|\pi_0(C)|$ is the number of connected components of $C$.
  Again, except the data already encoded in the weighted graph $\Gamma_{(C,{\cal F})}$
    associated to $(C,{\cal F})$,
   $h^0({\cal F})$ is independent of the isomorphism class of the pair $(C,{\cal F})$.
\end{lemma}

\begin{proof}
 It follows from the classification of the germs of a torsion-free sheaf at a node of a nodal curve,
  given by C.S.~Seshadri in [Se1: Chapter 8],
 that the torsion-free sheaf ${\cal F}$  is reconstructible from
  $({\cal F}|_{C_{ij}})_{\torsionfreescriptsize}$
  as follows:
  \begin{itemize}
   \item[{\Large $\cdot$}]
    Let $e$ be an edge in $\Gamma_{(C,{\cal F})}$
 	of weight $\delta_{i,jj^{\prime}}$   and
	connecting vertices $v_{ij}$ and $v_{ij^{\prime}}$.
	
   \item[{\Large $\cdot$}]
    At the level of curves,
	 $e$ means there are a pair of points,
	 $p_{i,jj^{\prime}}\in C_{ij}$ and  $p_{i,j^{\prime}j}\in C_{ij^{\prime}}$,
     which are glued to give a node $p$ in $C$.

   \item[{\Large $\cdot$}]
    At the level of sheaves, 	
	 $e$ means there are a pair of codimension-$\delta$ subspaces,
	   $$
	      H_{i,jj^{\prime}}\subset
	       ({\cal F}|_{C_{ij}})_{\torsionfreescriptsize}
		                                                                                           |_{p_{i,jj^{\prime}}}
		 \hspace{2em}\mbox{and}\hspace{2em}
	      H_{i,j^{\prime}j}\subset
	       ({\cal F}|_{C_{ij^{\prime}}})_{\torsionfreescriptsize}
		                                                                                           |_{p_{i,j^{\prime}j}}\,,
	   $$
      together with
        projection maps
		$$
		  \pi_{i,jj^{\prime}}:
		       ({\cal F}|_{C_{ij}})_{\torsionfreescriptsize}
		                                                                                           |_{p_{i,jj^{\prime}}}
            \longrightarrow     H_{i,jj^{\prime}}
		   	 \hspace{1em}\mbox{and}\hspace{1em}
	      \pi_{i,j^{\prime}j}:
		       ({\cal F}|_{C_{ij}})_{\torsionfreescriptsize}
		                                                                                           |_{p_{i,j^{\prime}j}}
            \longrightarrow     H_{i,j^{\prime}j}\,,
		$$
        and an isomorphism
		$$
		   h_{i,jj^{\prime}}\; :\;
		        H_{i,jj^{\prime}}\;\longrightarrow\; H_{i,j^{\prime}j}\,.				
		$$
	
    \item[{\Large $\cdot$}]
	 The data
	    $$
		      \{(H_{i,jj^{\prime}},\, \pi_{i,jj^{\prime}}\,;\,
                                  H_{i,j^{\prime}j},\, \pi_{i,j^{\prime}j}\,;\,
          							h_{i,jj^{\prime}})\}
	                           _{\{ij, ij^{\prime}\}\,
	                                                \mbox{\scriptsize corresponds to an edge
                													in $\Gamma_{C,{\cal F}}$}}
		$$
      recovers the torsion-free sheaf ${\cal F}$ on $C$
	   by gluing through the isomorphism $h_{\bullet}$
	   of codimension-$\delta_{\bullet}$ subspaces
        in the paired fibers of $({\cal F}|_{C_{ij}})_{\torsionfreescriptsize}$'s
        at paired points in $C_{ij}$'s
	 with a sheaf-of-local-sections structure through those
	    of  $({\cal F}|_{C_{ij}})_{\torsionfreescriptsize}$'s 	
		and the projection maps $\pi_{\bullet}$.	
	 In particular, it specifies the ${\Bbb C}$-vector space structure on the fiber
        at the node $p$ after gluing as the fibered product
		$$
		 \begin{array}{l}
		   ({\cal F}|_{C_{ij}})_{\torsionfreescriptsize}
		                                                                                           |_{p_{i,jj^{\prime}}}\;\;
			  _{\raisebox{-.8ex}
			                            {\scriptsize $\pi_{i,jj^{\prime}}, h_{i,jj^{\prime}}$}}\!\!
			  \oplus_{\raisebox{-.8ex}{\scriptsize $\pi_{i,jj^{\prime}}$}}
			 ({\cal F}|_{C_{ij}})_{\torsionfreescriptsize}
		                                                                                           |_{p_{i,j^{\prime}j}}\\[2ex]	
		 \hspace{1em}
		 :=\; \mbox{\footnotesize $
		         \Ker\left(
		       (-\,h_{i,jj^{\prime}}\circ \pi_{i,jj^{\prime}}\,,\,
                         \pi_{i,j^{\prime}j})\;:\;
		        ({\cal F}|_{C_{ij}})_{\mbox{\tiny\it torsion-free}}
		                                                                                    |_{p_{i,jj^{\prime}}}	
			      \oplus
			    ({\cal F}|_{C_{ij}})_{\mbox{\tiny\it torsion-free}}
		                                                                                    |_{p_{i,j^{\prime}j}}\;
				   \longrightarrow\;  H_{i,j^{\prime}j}		               		
		               \right)\,.$}
		\end{array}
	   $$
  \end{itemize}
 In terms of this gluing data,
  $$
    H^0(C,{\cal F})\; \simeq\;
	 \left\{
	   (s_{ij})_{ij}\;    \left|\;
	        \begin{array}{cl}
			 & \hspace{-1em}\mbox{\Large $\cdot$ }
			      s_{ij}\; \in\;
  				  H^0(C_{ij},({\cal F}|_{C_{ij}})_{\torsionfreescriptsize}) \\[1.2ex]
			 & \hspace{-1em}\mbox{\Large $\cdot$}\;
                 (h_{i,jj^{\prime}}\circ \pi_{i,jj^{\prime}})
				              (s_{ij}(p_{i,jj^{\prime}}))\;
				    =\; \pi_{i,j^{\prime}j}(s_{ij^{\prime}}(p_{i,j^{\prime}j}))\\[.8ex]
             & \mbox{for $\{ij, ij^{\prime}\}$ corresponding to an edge
                													                 in $\Gamma_{(C,{\cal F})}$}					
	        \end{array}
	                                            \right.
	  \right\}\,.
  $$
 Compared with the special case
   when ${\cal F}$ is locally free (i.e.\ $\delta_{\flatscriptsize}({\cal F})=0$)
   and the procedure of extending an existing global setion to cross a node to over another irreducible
   component of $C$, as studied in Sec.~2.2,
  here the same goes through except that each node $p\in C$ now contributes
     an additional $\delta_{\flatscriptsize}({\cal F},p) $-many dimensions
	 to $H^0(C,{\cal F})$.
 The lemma follows.
 
\end{proof}

\bigskip

\begin{lemma}\label{csptf}
 {\bf [criterion of strict positivity].}
 {\rm (Cf.\ Lemma~\ref{csp}.)}
                   % Lemma [criterion of strict positivity]
 Let ${\cal F}$ be a torsion-free sheaf on a ${\Bbb P}^1$-tree $C$.
 Then ${\cal F}$ is strictly positive
   if and only if
    ${\cal F}$ is globally generated   and
    the locally-free quotient $({\cal F}|_{{\Bbb P}^1})_{\torsionfreescriptsize}$
  	of its restriction to each irreducible component ${\Bbb P}^1$ has positive degree.
\end{lemma}
 
\begin{proof}
 We shall adopt the notation from Definition~\ref{wgCF}.
                                   % Definition [weighted graph associated to $(C,{\cal F})$]
								
 Given a strictly positive torsion-free sheaf ${\cal F}$ on a ${\Bbb P}^1$-tree $C$,
  by definition
    the locally-free quotient $({\cal F}|_{{\Bbb P}^1})_{torsionfreescriptsize}$
  	of its restriction to each irreducible component ${\Bbb P}^1$ has positive degree.
  Furthermore, the proof of Lemma~\ref{h0F-2} implies that
                                          % Lemma [$h^0({\cal F})$]
    the sequence from the natural evaluation homomorphism											
    ${\cal O}_C\times H^0(C,{\cal F})\rightarrow {\cal F}\rightarrow  0$ is exact.
	
 Conversely, suppose that
   ${\cal O}_C\times H^0(C,{\cal F}) \rightarrow {\cal F}\rightarrow  0$ is exact.
 Then, so are
   ${\cal O}_{C_{ij}}\times H^0(C,{\cal F})
          \rightarrow {\cal F}|_{C_{ij}}\rightarrow 0$.
  Thus, the composition
    ${\cal O}_{C_{ij}}\times H^0(C,{\cal F})
	      \rightarrow {\cal F}|_{C_{ij}}\rightarrow
	    ({\cal F}|_{C_{ij}})_{\torsionfreescriptsize}$
    is surjective.
 Since, furthermore, $({\cal F}|_{C_{ij}})_{\torsionfreescriptsize}$
   has positive degree,
   ${\cal F}$ must be strictly positive.
   
\end{proof}

\bigskip

\begin{flushleft}
{\bf Decrease of $\delta_{\flatscriptsize}(\,\bullet\,)$
          when bubbling off a ${\Bbb P}^1$-tree with a positive sheaf}
\end{flushleft}
\begin{proposition}\label{ddwbfTt}
 {\bf [decrease of $\delta_{\flatscriptsize}(\,\bullet\,)$
            when bubbling off ${\Bbb P}^1$-tree with positive sheaf].}
  Let
    $h:C^{\prime}\rightarrow C$ be a collapsing morphism of nodal curves that contracts
    	a ${\Bbb P}^1$-tree subcurve $C^{\prime}_u$ of $C^{\prime}$,
	${\cal F}$ be a torsion-free sheaf on $C$ of rank $r$,   and
	${\cal F}^{\prime}$ be a torsion-free sheaf on $C^{\prime}$, also of rank $r$,
	such that
	 \begin{itemize}
	   \item[(1)]
	   $({\cal F}^{\prime}|_{C^{\prime}_u})_{\torsionfreescriptsize}$
  		  is positive on $C^{\prime}_u$;
		
	 \item[(2)]
        there exists a surjective homomorphism
		 $\beta: h_{\ast}({\cal F}^{\prime})\rightarrow {\cal F}$
		 of ${\cal O}_C$-modules.		 		
     \end{itemize}
 Then,
    $$
      \delta_{\flatscriptsize}({\cal F}^{\prime})\,
	    -\,\delta_{\flatscriptsize}({\cal F})\;
	  =\;  -\, \degree(({\cal F}^{\prime}|_{C^{\prime}_u})
	                                                                                   _{\torsionfreescriptsize})\;
	  <\; 0\,.	
    $$	
\end{proposition}

\begin{proof}
 Since
   ${\cal F}$ is torsion-free and
   $h_{\ast}({\cal F}^{\prime})$ and ${\cal F}$ have the same rank,
  $$
      \Ker(\beta)\; =\;  (h_{\ast}({{\cal F}^{\prime}}))_{\torsionscriptsize}
         \hspace{2em}\mbox{and}\hspace{2em}
    (h_{\ast}({\cal F}^{\prime}))_{\torsionfreescriptsize}\;
	    \simeq\; {\cal F}\,.
  $$
 %%%%%%%%%%%%%%%%%%%%%%%%
 % Consider now the short exact  sequence
 %  $$
 %   0\; \longrightarrow\;  (h_{\ast}({\cal F}^{\prime}))_{\torsionscriptsize}\;
 %       \longrightarrow\; h_{\ast}({\cal F}^{\prime})\; \longrightarrow\;
 % 	  {\cal F}\;\longrightarrow\; 0\,.
 %  $$
 % Since ${\cal F}$ is torsion-free,
 %   $\Tor_1^{\,C}({\cal O}_q,{\cal F})=0$ for all $q\in C$.
 % It follows that the correspondence sequence at fibers at $q$
 %  $$
 %    0\; \longrightarrow\;  (h_{\ast}({\cal F}^{\prime}))_{\torsionscriptsize}|_q\;
 %       \longrightarrow\; (h_{\ast}({\cal F}^{\prime}))|_q\; \longrightarrow\;
 %	  {\cal F}|_q\;\longrightarrow\; 0\,.
 %  $$
 %  is also exact.
 % Consequently,
 %  $$
 %     \delta_{\flatscriptsize}(h_{\ast}({\cal F}^{\prime}))\,
 % 	   -\, \delta_{\flatscriptsize}({\cal F})\;
 % 	 =\;  l((h_{\ast}({\cal F}^{\prime})_{\torsionscriptsize}) \;
 % 	 =\; \chi(\Ker(\beta))\,.
 %  $$
 % It remains to compute
 %   $\delta_{\flatscriptsize}({\cal F}^{\prime})
 %     - \delta_{\flatscriptsize}(h_{\ast})$
 %%%%%%%%%%%%%%%%%%%%%%%%%%%%%
 Furthermore,
 since
    the restriction
	 $h:C^{\prime}-C^{\prime}_u\rightarrow C-h(C^{\prime}_u)$
	 is an isomorphism,
    ${\cal F}^{\prime}$ is torsion-free,
   and the scheme-theoretical preimage  $h^{-1}(h(C^{\prime}_u))$
     is exactly $C^{\prime}_u$,
  $\, \Ker(\beta)$ is supported at the finite set $h(C^{\prime}_u)$ of points on $C$.
 As we are concerned only with the difference
  $$
     \delta_{\flatscriptsize}({\cal F}^{\prime})\,
            -\,  \delta_{\flatscriptsize}({\cal F})\;
     =\;   \delta_{\flatscriptsize}({\cal F}^{\prime})\,
               -\,  \delta_{\flatscriptsize}
			         ((h_{\ast}({\cal F}^{\prime}))_{\torsionfreescriptsize})\,, 	
  $$
  we may assume, without loss of generality, that
   ${\cal F}^{\prime}$ is locally free on $C^{\prime}-C^{\prime}_u$
     (and, hence, ${\cal F}$ and $h_{\ast}({\cal F}^{\prime})$
     	 are locally free on $C-h(C^{\prime}_u)$ ).
    
 With this additional assumption,
   let $C^{\prime}=C^{\prime}_0\cup C^{\prime}_u$  be the decomposition fo $C$
   into a union of the contracted ${\Bbb P}^1$-tree $C^{\prime}_u$
   and the subcurve $C^{\prime}_0$ that's not contracted by $h$.
 Note that the restriction $h: C^{\prime}_0\rightarrow C$
     is birational, affine, surjective, and of relative dimension $0$.
 Let
    $p \in h(C^{\prime}_u)$   and
	$C^{\prime}_{u,p}:= h^{-1}(p)$,
	  a connected ${\Bbb P}^1$-tree subcurve in $C^{\prime}_u$.

 \begin{itemize}
  \item[{\Large $\cdot$}]
 {\it Case $(a):$  $p$ is a smooth point on $C$.}\\[.6ex]
   Then, $C^{\prime}_{u,p}$ has only one contact point
      $p^{\prime}$ with $C^{\prime}_0$.
    Under $h_{\ast}$,
      ${\cal F}^{\prime}|_{C^{\prime}_{u,p}}$
      contributes only to $\Ker(\beta)\subset h_{\ast}({\cal F}^{\prime})$
	  as the torsion subsheaf supported at the smooth point $p$.
	It has length
    $$
	  \degree
	      (({\cal F}^{\prime}|_{C^{\prime}_{u,p}})_{\torsionfreescriptsize})\,
	  +\, \delta_{\flatscriptsize}
	      (({\cal F}^{\prime}|_{C^{\prime}_{u,p}})_{\torsionfreescriptsize})\,
	  +\, \delta_{\flatscriptsize}({\cal F}^{\prime}; p^{\prime})\;
	  >\; 0\,,	
    $$
	as a consequence of Lemma~\ref{h0F-2}, 
	                                 % Lemma [$h^0({\cal F})$]
    though irrelevant to 
	  $\delta_{\flatscriptsize}({\cal F}^{\prime})
	       - \delta_{\flatscriptsize}({\cal F})$.
 \end{itemize}
	
 \begin{itemize}
   \item[{\Large $\cdot$}]
 {\it Case $(b):$  $p$ is a node on $C$.} \\[.6ex]
  Then, $C^{\prime}_{u,p}$ has two contact points,
   $p^{\prime}_-$ and $p^{\prime}_+$  with $C^{\prime}_0$.
  Under $h_{\ast}$,
   $\,{\tilde{\cal F}}^{\prime}|_{C^{\prime}_{u,p}}$  may contribute to
	both $\Ker(\beta)$ and
    $\delta_{\flatscriptsize}(
       (h_{\ast}({\cal F}^{\prime}))_{\torsionfreescriptsize}; p)$.      \\[.6ex]
   \mbox{$\hspace{1.2em}$}
   For the $\Ker(\beta)$ part, 	again as a consequence of Lemma~\ref{h0F-2},
                                                                           % Lemma [$h^0({\cal F})$]
   ${\cal F}^{\prime}|_{C^{\prime}_{u,p}}$
     contributes to $\Ker(\beta)$
	  as the torsion subsheaf supported at the node $p$ of length $\max\{0, \eta_p-r\}$,
	  where
    $$
	  \begin{array}{l}
	     \eta_p\;   :=\;
	        \degree
	       (({\cal F}^{\prime}|_{C^{\prime}_{u,p}})
		                                                                               _{\torsionfreescriptsize})\,
	        +\,  \delta_{\flatscriptsize}
	            (({\cal F}^{\prime}|_{C^{\prime}_{u,p}})
				                                                                       _{\torsionfreescriptsize})\\[1.2ex]
        \hspace{3.2em}																					
	     +\,  \delta_{\flatscriptsize}({\cal F}^{\prime}; p^{\prime}_-)\,
	     +\,  \delta_{\flatscriptsize}({\cal F}^{\prime}; p^{\prime}_+)\,
	       \hspace{9.2em},
	  \end{array}	
    $$	
	 which can be either $\le r$ or $> r$.
	(By convention, a torsion sheaf of length $0$ is the $0$-sheaf.)   \\[.6ex]
    \mbox{$\hspace{1.2em}$}
   For the
     $\delta_{\flatscriptsize}
       ((h_{\ast}({\cal F}^{\prime}))_{\torsionfreescriptsize}; p)$
	 part,
	since all local sections in the stalk of $h_{\ast}({\cal F}^{\prime})$ at $p$
	  contribute
	  to the fiber
	  $$
	     (h_{\ast}({\cal F }^{\prime}))|_{p}\;
		   \simeq\;
    		   \Ker(\beta)|_{p}\,
		       \oplus\,
 		      ((h_{\ast}({\cal F }^{\prime}))_{\torsionfreescriptsize})|_p
	  $$
	  of $h_{\ast}({\cal F})$ at $p$,
	  it follows from Lemma~\ref{h0F-2} and the above that
                          % Lemma [$h^0({\cal F})$]
     $$
      \dimm  ((h_{\ast}({\cal F }^{\prime}))_{\torsionfreescriptsize})|_p\;
	    =\; \left\{
		       \begin{array}{ccl}
			     \eta_p+r  && \mbox{if $\,\eta_p\le r\,$,}
			      \\[1.2ex]
			     2r                 &&  \mbox{if $\,\eta_p>r\,$.}
			   \end{array}
		     \right.
     $$
	 Note that, when $\eta_p\le r$, $r< \eta_p+r \le 2r$.
	 Thus,
	   $\,{\tilde{\cal F}}^{\prime}|_{C^{\prime}_{u,p}}$ always contributes to\\
        $\delta_{\flatscriptsize}(
       (h_{\ast}({\cal F}^{\prime}))_{\torsionfreescriptsize})$:
     $$
       \delta_{\flatscriptsize}(
       (h_{\ast}({\cal F}^{\prime}))_{\torsionfreescriptsize}; p)\;
	    =\; \left\{
		       \begin{array}{ccl}
			     \eta_p  && \mbox{if $\,\eta_p\le r\,$,}
			      \\[1.2ex]
			     r                 &&  \mbox{if $\,\eta_p>r\,$,}
			   \end{array}
		     \right.
     $$	
	 which is always positive.
  \end{itemize}
 Summing over the node $p\in h(C^{\prime}_u)\subset C$ in Case (b) above,
 one concludes that
  $$
    \delta_{\flatscriptsize}({\cal F}^{\prime})\,
	 -\, \delta_{\flatscriptsize}({\cal F})\;
	 =\;  -\, \degree(({\cal F}^{\prime}|_{C^{\prime}_u})
	       _{\torsionfreescriptsize})\;
     <\; 0\,.	
  $$
 This proves the proposition.
  
\end{proof}

\bigskip

\section{The space of D-string world-sheet instantons:
  The moduli stack of $Z$-semistable morphisms from Azumaya nodal curves
  with a fundamental module to a Calabi-Yau 3-fold}

With the preparations from [L-L-S-Y] (D(2)), [L-Y3] (D(10.1)),
and Sec.~\ref{preliminary}  of this note,
 we are now ready to define the notion of $Z$-semistable morphisms of a fixed type
 from general Azumaya nodal curves with a fundamental module to a Calabi-Yau $3$-fold
 and the stack of such objects, Sec.~3.1.
A natural morphism from this stack to the stack
  $\FM_g^{1,[0];\scriptsizeZss}(Y;c)$
  of Fourier-Mukai transforms is explained in Sec.~3.2.

\bigskip
  
\subsection{The moduli stack of $Z$-semistable morphisms from Azumaya nodal curves
        with a fundamental module to a Calabi-Yau 3-fold}
        \label{msssm}
We now bring out the main character of the D(10)-series of the project.

\bigskip
  
\begin{flushleft}
{\bf  $Z$-Semistable morphisms from Azumaya nodal curves with a fundamental module}
\end{flushleft}
For the current Sec.~3 and the next Sec.~4, we fix the following data:
 \begin{itemize}
   \item[{\Large $\cdot$}] {\it Domain data}$\,$:
	 \begin{itemize}
	  \item[{\Large $\cdot$}]
       \parbox{5em}{$\overline{\cal M}_g$}:
 	    the  moduli stack of stable curves of genus $g$,
	
      \item[{\Large $\cdot$}]
      \parbox{5em}{$C_{\overline{\cal M}_g}/\overline{\cal M}_g$}:
	    the universal curve over $\overline{\cal M}_g$,
		
	  \item[{\Large $\cdot$}]	
      \parbox{5em}{$[L]$}:
       a relative positive degree class on $C_{\overline{\cal M}_g}/\overline{\cal M}_g$.
	\end{itemize}
 
   \item[{\Large $\cdot$}]{\it Target data}$\,$:
    \begin{itemize}
	 \item[{\Large $\cdot$}]
      $(Y, B+\sqrt{-1}J)$: a projective Calabi-Yau 3-fold with a complexified K\"{a}hler class.
    \end{itemize}
 \end{itemize}
 
\bigskip

\begin{definition}\label{Zssm}
  {\bf [$Z$-(semi)stable morphism].} {\rm
 Let
   \begin{itemize}
    \item[{\Large $\cdot$}]
	 $(C^{\prime},{\cal E}^{\prime})$
      be a (connected) nodal curve $C^{\prime}$ of genus $g$
	  with a locally free sheaf ${\cal E}^{\prime}$ of rank $r$ and Euler characteristic $\chi$,
   
    \item[{\Large $\cdot$}]
     $\varphi\,:\,
	   (C^{\prime},
	    {\cal O}_{C^{\prime}}^{Az}
		     =\Endsheaf_{{\cal O}_{C^{\prime}}}({\cal E}^{\prime});
		{\cal E}^{\prime})\,
	    \rightarrow\,  Y	\;$
     be a morphism from an Azumaya nodal curve with a fundamental module
     to a projective Calabi-Yau $3$-fold $Y$
	 such that $[\varphi_{\ast}({\cal E}^{\prime})]=\beta\in A_1(Y)$,
	 (i.e.\ $\varphi$ is a morphism from an Azumaya nodal curve with a fundamental module to $Y$
	            of type $(g;r,\chi;\beta)$)
	
	\item[{\Large $\cdot$}]
	$\tilde{\cal E}^{\prime}_{\varphi}\in \CohCategory(C^{\prime}\times Y)$
	be the graph of $\varphi$, which is a $1$-dimensional coherent sheaf on
	$C^{\prime}\times Y$ that is flat over $C^{\prime}$ and
	of relative dimension $0$ and relative length $r$ over $C^{\prime}$,

   \item[{\Large $\cdot$}]	
    $\rho:C^{\prime}\rightarrow C$ be the collapsing morphism from $C^{\prime}$
	  to the stable curve $C$ associated to $C^{\prime}$
	 (i.e.\ $\rho$ stabilizes $C^{\prime}$ ),
	
   \item[{\Large $\cdot$}]
    $\Id_Y: Y\rightarrow Y$ be the identity map.
  \end{itemize}
 We say that $\varphi$ is {\it $Z$-semistable of type $(g;r,\chi;\beta,c)$}
  if the following additional conditions hold:
 \begin{itemize}
   \item[(1)]
	 $(\rho\times \Id_Y)_{\ast}(\tilde{\cal E}^{\prime}_{\varphi})
	     =: \tilde{\cal F}$
   	 is a $Z$-semistable Fourier-Mukai transform from $C$ to $Y$
	 with twisted central chagre $Z^{B+\sqrt{-1},[L]}(\tilde{\cal F})=c$.
   (Note that
	    $[{\pr_Y}_{\ast}(\tilde{\cal F})]
	         =[{\pr_Y}_{\ast}(\tilde{\cal E}^{\prime}_{\varphi})]\;
		     = [\varphi_{\ast}({\cal E}^{\prime})] =\beta$.)  	

   \item[(2)]
     The natural sequence of homomorphisms of ${\cal O}_{C^{\prime}\times Y}$-modules
	   $\;(\rho\times\Id_Y)^{\ast}(\tilde{\cal F})
	       \rightarrow \tilde{\cal E}^{\prime}_{\varphi}\rightarrow  0$\\
	  is exact.

   \item[(3)]
     For each ${\Bbb P}^1$-component (denoted by ${\Bbb P}^1$)
	    of the ${\Bbb P}^1$-tree subcurve of $C^{\prime}$
	  that is collapsed by $\rho$ to points in $C$,
     if $\varphi|_{{\Bbb P}^1}$	  is a constant morphism,
     then ${\Bbb P}^1$ has at least three special points
	  (i.e.\ points where ${\Bbb P}^1$ intersects with other components of $C^{\prime}$).
  %	
  %%%%%%%%%%%%%%%%%%	
  % \item[(4)]
  %	 $R^1\!(\rho\times \Id_Y)_{\ast}(\tilde{\cal E}^{\prime}_{\varphi})= 0$.	
  %%%%%%%%%%%%%%%%%%
  \end{itemize}

 We say that $\varphi$ is {\it $Z$-stable} if Condition (1) above is replaced by:
   \begin{itemize}
    \item[(1$^{\prime}$)]
   	 $(\rho\times \Id_Y)_{\ast}(\tilde{\cal E}^{\prime}_{\varphi})
	     =: \tilde{\cal F}$
   	 is a $Z$-stable Fourier-Mukai transform from $C$ to $Y$.
	\end{itemize}
	
 When the central charge functional  $Z$ is known and fixed either explicitly or implicitly,
  we may use the terminology {\it semistable morphism, stable morphism} for simplicity.	
}\end{definition}

\bigskip

\begin{remark}\label{C123}
 {$[$Conditions (1), (2), (3) in Definition~\ref{Zssm}$\,]$.} {\rm
                                                                                 % Definition [$Z$-(semi)stable morphism]
 The meaning behind Conditions (1), (2), and (3)
   in Definition~\ref{Zssm} is illuminated below:
                             % Definition [$Z$-(semi)stable morphism]
  %
  \begin{itemize} 	
    \item[(1)] {\it Condition (1) }$\;$ says that,
	 though the notion of (semi)stability of morphisms in our problem cannot be defined
	   completely just using twisted central charge,
	 that for the restriction of the morphism on the ``{\it main part}" of a general nodal curves
      {\it remains following the notion of (semi)stability associated to the twisted central charges}.
	
    \item[(2)] {\it Condition (2)}$\; $
  	is a ``{\it positivity condition}" on the restriction of morphisms
	  to the unstable bubbling ${\Bbb P}^1$-tree of nodal curves.
	It compensates for the fact that such positivity can be lost and undetectable
   	  from the twisted central charge due to ``overspreading of the twisted central charge".
	This condition says that, while possibly undetectable by twisted central charge,
	   the restriction of a morphism to such ${\Bbb P}^1$-tree can still cost
	   some ``positive internal rotating/winding energy".
	From this aspect, Condition (2) might be weakened to:
	  \begin{itemize}
	   \item[(2$^{\prime}$)]
        The natural sequence of homomorphisms of ${\cal O}_{C^{\prime}\times Y}$-modules
	      $\;(\rho\times\Id_Y)^{\ast}(\tilde{\cal F})
	       \rightarrow \tilde{\cal E}^{\prime}_{\varphi}\rightarrow  0$\\
	    is exact at every $1$-dimensional generic points of
		$\Supp(\tilde{\cal E}^{\prime}_{\varphi})$.
	  \end{itemize}
	 I.e.\ requiring only that
       $\Coker(
	   (\rho\times\Id_Y)^{\ast}(\tilde{\cal F})
	       \rightarrow \tilde{\cal E}^{\prime}_{\varphi}  )$
	    is $0$-dimensional, instead of $0$.	
	
    \item[(3)] {\it Condition (3)}$\;$ reminds one of a similar condition
	  for stable maps in algebraic Gromov-Witten theory ([Ko] and, e.g.\ [Be] for a survey).
	 In our case, even with Conditions (1) and (2),
	 there remains no control of the size of ${\Bbb P}^1$-chains
	  for which the restriction of the morphism is connected-componentwise constant.
	 As such chain takes no twisted central charge and costs no ``winding energy' of any sort,
      they are completely undetectable and uncontrollable.
     Fortunately, they are on the other hand
	  the only kind of ${\Bbb P}^1$-trees in a nodal curve
	   such that collapsing the ${\Bbb P}^1$-tree gives rise to a curve that remains nodal.
    From the above reasoning, to get a bounded family of morphisms,
	  one has no choice but to collapse all such ${\Bbb P}^1$-chains.
    What's left is exactly described in Condition (3).	
  (See Sec.~4.3, Theme `{\it
        Step $(d):\,$  Termination of the reduction
        -- Recovery of a regular morphism in our category over $t\in T$}'
      for a more technical and precise discussion.)
  \end{itemize}
}\end{remark}

\bigskip

\begin{definition}\label{mbZssm}
 {\bf [morphism between $Z$-semistable morphisms].} {\rm
 Let
  $$
	 \varphi_1:
      (C^{\prime}_1,
        {\cal O}^{Az}_{C^{\prime}_1}
		        :=\Endsheaf_{{\cal O}_{C^{\prime}}}({\cal E}^{\prime}_1);
		 {\cal E}^{\prime}_1)\;\longrightarrow\; Y
     \hspace{1.6em}\mbox{and}\hspace{1.6em}
	 \varphi_2:
      (C^{\prime}_2,
        {\cal O}^{Az}_{C^{\prime}_2}
		        :=\Endsheaf_{{\cal O}_{C^{\prime}}}({\cal E}^{\prime}_2);
		 {\cal E}^{\prime}_2)\;\longrightarrow\; Y	
  $$
   be two $Z$-semistable morphisms from Azumaya nodal curves with a fundamental module
    to $Y$ of type $(g;r,\chi;\beta,c)$.
 Then,
    a {\it morphism from $\varphi_1$ to $\varphi_2$},
	phrased directly in terms of their graph $\tilde{\cal E}^{\prime}_{\varphi_1}$
	 and $\tilde{\cal E}^{\prime}_{\varphi_2}$ respectively,
	is a pair  $(h^{\prime},\tilde{h}^{\prime})$,
	where
	 \begin{itemize}
	  \item[{\Large $\cdot$}]
	   $h^{\prime}:C^{\prime}_1\rightarrow C^{\prime}_2$
    	   is an isomorphism of nodal curves,
	
	  \item[{\Large $\cdot$}]
	   $\tilde{h}^{\prime}:
   	   (h^{\prime}\times \Id_Y)^{\ast}
	                  (\tilde{\cal E}^{\prime}_{\varphi_2})
	       \rightarrow   \tilde{\cal E}^{\prime}_{\varphi_1}$
		is an isomorphism of coherent sheaves on $C^{\prime}_1\times Y$.
	 \end{itemize}
 Note that,
   with the collapsing morphisms $\rho_1:C^{\prime}_1\rightarrow C_1$ and
   $\rho_2:C^{\prime}_2\rightarrow C_2$ specified,
 $h^{\prime}$ induces a unique isomorphism $h:C_1\rightarrow C_2$
   so that the following diagram commutes
   $$
     \xymatrix{
       C^{\prime}_1  \ar[rr]^-{h^{\prime}}  \ar[d]_-{\rho_1}
	     && C^{\prime}_2  \ar[d]^-{\rho_2}\\
       C_1 \ar[rr]^-{h}  && C_2 	
     }
   $$
   and
   $\tilde{h}^{\prime}$ induces further an isomorphism
   $\tilde{h}:
     (h\times\Id_Y)^{\ast}(\tilde{\cal F}_2)\rightarrow \tilde{\cal F}_1$
	so that the following diagram commutes
	$$
	 \xymatrix{
	  & (h\times\Id_Y)^{\ast}
	        ((\rho_2\times \Id_Y)_{\ast}(\tilde{\cal E}^{\prime}_{\varphi_2}))
		    \ar@{=}[d]	   \ar[rr]^-{\sim}	
           &&(\rho_1\times\Id_Y)_{\ast}
	            ((h^{\prime}\times\Id_Y)^{\ast}
	                      (\tilde{\cal E}^{\prime}_{\varphi_2}))
                  \ar[d]^-{(\rho_1\times I\!d_Y)_{\ast}(\tilde{h}^{\prime})} 	\\
     & (h\times\Id_Y)^{\ast}(\tilde{\cal F}_2)  \ar[rr]^{\tilde{h}}
         && \tilde{\cal F}_1   &,              	
	  }
    $$	
	where
     the isomorphism	
	   $(h\times\Id_Y)^{\ast}
	      ((\rho_2\times \Id_Y)_{\ast}(\tilde{\cal E}^{\prime}_{\varphi_2}))
		    \stackrel{\sim}{\rightarrow}
           (\rho_1\times\Id_Y)_{\ast}
	            ((h^{\prime}\times\Id_Y)^{\ast}
	                      (\tilde{\cal E}^{\prime}_{\varphi_2}))$
		on the top is the natural homomorphism (isomorphism in the current case)
         from interchanging the order of push and pull.
 These induced data of isomorphisms are automatically assumed when in need.
}\end{definition}

\bigskip

\begin{flushleft}
{\bf The moduli stack of $Z$-semistable morphisms of type $(g;r,\chi;\beta,c)$ }
\end{flushleft}
\begin{definition}\label{fZssm}
 {\bf [family of $Z$-semistable morphisms].} {\rm
 Let $S$ be a Noetherian scheme$/{\Bbb C}$.
 An {\it $S$-family $\varphi_S$ of $Z$-semistable morphisms of type $(g;r,\chi;\beta,c)$}
  from Azumaya nodal curves with a fundamental module to the Calabi-Yau $3$-fold $Y$
  is given, in terms of their graphs, by the following data:
  \begin{itemize}
   \item[{\Large $\cdot$}]
     \parbox{9em}{$C^{\prime}_S/S$}:
	 a flat family of nodal curves of genus $g$ over $S$;
	
   \item[{\Large $\cdot$}]	
	 \parbox{9em}{$\tilde{\cal E}^{\prime}_S
	                                       \in \CohCategory(C^{\prime}_S\times Y)$}:
     \parbox[t]{28em}{a coherent sheaf on $C^{\prime}_S\times Y$
	    that is flat over $C^{\prime}_S$   and
        		of relative dimension $0$ and relative length $r$ over $C^{\prime}_S$
	    such that
		\begin{itemize}
		  \item[{\Large $\cdot$}]
		   as a coherent sheaf on $C^{\prime}_S/S$,
		   ${\pr_{C^{\prime}_S}}_{\ast}(\tilde{\cal E}^{\prime}_S)$
		   is a flat family of locally free sheaves of rank $r$ and Euler characteristic $\chi$
		   on nodal curves over $S$,
		
		  \item[{\Large $\cdot$}]
		   for each $s\in S$,
		   $[{\pr_Y}_{\ast}(\tilde{\cal E}^{\prime}_s)]=\beta\in A_1(Y)$
		\end{itemize}}
  \end{itemize}
 that satisfy the following properties:
 \begin{itemize}
   \item[(1)]
    Let $\rho_S:C^{\prime}_S/S \rightarrow C_S/S$
	  be the collapsing $S$-morphism  that defines a morphism
	  $S\rightarrow \overline{\cal M}_g$.
    Then,		
	 $(\rho_S\times \Id_Y)_{\ast}(\tilde{\cal E}^{\prime}_S)
	     =: \tilde{\cal F}_S$
   	 is a flat family of $Z$-semistable Fourier-Mukai transforms from fibers of $C_S/S$ to $Y$
	 with twisted central charge $Z^{B+\sqrt{-1},[L]}(\tilde{\cal F}_s)=c$ for all $s\in S$.

   \item[(2)]
     The natural sequence of homomorphisms of
  	 ${\cal O}_{C^{\prime}_S\times Y}$-modules
	   $\;(\rho_S\times\Id_Y)^{\ast}(\tilde{\cal F}_S)
	       \rightarrow \tilde{\cal E}^{\prime}_S\rightarrow  0$\\
	  is exact.

   \item[(3)]
     For each ${\Bbb P}^1$-component (denoted by ${\Bbb P}^1$)
	    of the ${\Bbb P}^1$-tree subcurve of $C^{\prime}_s$, $s\in S$,
	  that is collapsed by $\rho_s$ to points in $C_s$,
     if $\tilde{\cal E}^{\prime}|_{{\Bbb P}^1\times Y}$
	    is isomorphic to $\pr_Y^{\ast}(\,\bullet\,)$ for some $0$-dimensional sheaf $\bullet$
		of length $r$ on $Y$
	 (i.e.\ $\varphi_s|_{{\Bbb P}^1}$	  is a constant morphism),
     then ${\Bbb P}^1$ has at least three special points in $C^{\prime}_s$.
 \end{itemize}	
} \end{definition}

\bigskip

\begin{definition}\label{mbSfZssm}
 {\bf [morphism between $S$-families of $Z$-semistable morphisms].}  {\rm
Let
  $$
	 \varphi_{1,S}:
      (C^{\prime}_{1,S},
        {\cal O}^{Az}_{C^{\prime}_{1,S}}
		        :=\Endsheaf_{{\cal O}_{C^{\prime}_{1,S}}}
				                     ({\cal E}^{\prime}_{1,S});
		 {\cal E}^{\prime}_{1,S})/S\;
	 \longrightarrow\; Y
  $$
   and
  $$
	 \varphi_{2,S}:
      (C^{\prime}_{2,S},
        {\cal O}^{Az}_{C^{\prime}_{2,S}}
		        :=\Endsheaf_{{\cal O}_{C^{\prime}_{2,S}}}
				                              ({\cal E}^{\prime}_{2,S});
		 {\cal E}^{\prime}_{2,S})/S\;
	   \longrightarrow\; Y	
  $$
   be two $S$-families of  $Z$-semistable morphisms
   from Azumaya nodal curves with a fundamental module to $Y$ of type $(g;r,\chi;\beta,c)$.
 Then,
    a {\it morphism from $\varphi_{1,S}$ to $\varphi_{2,S}$},
	phrased directly in terms of their graph
	 $\tilde{\cal E}^{\prime}_{\varphi_{1,S}}$
	 and $\tilde{\cal E}^{\prime}_{\varphi_{2,S}}$ respectively,
	is a pair  $(h^{\prime}_S, \tilde{h}^{\prime}_S)$,
	where
	 \begin{itemize}
	  \item[{\Large $\cdot$}]
	   $h^{\prime}_S : C^{\prime}_{1,S}/S
	        \rightarrow C^{\prime}_{2,S}/S$
    	   is an $S$-isomorphism of nodal curves,
	
	  \item[{\Large $\cdot$}]
	   $\tilde{h}^{\prime}_S:
   	   (h^{\prime}_S\times \Id_Y)^{\ast}
	                  (\tilde{\cal E}^{\prime}_{\varphi_{2,S}})
	       \rightarrow   \tilde{\cal E}^{\prime}_{\varphi_{1,S}}$
		is an $S$-isomorphism of coherent sheaves on $(C^{\prime}_{1,S}\times Y)/S$.
	 \end{itemize}
 As before,
   with the collapsing morphisms
     $\rho_{1,S}:C^{\prime}_{1,S}\rightarrow C_{1,S}$ and
     $\rho_{2,S}:C^{\prime}_{2,S}\rightarrow C_{2,S}$ specified,
 $h^{\prime}_S$ induces a unique isomorphism $h_S:C_{1,S}\rightarrow C_{2,S}$
   so that the following diagram commutes
   $$
     \xymatrix{
       C^{\prime}_{1,S}  \ar[rr]^-{h^{\prime}_S}  \ar[d]_-{\rho_{1,S}}
	     && C^{\prime}_{2,S}  \ar[d]^-{\rho_{2,S}}\\
       C_{1,S} \ar[rr]^-{h_S}  && C_{2,S} 	
     }
   $$
   and
   $\tilde{h}^{\prime}_S$ induces further an isomorphism
   $\tilde{h}_S:
     (h_S\times\Id_Y)^{\ast}(\tilde{\cal F}_{2,S})
	     \rightarrow \tilde{\cal F}_{1,S}$
	so that the following diagram commutes
	$$
	 \xymatrix{
	  & (h_S\times\Id_Y)^{\ast}
	        ((\rho_{2,S}\times \Id_Y)_{\ast}
			      (\tilde{\cal E}^{\prime}_{\varphi_{2,S}}))
		    \ar@{=}[d]	   \ar[rr]^-{\sim}	
           &&(\rho_{1,S}\times\Id_Y)_{\ast}
	            ((h^{\prime}_S\times\Id_Y)^{\ast}
	                      (\tilde{\cal E}^{\prime}_{\varphi_{2,S}}))
               \ar[d]^-{(\rho_{1,S}\times I\!d_Y)_{\ast}(\tilde{h}_S^{\prime})}\\
     & (h_S\times\Id_Y)^{\ast}(\tilde{\cal F}_{2,S})  \ar[rr]^{\tilde{h}_S}
         && \tilde{\cal F}_{1,S}   &.
	  } 	
    $$		
}\end{definition}

\bigskip

\begin{definition}\label{stack-Zssm}
  {\bf [stack of $Z$-semistable morphisms].} {\rm
 Define the {\it stack ${\frak M}_{Az^{\!f}\!(g;r,\chi)}^{\tinyZss}(Y; \beta, c)$
   of $Z$-semistable morphisms from Azumaya nodal curves with a fundamental module
   to the Calabi-Yau $3$-fold $Y$ of type $(g;r,\chi;\beta,c)$ }
   to be  the sheaf of groupoids that associates to each scheme $S$ over ${\Bbb C}$
   the groupoid ${\frak M}_{Az^{\!f}\!(g;r,\chi)}^{\tinyZss}(Y; \beta, c)(S)$
   whose objects are $S$-families of $Z$-semistable morphisms
     from Azumaya nodal curves with a fundamental to $Y$ of type $(g;r,\chi;\beta,c)$
	  in Definition~\ref{fZssm},
	      % Definition [family of $Z$-semistable morphisms]	
	 and whose morphisms are morphisms between $S$-families of $Z$-semistable morphisms
      in Definition~\ref{mbSfZssm}.
	      % Definition [morphism between $S$-families of $Z$-semistable morphisms]
} \end{definition}

\bigskip

\subsection{A natural morphism
	  from ${\frak M}_{Az^{\!f}\!(g;r,\chi)}^{\tinyZss}(Y; \beta, c)$
	  to $\FM_g^{1,[0];\scriptsizeZss}(Y;c)$}

\begin{lemma}\label{vR1sm}
{\bf [vanishing of $R^1\!(\rho\times \Id_Y)_{\ast}(\,\mbox{graph}\,)$
           for semistable morphism].}
   Let
    $S$ be an affine base scheme,
    $$
	    \varphi_S\; :\;
  		(C^{\prime}_S,
		 {\cal O}_{C_S}^{Az}
		     :=\Endsheaf_{{\cal O}_{C_S}}({\cal E}^{\prime}_S)/S;
		   {\cal E}^{\prime}_S)\; \longrightarrow\;   Y
	 $$
      be an $S$-family of semistable morphisms
	    from Azumaya nodal curves with a fundamental module to $Y$	 and
	 $\tilde{\cal E}^{\prime}_{\varphi_S}
	     \in \CohCategory((C^{\prime}_S\times Y)/S)$  	
       be its graph.
  Let  $\, \rho_S : C^{\prime}_S  \rightarrow C_S\,$
   be the built-in contracting homomorphism from nodal curves to stable curves.
  Then, Condition (2) required of semistable morphisms implies that
    $$
	   R^1\!(\rho_S\times \Id_Y)_{\ast}(\tilde{\cal E}^{\prime}_{\varphi_S})\;
	    =\; 0
	$$	
     in $\CohCategory((C_S\times Y)/S)$.
\end{lemma}

\begin{proof}
 Let $\Sigma_S\subset C_S$ be the image of contracted ${\Bbb P}^1$-trees
  on fibers of $C^{\prime}/S$  under $\rho_S$.
 Then $\Sigma_S$ has relative dimension $0$ over $S$.
 It follows that, after passing to a base change (still denoted by $S$) if necesary,
  there exists a relative ample Cartier divisor $H_{C_S/S}$  on $C_S/S$
    that has no intersection with $\Sigma_S$.
 Together with any ample Cartier divisor $H_Y$ on $Y$,
  they define a relative ample line bundle ${\cal O}_{(C_S\times Y)/S}(1)$
   on $(C_S\times Y)/S$
   that has the following property:
   $$
      R^1\!(\rho_S\times\Id_Y)_{\ast}(\tilde{\cal E}^{\prime}_{\varphi}
          \otimes_{{\cal O}_{C^{\prime}_S\times Y}}
		  (\rho_S\times \Id_Y)^{\ast}{\cal O}_{(C_S\times Y)/S}(m))\;
	  \simeq\;
	  R^1\!(\rho_S\times\Id_Y)_{\ast}(\tilde{\cal E}^{\prime}_{\varphi})
   $$
	for all $m$.
 
 Consider now the push-forward
   $\tilde{\cal F}_S
      := (\rho_S\times\Id_Y)_{\ast}(\tilde{\cal E}^{\prime}_{\varphi_S})$
   on $(C_S\times Y)/Y$.
 Then, there is an $m>0$ such that $\tilde{\cal F}_S(m)$ is globally generated.
 Let
     $$
	  {\cal O}_{C_S\times Y}^{\,\oplus k}\;\longrightarrow\;
	      \tilde{\cal F}_S(m)\;\longrightarrow\; 0
	 $$
  be the associated exact sequence of ${\cal O}_{(C_S\times Y)/S}$-modules.
 It follows from Condition (2) and right exactness of tensoring/pulling-back that
   this induces an exact sequence of ${\cal O}_{C^{\prime}_S\times Y}$-modules
   $$
     0;\longrightarrow\; \tilde{\cal H}^{\prime}_S\;    \longrightarrow\;
	   {\cal O}_{C^{\prime}_S\times Y }^{\,\oplus k}\;    \longrightarrow\;
	   \tilde{\cal E}^{\prime}_{\varphi_S}
	      \otimes _{{\cal O}_{C^{\prime}_S\times Y}}
	    (\rho_S\times \Id_Y)^{\ast}({\cal O}_{(C_S\times Y)/S}(m))\;
		\longrightarrow\; 0\,,
   $$
  and its associated long exact sequence of ${\cal O}_{C_S\times Y}$-modules
   $$
    \begin{array}{l}
      \cdots\; \longrightarrow\;
    	 R^1\!(\rho_S\times\Id_Y)_{\ast}
    	  ({\cal O}_{C^{\prime}_S\times Y }^{\,\oplus k})\;
        \longrightarrow\;
	   R^1\!(\rho_S\times\Id_Y)_{\ast}
	    (\tilde{\cal E}_{\varphi_S}
              \otimes_{{\cal O}_{C^{\prime}_S\times Y}}
		    (\rho_S\times \Id_Y)^{\ast}{\cal O}_{(C_S\times Y)/S}(m))\\[1.8ex]
       \longrightarrow\;
	   R^2\!(\rho_S\times\Id_Y)_{\ast}(\tilde{\cal H}^{\prime}_S)\;
	  \longrightarrow\;   \cdots\,.
	\end{array}
   $$
  Observe that
    $R^2\!(\rho\times\Id_Y)_{\ast}(\tilde{\cal H}^{\prime})=0$
	since $\rho_S\times\Id_S:C^{\prime}_S\times Y\rightarrow C_S\times Y$
	 has relative dimension $\le 1$.
  In addition,
    $R^1\!(\rho_S\times\Id_Y)_{\ast}
	  ({\cal O}_{C^{\prime}_S\times Y }^{\,\oplus k})=0$
	as well since $\rho_S\times\Id _Y$ collapses ${\Bbb P}^1$-trees.	
  It follows that 	
   $ R^1\!(\rho\times\Id_Y)_{\ast}
	  (\tilde{\cal E}_{\varphi}
            \otimes_{{\cal O}_{C^{\prime}\times Y}}
		  (\rho\times \Id_Y)^{\ast}{\cal O}_{C\times Y}(m))=0$
   and, hence, $R^1\!(\rho\times\Id_Y)_{\ast}(\tilde{\cal E}_{\varphi})=0$.
   
  So far, this is over the base-changed $S$ to guarantee the existence of
    the special $H_{C_S/S}$ at the beginning of the proof.
  However, the Theorem on Formal Functions implies that
   the same vanishing statement must hold over the original $S$.
 This proves the lemma.
   
\end{proof}

%%%%%%%%%%%%%%%%%%%%%%%%
% \bigskip
%
% \begin{lemma}{\bf [from pointwise to family].}
%  %
%   \marginpar{\raggedright\tiny $\bullet$
%      Put here as a lemma tentatively.   Remain to be justified or replaced.}
%  %
%  Let
%    $S$ be a separated Noetherian base scheme;
%    $f_S:X^{\prime}_S \rightarrow X_S$ be a projective $S$-morphism of relative dimension $1$
%      between  separated Noetherian schemes of finite type over $S$;  and
% 	${\cal F}^{\prime}_S$ be a coherent sheaf on $X^{\prime}_S$ that is flat over $S$.
%  Assume that $R^1\!{f_s}_{\ast}({\cal F}^{\prime}_s)=0$ for all $s\in S$,
%  then $R^1\!{f_S}_{\ast}({\cal F}^{\prime}_S)=0$.
% \end{lemma}
%
% \begin{proof}
%
%
%
%  ????????????????????.
%
% \end{proof}
%
%%%%%%%%%%%%%%%%%%%%%%%%

\bigskip

Continuing the setting in the above lemma.
It follows then from Lemma~\ref{cpf}  that
                    % Lemma [criterion for pushing forward a flat family to another flat family]
 $$
    \tilde{\cal F}_S\;
     :=\;
	  (\rho_S\times\Id_Y )_{\ast}(\tilde{\cal E}^{\prime}_{\varphi_S})
 $$
 is flat over $S$ and, hence,
defines a morphism
 $$
    S\; \longrightarrow\;  \FM_g^{1,[0];\scriptsizeZss}(Y;c)\,.
 $$
This shows that
   the built-in contracting morphisms $\rho\times\Id_Y$ in Definition~\ref{Zssm}
                                     % Definition [$Z$-(semi)stable morphism]
 induce a natural morphism
 $$
   \Xi\; :\; {\frak M}_{Az^{\!f}\!(g;r,\chi)}^{\tinyZss}(Y; \beta, c)\;
     \longrightarrow\;   \FM_g^{1,[0];\scriptsizeZss}(Y;c)
 $$
 from the moduli stack of semistable morphisms from general Azumaya nodal curves
    with a fundamental module
  to the moduli stack of Fourier-Mukai transforms from stable curves.

\bigskip

Altogether, one has the following diagram of natural morphisms between moduli stacks:
$$
 \xymatrix{
  & {\frak M}_{Az^{\!f}\!(g;r,\chi)}^{\tinyZss}(Y; \beta, c)
           \ar[rr]^-{\Xi} \ar[d]_-{\Phi}
        && \FM_g^{1,[0];\scriptsizeZss}(Y;c)
                 \ar[d]	 	\\
  & {\cal M}_{Az^{\!f}\!(g;r,\chi)}\ar[rr]
        &&   \overline{\cal M}_g  & .
   }
$$
Here,
 $\Phi:{\frak M}_{Az^{\!f}\!(g;r,\chi)}^{\scriptsizeZss}(Y;\beta,c)
        \rightarrow  {\cal M}_{Az^{\!f}\!(g;r,\chi)}$
 is the forgetful morphism
 $$
   [\varphi:(C^{\prime},
                     {\cal O}_C^{Az}
					       := \Endsheaf_{{\cal O}_{C^{\prime}}}({\cal E}^{\prime});
                       {\cal E}^{\prime})\rightarrow Y ]    \;
    \longmapsto\;
	[(C^{\prime},
                     {\cal O}_C^{Az}
					       := \Endsheaf_{{\cal O}_{C^{\prime}}}({\cal E}^{\prime});
                       {\cal E}^{\prime})]\,
	=\, [(C^{\prime},{\cal E}^{\prime})]\,,	
 $$
 ${\cal M}_{Az^{\!f}\!(g;r,\chi)}\rightarrow  \overline{\cal M}_g$
    comes from the collapsing morphism that stabilizes a nodal curve, and
  $\FM_g^{1,[0];\scriptsizeZss}(Y;c)  \rightarrow \overline{\cal M}_g$
      is the forgetful morphism
	 $[(C, \tilde{\cal F}\in \CohCategory(C\times Y))]  \mapsto [C]$.

\bigskip

\section{Compactness of the moduli stack
        ${\frak M}_{Az^{\!f}\!(g;r,\chi)}^{\scriptsizeZss}(Y;\beta,c)$\\
		of $Z$-semistable morphisms}

We now state the main theorem of the current notes, whose proof takes this whole section:

\bigskip

\noindent
{\bf Theorem 4.0.
 [${\frak M}_{Az^{\!f}\!(g;r,\chi)}^{\scriptsizeZss}(Y;\beta,c)$ compact].} 
{\it 
  The stack ${\frak M}_{Az^{\!f}\!(g;r,\chi)}^{\tinyZss}(Y; \beta, c)$
   of $Z$-semistable morphisms from Azumaya nodal curves with a fundamental module
   to the Calabi-Yau $3$-fold $Y$ of type $(g;r,\chi;\beta,c)$ 
  is bounded and complete (i.e.\ compact).
} % end-theorem

\bigskip

\bigskip

\subsection{Boundedness of
         ${\frak M}_{Az^{\!f}\!(g;r,\chi)}^{\scriptsizeZss}(Y;\beta,c)$}

Recall from [L-L-S-Y: Proposition~2.3.1]  (D(2)) (with Si Li and Ruifang Song):

\bigskip

\begin{proposition}\label{bounded-mbfd}
{\bf [boundedness of morphisms from bounded family of domains].}
 Let $(C^{Az}_S,{\cal E}_S)/S$ be a bounded family of Azumaya nodal curves
  with a fundamental module of type $(g;r,\chi)$ over a base scheme $S$ (of finite type).
 Then the stack ${\frak M}_{(C^{Az}_S\!,\,{\cal E}_S)/S}(Y;\beta)$
  of  morphisms of type $(g;r, \chi; \beta)$ from fibers of
  $(C^{Az}_S,{\cal E}_S)/S$ to a projective scheme $Y$ is bounded.
\end{proposition}

%%%%%%%%%%%%%%%%%
% \begin{proof}
%  Fix a relative ample line bundle on ${\cal C}_S/S$ and a ample line bundle on $Y$.
%  Then
%   any $\tilde{\cal E }\in \CohCategory(C_s\times Y)$
%   that gives a morphism from $(C^{Az}_s,{\cal E}_s)$, $s\in S$, to $Y$
%   has the following properties:
%   %
%   \begin{itemize}
%    \item[(1)]
%     The natural morphism
% 	  $\pr_{C_s}^{\ast}({\cal E}_s)\rightarrow \tilde{\cal E}_s $
% 	  is surjective.
%
%    \item[(2)]
%      $R^I{\pr_{C_s}}_{\ast}(\tilde{\cal E}_s)=0$, for $i>0$;
%      for any $m\in {\Bbb Z}$, there exists an $m^{\prime}\in {\Bbb Z}$ such that
%        $H^i(C_s\times Y,\tilde{\cal E}(m))= H^i(C_s,{\cal E}(m^{\prime}))$
% 	   for all $i$.
%   \end{itemize}
%   %
%  Both are the consequence of the fact that the restriction
%    $\pr_{C_s}:\Supp(\tilde{\cal E})\rightarrow C_s$
%     is an affine morphism of relative dimension $1$.
%  Property (1)    implies that any $\tilde{\cal E}$ involved is a quotient of a fiber of
%   the fixed bounded family $\pr_{C_S}^{\ast}({\cal E}_S)/S$ on $(C_S\times Y)/S$.
%  Property (2) implies that $$ the Mumford-Castelnuovo regularity of $\tilde{\cal E}$
%   on $C_s\times Y$ is bounded.
%    ?????????????
%
% \end{proof}
%%%%%%%%%%%%%%%%%%%%%%%%%%%
 
\bigskip

\noindent
It follows that to show that
 ${\frak M}_{Az^{\!f}\!(g;r,\chi)}^{\scriptsizeZss}(Y;\beta,c)$
  is bounded,
 we only need to show that the image stack of the forgetful morphism
   $\Phi:{\frak M}_{Az^{\!f}\!(g;r,\chi)}^{\scriptsizeZss}(Y;\beta,c)
        \rightarrow  {\cal M}_{Az^{\!f}\!(g;r,\chi)}$
   is bounded.
In other words,
consider the following commutative diagram of schemes and coherent sheaves thereupon:
 $$
   \xymatrix{
   &{\tilde{\cal E}}^{\prime}\ar@{.}[rd]
       &&&&&&&&  \tilde{\cal F}\ar@{.}[ld]                                                              \\
   && C^{\prime}\times Y
		     \ar[rrrrrr]^-{\rho\times\scriptsizeId_Y }
		     \ar[ddd]^-{pr_{C^{\prime}}}
       &&&&&
	   & C \times Y     \ar[ddd] _-{pr_C}                                                                       \\  \\
   &  {\cal E}^{\prime}\ar@{.}[rd]    &&&&&&&& {\cal F}\ar@{.}[ld]  \\
   && C^{\prime}\ar[rrrrrr]^-{\rho}       &&&&&&	 C   &&.
   }
 $$
Here,
\begin{itemize}
 \item[{\LARGE $\cdot$}]	
  $\tilde{\cal E}^{\prime}$ is a coherent ${\cal O}_{C^{\prime}\times Y}$-module
   that gives a $Z$-semistable morphism\\
    $[\varphi_{\tilde{\cal E}^{\prime}}]
	   \in {\frak M}_{Az^f(g;r,\chi)}^{\scriptsizeZss}(Y; \beta, c)$;

 \item[{\LARGE $\cdot$}]
  $\rho:C^{\prime}\rightarrow C$ is the contraction of unstable ${\Bbb P}^1$-trees
    in the nodal curve $C^{\prime}$ that renders it a stable curve $C$;

 \item[{\LARGE $\cdot$}]	
  $\tilde{\cal F}=(\rho\times \Id_Y)_{\ast}(\tilde{\cal E}^{\prime})$
   on $C\times Y$
    gives a $Z$-semistable Fourier-Mukai transform\\
	$[\tilde{\cal F}]  \in \FM_g^{1,[0];\scriptsizeZss}(Y;c)$;
	
 \item[{\LARGE $\cdot$}]	
  ${\cal E}^{\prime}
      ={\pr_{C^{\prime}}}_{\ast}(\tilde{\cal E}^{\prime})$
	is locally free on $C^{\prime}$;    and
	
 \item[{\LARGE $\cdot$}]	
  ${\cal F}={\pr_C}_{\ast}(\tilde{\cal F})
     = \rho_{\ast}({\cal E}^{\prime})$
   is a coherent ${\cal O}_C$-module, possibly with torsion.
\end{itemize}
Then,

\bigskip

\begin{theorem}\label{bounded-sZssm}
 {\bf [boundedness of stack of $Z$-semistable morphisms of fixed type].}
 The substack
   $$
    \Image\Phi\; :=\;
    \left\{
	   (C^{\prime},
	  {\cal O}_{C^{\prime}}^{Az}
	      (:=\Endsheaf_{{\cal O}_{C^{\prime}}}({\cal E}^{\prime}));
	  {\cal E }^{\prime})\;
	  \left|\;\:  \parbox{17.6em}{$[\varphi_{\tilde{\cal E}^{\prime}}]$
	           runs over all 
			  ${\frak M}_{Az^{\!f}\!(g;r,\chi)}^{\scriptsizeZss}(Y;\beta,c)$ and
	          $(C^{\prime},{\cal E}^{\prime})$ arises from the bottom-left corner
              of the above diagram. }
	     \right.\ \right\}
   $$
   of ${\cal M}_{Az^{\!f}\!(g;r,\chi)}$ is bounded    and, hence,
 the moduli stack
   ${\frak M}_{Az^{\!f}\!(g;r,\chi)}^{\scriptsizeZss}(Y; \beta, c)$
  of $Z$-semistable morphisms of type $(g;r,\chi;\beta,c)$
  from Azumaya nodal curves with a fundamental module to the Calabi-Yau $3$-fold $Y$
  is bounded.
\end{theorem}

\noindent
The proof of the theorem takes the rest of this subsection.
For the convenience of referral, we will call the above commutative diagram
 {\it the basic diagram of four}.

%\vspace{5em}
\bigskip

\begin{flushleft}
{\bf Reduction to boundedness of the family of pairs for stable and unstable subcurves.}
\end{flushleft}
Let
 $$
   C^{\prime}= C^{\prime}_0\cup C^{\prime}_u
 $$
  be a decomposition of $C^{\prime}$ into a union of two subcurves,
   where $C^{\prime}_u$ is the union of all connected ${\Bbb P}^1$-trees
      that are contracted by $\rho$,
  and
  $$
    C^{\prime}_u\;
	 =\;  \left(C^{\prime (0)}_{u,+}\cup C^{\prime(0)}_{u,0}\right)
	          \bigcup\,  C^{\prime (1)}_u
  $$
  be the further decomposition of $C^{\prime}_u$
   into the union of three ${\Bbb P}^1$ sub-trees such that
  \begin{itemize}
    \item[{\Large $\cdot$}]
	 for each ${\Bbb P}^1$-component of $C^{\prime (0)}_{u,+}$,
	  ${\pr_Y}_{\ast}(\tilde{\cal E}^{\prime}|_{{\Bbb P}^1})$
	  is $0$-dimensional\\  with
	  $\varphi_{\tilde{\cal E}^{\prime}}|_{{\Bbb P}^1}$ a nonconstant morphism;
	
    \item[{\Large $\cdot$}]
	 for each ${\Bbb P}^1$-component of $C^{\prime (0)}_{u,0}$,
	  ${\pr_Y}_{\ast}(\tilde{\cal E}^{\prime}|_{{\Bbb P}^1})$
	  is $0$-dimensional\\  with
	  $\varphi_{\tilde{\cal E}^{\prime}}|_{{\Bbb P}^1}$ a constant morphism;	
		
    \item[{\Large $\cdot$}]
 	 for each ${\Bbb P}^1$-component of $C^{\prime (1)}_u$,
	  ${\pr_Y}_{\ast}(\tilde{\cal E}^{\prime}|_{{\Bbb P}^1})$
	  is $1$-dimensional.
  \end{itemize}
For convenience,
 denote also the union as subcurves in $C^{\prime}$
 $$
    C^{\prime(0)}_u\;
       :=\;  C^{\prime(0)}_{u,+}\cup C^{\prime(0)}_{u,0}\,.
 $$
 
Our goal is to prove the following theorem:

\bigskip

\begin{theorem}\label{bounded-fpsc}
 {\bf [boundedness for family of pairs for subcurves].}
 The family of isomorphism classes of all possible pairs
    $(C^{\prime}_0,{\cal E}^{\prime}|_{C^{\prime}_0})$
    (resp.\  $(C^{\prime(0)}_{u,+},
                    	{\cal E}^{\prime}|_{C^{\prime(0)}_{u,+}})$,
                    $(C^{\prime(0)}_{u,0},
                    	{\cal E}^{\prime}|_{C^{\prime(0)}_{u,0}})$,							
                    $(C^{\prime(1)}_u, {\cal E}^{\prime}|_{C^{\prime(1)}_u})$)
     in the problem is bounded.
\end{theorem}

\bigskip

\noindent
Once this is justified, then since
   the gluing of the four locally free sheaves
     $$
	 (C^{\prime}_0,{\cal E}^{\prime}|_{ C^{\prime}_0})\,, \hspace{1em}
	 (C^{\prime(0)}_{u,a},{\cal E}^{\prime}|_{C^{\prime(0)}_{u,+}})\,,
	     \hspace{1em}
     (C^{\prime(0)}_{u,b},{\cal E}^{\prime}|_{C^{\prime(0)}_{u,0}})\,,	
      	 \hspace{1em}\mbox{and}\hspace{1em}
	 (C^{\prime(1)}_u,{\cal E}^{\prime}|_{C^{\prime(1)}_u})
	 $$
	 along fibers at attached points
	contributes only a finite-dimensional moduli to the resulting pair
	  $(C^{\prime},{\cal E}^{\prime})$ the problem,
 the resulting family of pairs $(C^{\prime},{\cal E}^{\prime})$
   would then be bounded.
 This would thus prove Theorem~\ref{bounded-sZssm}.
                           % Theorem [boundedness of stack of $Z$-semistable morphisms of fixed type]
 We have thus reduced the proof of Theorem~\ref{bounded-sZssm}
   to the justification of Theorem~\ref{bounded-fpsc}.
                                 % Theorem [boundedness for family of pairs for subcurves]

\bigskip
 
Before further study in details, note that
if letting $H$ be the hyperplane class on $Y$ associated to the very ample line bundle
   ${\cal O}_Y(1)$ on $Y$,
then, for each ${\Bbb P}^1$-component ${\Bbb P}^1$ of $C^{\prime(1)}_u$,
 the curve class
   $[{\pr_Y}_{\ast}(\tilde{\cal E}^{\prime}|_{{\Bbb P}^1})]$ on $Y$
    is effective and hence has a positive intersection number with $H$.
Since
 $$
   \sum_{\mbox{\scriptsize ${\Bbb P}^1\, :\,
                                    {\Bbb P}^1$-components of $C^{\prime(1)}_u$ }}
     H\cdot [{\pr_Y}_{\ast}(\tilde{\cal E}^{\prime}|_{{\Bbb P}^1})]\;
  =\; H \cdot
         [{\pr_Y}_{\ast}(\tilde{\cal E}^{\prime}|_{C^{\prime(1)}_u})]\;
  <\; H\cdot \beta\,,
 $$
 the number of irreducible components of $C^{\prime(1)}_u$  is uniformly bounded by
  $H\cdot \beta$.
Together with the fact that $C^{\prime}_0$ is a partial normalization of a subcurve
 of a stable curve of genus $g$,  one has the following lemma:

\bigskip

\begin{lemma}\label{bounded-ctC0Cp1u} {\bf
[boundedness of combinatorial types of $C^{\prime}_0$ and $C^{\prime(1)}_u$].}
 The set of combinatorial type (and hence homeomorphism classes)
  of $C^{\prime}_0$ and ${\Bbb P}^1$-trees $C^{\prime(1)}_u$
  that can occur in our problem is finite,
  depending only on the type $(g; r,\chi;\beta,c)$ of morphisms considered.
\end{lemma}

\bigskip

\noindent
In particular, both numbers of connected components,
   $|\pi_0(C^{\prime}_0)|$ and  $|\pi_0(C^{\prime(1)}_u)|$,
  are uniformly bounded, depending only on $g$  and $\beta$ respectively.

% \vspace{3em}
\bigskip

\begin{flushleft}
{\bf Boundedness of the family of isomorphisms classes of pairs
   $(C^{\prime (0)}_{u,+},
        {\cal E}^{\prime}|_{C^{\prime (0)}_{u,+}})$}
\end{flushleft}
$(a)$
{\it Strict positivity of
          ${\cal E}^{\prime}|_{C^{\prime(0)}_{u,+}}$}
		
\medskip

\noindent
Observe that from the way we define a semistable morphism in Definition~\ref{Zssm},
                                                               % Definition [$Z$-(semi)stable morphism]
 the restriction of $\tilde{\cal E}^{\prime}$ to over a ${\Bbb P}^1$-component,
   denoted simply  as ${\Bbb P}^1$,
  of $C^{\prime(0)}_{u,+}$ satisfies the three properties:
   \begin{itemize}
   \item[(1)]
     $\tilde{\cal E}^{\prime}|_{{\Bbb P}^1}$
     	 is realizable as a quotient of a trivial bundle on $C^{\prime}\times Y$,
		 	
   \item[(2)]	
     ${\pr_Y}_{\ast}(\tilde{\cal E}^{\prime}|_{{\Bbb P}^1})$
      is $0$-dimensional,
   
  \item[(3)]
    $\varphi_{\tilde{\cal E}^{\prime}}$ is not constant on this ${\Bbb P}^1$.
 \end{itemize}	
  in the beginning of Sec.~2.2.
In other words, $\varphi_{\tilde{\cal E}^{\prime}}|_{C^{\prime (0)}_{u,+}}$
  is a special morphism studied in Sec.~2.2.
It follows from  Proposition~\ref{ncPEasm}
                      % Proposition [necessary condition for $({\Bbb P}^1\mbox{-tree}, {\cal E})$
                      %                       to admit special morphism]	
  that:
  
\bigskip
  
\begin{lemma}\label{spEpCp0u+}
{\bf [strict positivity of ${\cal E}^{\prime}|_{C^{\prime(0)}_{u,+}}$].}\\
     ${\cal E}^{\prime}|_{C^{\prime(0)}_{u,+}}$
	 is a strictly positive locally-free sheaf on $C^{\prime(0)}_{u,+}$.
\end{lemma}

\bigskip
\bigskip

\noindent
$(b)$
{\it $H^0({\cal E}^{\prime}|_{C^{\prime(0)}_{u,+}}   )$
          and torsions and non-locally-free-ness of ${\cal F}$}

\medskip

\noindent
We now study how the size of the ${\Bbb P}^1$-tree $C^{\prime(0)}_{u,+}$
 influences the push-forward ${\cal F}:= \rho_{\ast}({\cal E}^{\prime})$ on $C$
 in the basic diagram of four.
 
Consider the decomposition	
 $$
   C^{\prime(0)}_u\;  =\; C^{\prime(0)}_{u,a}\cup C^{\prime(0)}_{u,b}\,,
 $$
 where
  $C^{\prime(0)}_{u,a}$ is the union of connected components of $C^{\prime(0)}_u$
    each of which contains some connected components of $C^{\prime(0)}_{u,+}$   and
  $C^{\prime(0)}_{u,b}\subset C^{\prime(0)}_{u,0}$. 	
Since
  ${\cal E}^{\prime}|_{C^{\prime(0)}_{u,0}}
       \simeq {\cal O}_{C^{\prime(0)}_{u,-}}^{\;\oplus r}$,
 one can rephrase Lemma~\ref{h0E} in the current notations as:
                            % Lemma [$h^0({\cal E})$]

\bigskip

\begin{lemma}\label{h0EpCp0ua}
 {\bf [$h^0({\cal E}^{\prime}|_{C^{\prime(0)}_{u,a}})$].}
 {\rm(Cf.\ Lemma~\ref{h0E}.) }
 Let $C^{\prime(0)}_{u,a;ij}\simeq {\Bbb P}^1$ be the irreducible components
    of $C^{\prime(0)}_{u,a}$,
    where
	  $i$ labels the connected components of $C^{\prime(0)}_{u,a}$   and
      $j$ labels the irreducible components in the $i$-th connected component
    	  of $C^{\prime(0)}_{u,a}$ that lie in $C^{\prime(0)}_{u,+}$.
  Suppose that 		
    ${\cal E}|_{C^{\prime(0)}_{u,a;ij}}
	      \simeq  \oplus_{k=1}^r{\cal O}_{{\Bbb P}^1}(a_{ijk})$.
  Recall from Part (a) that, up to a relabeling,
     $0\le a_{ij1}\le \cdots \le a_{ijr}$ with $a_{ijr}>0$.		
  Then
   $$
     h^0({\cal E}^{\prime}|_{C^{\prime(0)}_{u,a}})\;
  	 :=\;   \dimm H^0(C^{\prime(0)}_{u,a},
	                                {\cal E}^{\prime}|_{C^{\prime(0)}_{u,a}})\;
	  =\;   r\cdot  |\pi_0( C^{\prime(0)}_{u,a})|\;   +\;   \sum_{i,j,k}a_{ijk}\,,	
   $$
   where $|\pi_0(C^{\prime(0)}_{u,a})|$
     is the number of connected components of $C^{\prime(0)}_{u,a}$.
  In particular,  except the number of connected components of $C^{\prime(0)}_{u,a}$,
  it is independent of how this collection of locally free sheaves on ${\Bbb P}^1$'s
   are  glued to give ${\cal E}^{\prime}|_{C^{\prime(0)}_{u,a}}$
   on $C^{\prime(0)}_{u,a}$.
\end{lemma}

\bigskip

As a consequence, one has the following lemma:
                              % Lemma [$h^0({\cal E}^{\prime}|_{C^{\prime(0)}_u})$]

\bigskip

\begin{lemma}\label{lFtfbCp0u+}
{\bf [length of ${\cal F}_{\torsionscriptsize}$
            and bound on $(C^{\prime(0)}_{u,+},
			                            {\cal E}^{\prime}|_{C^{\prime(0)}_{u,+}})$].}
 Let
  $$
    n_a\; :=\; |C^{\prime(0)}_{u,a}\cap (C^{\prime}_0\cup C^{\prime(1)}_u)|
  $$
    be the number of attached points on $C^{\prime(0)}_u$ to the rest of $C^{\prime}$.
 Note that since
  $C^{\prime(0)}_{u,a}\cup (C^{\prime}_0\cup C^{\prime(1)}_u)$
   is a nodal curve of genus $\le g$,
 it follows from Lemma~\ref{bounded-ctC0Cp1u} that
    % Lemma [boundedness of combinatorial types of $C^{\prime}_0$ and $C^{\prime(1)}_u$]
  $n_a$ is uniformly bounded, depending only on $(g;r,\chi;\beta,c)$.
 Let
   ${\cal F}_{\torsionscriptsize}$ be the torsion subsheaf of
      ${\cal F}= \rho_{\ast}({\cal E}^{\prime})$   and
  $l({\cal F}_{\torsionscriptsize})$ be its length.
 Then,
   $$
     l({\cal F}_{\torsionscriptsize})\;
	 \ge \;  h^0({\cal E}^{\prime}|_{C^{\prime(0)}_{u,a}})\,-\, r\,n_a\,.
   $$
 In particular,
 if there is a uniform upper bound for $l({\cal F}_{\torsionscriptsize})$
  then there are
    a uniform upper bound for the number of irreducible components
	of $C^{\prime(0)}_{u,+}$ and
    a uniform upper bound for $a_{ijk} \ge 0$ in Lemma~\ref{h0EpCp0ua}.
		         % Lemma [$h^0({\cal E}^{\prime}|_{C^{\prime(0)}_{u,a}})$]		
\end{lemma}

\begin{proof}
 Note that $\rho_{\ast}$ takes
  $$
    \{s\in H^0({\cal E}^{\prime}|_{C^{\prime(0)}_{u,a}},)\,|\,
	     \mbox{$s$ vanishes at all the attached points on $C^{\prime(0)}_{u,a}$ }\}
  $$
  injectively to the torsion subsheaf ${\cal F}_{\torsionscriptsize}$ of ${\cal F}$.
 Recall Remark~\ref{cstcpt}
        % Remark [constrained sections and torsions under collapsing ${\Bbb P}^1$-tree]
    and observe from the explicit formula in Lemma~\ref{h0EpCp0ua} that
                      % Lemma [$h^0({\cal E}^{\prime}|_{C^{\prime(0)}_{u,a}})$]		
   $$
      h^0({\cal E}^{\prime}|_{C^{\prime(0)}_{u,a}})\,-\,r\,n_a \;
	   \longrightarrow\;  \infty
   $$
   if either the number of irreducible components of $C^{\prime(0)}_{u,+}$ goes to infinity
      or some $a_{ijk}$ in Lemma~\ref{h0EpCp0ua} goes to infinity
	             % Lemma [$h^0({\cal E}^{\prime}|_{C^{\prime(0)}_{u,a}})$]
   {\it except} the situation of
     the connected components $C^{\prime\prime}$ of $C^{\prime}_{u,a}$
	 that contain only ${\Bbb P}^1$-chain connected component
      	   of $C^{\prime(0)}_{u,+}$
        that has total length $\le r$,  with an attached point on each ${\Bbb P}^1$- end, and 	
	    with the total degree of ${\cal E}^{\prime}|_{{\Bbb P }^1}$'s,
	   where ${\Bbb P}^1$ runs over the ${\Bbb P}^1$-components of $C^{\prime\prime}$,
	   $\le r$
	since collapsing such $C^{\prime\prime}$ may not create torsions from pushing forward.
 However, when such collection of ${\Bbb P}^1$-chains in $C^{\prime\prime}$
   do not contribute to ${\cal F}_{\torsionscriptsize}$,
   they then contribute positively to
  $\delta_{\flatscriptsize}({\cal F}_{\torsionfreescriptsize})$,
  where ${\cal F}_{\torsionfreescriptsize}:= {\cal F}/{\cal F}_{\flatscriptsize}$.
 Since
  $$
    \delta_{\flatscriptsize}({\cal F}_{\torsionfreescriptsize})\;
	 \le\; r \cdot \mbox{(number of nodes of $C$)}
  $$
  and the number of nodes of a stable curve of genus $g$ is uniformly bounded above,
 the total number of ${\Bbb P}^1$-chains in the union of all such $C^{\prime\prime}$
    is uniformly bounded
     and, hence, the family of isomorphism classes of
   $(C^{\prime\prime}\cap C^{\prime(0)}_{u,+},
     {\cal E}^{\prime}|_{C^{\prime\prime}\cap C^{\prime(0)}_{u,+}})$
    is also bounded.
  %
  %%%%%%%%%%%%%%%%%%%%%%%%%%%%%% 	
  %  \bigskip
  %
  %  \noindent {\it Case $(i)$:
  %    $C^{\prime\prime}$ is attached to $C^{\prime}_0$ at both ends.}
  %	
  %  \smallskip
  %
  %  \noindent
  %    Collapsing $C^{\prime\prime}$ in this case creates a node in $C$.
  %	  Since there is a uniform upper bound on the number of nodes  that can appear
  %	  on stable curve $C$ of genus $g$,
  %	  the number of such $C^{\prime\prime}$ is uniformly bounded above as well.
  %	
  %   \bigskip
  %
  %   \noindent {\it Case $(ii)$:
  %   $C^{\prime\prime}$ is attached to $C^{\prime}_0$ at one end
  %	 and to $C^{\prime(1)}_u$ at the other end.}
  %	
  %	 \smallskip
  %	
  %	 \noindent
  %	 Recall Lemma~???
  %	        % Lemma [bound on combinatorial types of $C^{\prime(1)}_u$]
  %     that the number of combinatorial types of $C^{\prime(1)}_u$
  %		  is uniformly bounded.
  %   Since both $|\pi_0(C^{\prime}_0)|$ and $|\pi_0(C^{\prime(1)}_u)|$		
  %       are uniformly bounded,
  %	 there can only be a uniformly bounded number of ${\Bbb P}^1$-chains
  % 	 connecting between $C^{\prime}_0$ and $C^{\prime(1)}_u$
  % 	 to keep the genus of the resulting curve $C$ fixed at $g$.
  %	
  %  \bigskip
  %
  %   \noindent {\it Case $(iii)$:
  %   $C^{\prime\prime}$ is attached to $C^{\prime(1)}_u$ at both ends.}
  %	
  %     \smallskip
  %
  %    \noindent	
  %    Since $|\pi_0(C^{\prime(1)}_u)|$ is uniformly bounded by $H\cdot \beta$,
  %	 the number of ${\Bbb P}^1$-chains that connect among connected components of
  %	 $C^{\prime(1)}_u$, keeping the result a ${\Bbb P}^1$-tree,
  %	 must also be uniformly bounded.
  %	
  %	\bigskip
  %
  %%%%%%%%%%%%%%%%%%%%%%%%%%%%
  %
 This proves the lemma.
 
\end{proof}	

\bigskip

Now $l({\cal F}_{\torsionscriptsize})$ does have a uniform upper bound
 since ${\cal F}$ comes from the push-forward of the elements in the compact stack
 $\FM_g^{1,[0];\scriptsizeZss}(Y;c)$
 of coherent sheaves.
This proves:

\bigskip

\begin{proposition}\label{bounded-Cp0u+}
{\bf [boundedness of
            $\{(C^{\prime(0)}_{u,+},
			   {\cal E}^{\prime}|_{C^{\prime(0)}_{u,+}})\}$].}
  The family of isomorphisms classes of pairs
   $(C^{\prime (0)}_{u,+}, {\cal E}^{\prime}|_{C^{\prime (0)}_{u,+}})$
   is bounded.		
\end{proposition}

%\vspace{3em}
\bigskip

\begin{flushleft}
{\bf Boundedness of the family of  isomorphisms classes of pairs
   $(C^{\prime(0) }_{u,0}, {\cal E}^{\prime}|_{C^{\prime(0) }_{u,0}})$}
\end{flushleft}
Since ${\cal E}^{\prime}|_{C^{\prime(0) }_{u,0}}
    \simeq  {\cal O}_{C^{\prime(0) }_{u,0}}^{\;\oplus r}$,
 we only need to show that
 there is a uniform bound on the combinatorial types of $C^{\prime(0)}_{u,0}$.
Consider the dual graph $\Gamma_{C^{\prime(0)}_{u,0}}$
  of $C^{\prime(0)}_{u,0}$, which is a tree-subgraph of the dual graph
  $\Gamma_{C^{\prime}}$ of  $C^{\prime}$.
Since
 \begin{itemize}
  \item[{\Large $\cdot$}]
    the genus of $C^{\prime}$ is fixed to be $g$,
	
  \item[{\Large $\cdot$}]	
   the sets of combinatorial types of
    $C^{\prime}_0$, $C^{\prime(1)}_u$, and $C^{\prime(0)}_{u,+}$
    respectively are all bounded from Lemma~\ref{bounded-ctC0Cp1u} and the previous theme,  and
     % Lemma [boundedness of combinatorial types of $C^{\prime}_0$ and $C^{\prime(1)}_u$]
	
  \item[{\Large $\cdot$}]	
   each vertex of the tree $\Gamma_{C^{\prime(0)}_{u,0}}$
     has valance at least $3$ as a vertex in $\Gamma_{C^{\prime}}$,
 \end{itemize}
 the number of  $1$-valance vertices of $\Gamma_{C^{\prime(0)}_{u,0}}$
  must be uniformly bounded.
This implies that the number of vertices of $\Gamma_{C^{\prime(0)}_{u,0}}$
  whose valance is $\ge 3$ must also be bounded,
 since
  $$
    |\{\mbox{$1$-valance vertics of a tree $\Gamma$}\}|   \;
	 \ge \;   |\{\mbox{vertics of valance $\ge 3$ in $\Gamma$}\}| \,+\, 2  \,.
  $$
Which implies in turn that
  the number of $2$-valance vertices of $\Gamma_{C^{\prime(0)}_{u,0}}$
  musty also be uniformly bounded
  since the genus $g(\Gamma_{C^{\prime}})\le g(C^{\prime})=g$.
In conclusion:

\bigskip

\begin{proposition}\label{bounded-Cp0u0}
{\bf [boundedness of
            $\{(C^{\prime(0)}_{u,0},
			   {\cal E}^{\prime}|_{C^{\prime(0)}_{u,0}})\}$].}
  The family of isomorphisms classes of pairs
   $(C^{\prime (0)}_{u,0},
       {\cal E}^{\prime}|_{C^{\prime (0)}_{u,0}})$
   is bounded.		
\end{proposition}

\vspace{3em}
\bigskip

\begin{flushleft}
{\bf Boundedness of the family of  isomorphisms classes of pairs
   $(C^{\prime }_0, {\cal E}^{\prime}|_{C^{\prime }_0})$}
\end{flushleft}
$(a)$
{\it Bound for  $|\chi({\cal E}^{\prime}|_{C^{\prime}_0})|$}

\medskip

\noindent
Let ${\cal I}_{C^{\prime}_u}$ be the ideal sheaf of $C^{\prime}_u$ in $C$.
Then one has the short exact sequence of ${\cal O}_C$-modules
 $$
   0\;\longrightarrow\; {\cal I}_{C^{\prime}_u}\cdot{\cal E}^{\prime}\;
        \longrightarrow\;  {\cal E}^{\prime}\; \longrightarrow\;
		{\cal E}^{\prime}|_{C^{\prime}_u}\; \longrightarrow\; 0\,.
 $$
Since
  ${\cal I}_{C^{\prime}_u}\cdot{\cal E}^{\prime}$
      is supported on $C^{\prime}_0$    and
  $\rho|_{C^{\prime}_0}\rightarrow C$  is affine birational of relative dimension $0$,
 one has the short exact sequence
 $$
   \begin{array}{ccccccccc}
     0   & \longrightarrow
	      & \rho_{\ast}({\cal I}_{C^{\prime}_u}\cdot{\cal E}^{\prime})
		  & \longrightarrow
		  & \rho_{\ast}({\cal E}^{\prime})
		  & \longrightarrow
		  & \rho_{\ast}({\cal E}^{\prime}|_{C^{\prime}_u})
		  & \longrightarrow
		  & R^1\!\rho_{\ast}({\cal I}_{C^{\prime}_u}\cdot{\cal E}^{\prime})\\[1.2ex]
    &&&&    \|                  &&&& \|                             \\[.8ex]
    &&&& {\cal F}    &&&& 0	
   \end{array}
 $$
   and
 $$
    \chi({\cal I}_{C^{\prime}_u}\cdot{\cal E}^{\prime})\;      =\;
	\chi(\rho_{\ast}({\cal I}_{C^{\prime}_u}\cdot{\cal E}^{\prime}))\,.
 $$
Since ${\cal I}_{C^{\prime}_u}|_{C^{\prime}_0}$
 is the ideal sheaf of the nodes $C^{\prime}_0\cap C^{\prime}_u$ in $C^{\prime}_0$,
one has also a short exact sequence of ${\cal O}_{C^{\prime}_0}$-modules
 $$
    0\; \longrightarrow\; {\cal I}_{C^{\prime}_u}\cdot{\cal E }^{\prime}\;
	      \longrightarrow\; {\cal E}^{\prime}|_{C^{\prime}_0}\;
		  \longrightarrow\; {\cal G}^{\prime}\; \longrightarrow\; 0\,,
 $$
 where
  ${\cal G}^{\prime}$ is $0$-dimensional,
     supported at $C^{\prime}_0\cap C^{\prime}_u$,  and
     of length $r\,|C^{\prime}_0\cap C^{\prime}_u|$,
   which is uniformly bounded by Lemma~\ref{bounded-ctC0Cp1u}
       and Lemma~\ref{lFtfbCp0u+}.
      % Lemma [bound on combinatorial types of $C^{\prime}_0$ and $C^{\prime(1)}_u$]
	  % Lemma [length of ${\cal F}_{\torsionscriptsize}$
      %                 and bound on $(C^{\prime(0)}_{u,+},
      %			                            {\cal E}^{\prime}|_{C^{\prime(0)}_{u,+}})$]
	
Since
    both $\rho_{\ast}({\cal I}_{C^{\prime}_u}\cdot{\cal E}^{\prime})$
	and  $\rho_{\ast}({\cal E}^{\prime}|_{C^{\prime}_0})$
    are torsion free on $C$,
one has a sequence of  inclusions of ${\cal O}_C$-modules
 $$
   0\;\hookrightarrow\;
    \rho_{\ast}({\cal I}_{C^{\prime}_u}\cdot{\cal E}^{\prime})\;
   \hookrightarrow\;
{\cal F}/{\cal F}_{\torsionscriptsize}
	\hookrightarrow\;
	\rho_{\ast}({\cal E}^{\prime}|_{C^{\prime}_0})\,.
 $$
From which one deduces that
 $$
    0\; \le \;
	l({\cal F}/
	     ({\cal F}_{\torsionscriptsize}
		      + \rho_{\ast}({\cal I}_{C^{\prime}_u}{\cal E}^{\prime}    )))\;
	 \le\;  l(\rho_{\ast}(
	      {\cal E}^{\prime}|_{C^{\prime}_0})
	         /\rho_{\ast}({\cal I}_{C^{\prime}_u}\cdot{\cal E}^{\prime}))\;
     \le\;  r\,|C^{\prime}_0\cap C^{\prime}_u|	\,,		
 $$
  which is, again, uniformly bounded by Lemma~\ref{bounded-ctC0Cp1u}
                                                             and Lemma~\ref{lFtfbCp0u+}.
           % Lemma [bound on combinatorial types of $C^{\prime}_0$ and $C^{\prime(1)}_u$]
	       % Lemma [length of ${\cal F}_{\torsionscriptsize}$
           %                 and bound on $(C^{\prime(0)}_{u,+},
           %			                            {\cal E}^{\prime}|_{C^{\prime(0)}_{u,+}})$]    		
It follows that
 $$
   \begin{array}{ccl}
    |\,  \chi({\cal E}^{\prime}|_{C^{\prime}_0})\,
	                -\,\chi({\cal F}/{\cal F}_{\torsionscriptsize})\, |
	  &   =
	  &  |\, \chi({\cal I}_{C^{\prime}_u}\cdot{\cal E}^{\prime})\,
	         +\, l({\cal G}^{\prime})
	         -\,\chi({\cal F}/{\cal F}_{\torsionscriptsize})\, |                 \\[1.2ex]
	& =
      & |\,	\chi(\rho_{\ast}({\cal I}_{C^{\prime}_u}\cdot{\cal E}^{\prime}))\,
	      +\,  l({\cal G}^{\prime})
		  -\,\chi({\cal F}/{\cal F}_{\torsionscriptsize}) \, |                   \\[1.2ex]
	&\le
      &  2\,r\,|C^{\prime}_0\cap C^{\prime}_u|	\,.
  \end{array}	
 $$
Together with Lemma~\ref{lFtfbCp0u+} and Lemma~\ref{bounded-ctC0Cp1u},
      % Lemma [bound on combinatorial types of $C^{\prime(1)}_u$]
	  % Lemma [length of ${\cal F}_{\torsionscriptsize}$
      %                 and bound on $(C^{\prime(0)}_{u,+},
      %			                            {\cal E}^{\prime}|_{C^{\prime(0)}_{u,+}})$]
 this proves:
   
\bigskip
 
\begin{lemma}\label{bEulerEpCp0}
  {\bf [bound on $|\chi({\cal E}^{\prime}|_{C^{\prime}_0})|$].}
  There is a uniform upper bound for
     $|\chi({\cal E}^{\prime}|_{C^{\prime}_0})|$
	that depends only on  $(g;r,\chi;\beta,c)$.
\end{lemma}

\bigskip

\noindent
$(b)$
{\it Boundedness for the family of pairs
        $(C^{\prime}_0,{\cal E}^{\prime}|_{C^{\prime}_0})$ }

\medskip

\noindent
From Part (a), one has an exact sequence of ${\cal O}_C$-modules
 $$
    0 \; \longrightarrow\; {\cal F}/{\cal F}_{\torsionscriptsize}\;
	      \longrightarrow\;   \rho_{\ast}({\cal E}^{\prime}|_{C^{\prime}_0})\;
		  \longrightarrow\;  {\cal G}\; \longrightarrow\; 0\,,
 $$
 where ${\cal G}$ is a $0$-dimensional ${\cal O}_C$-module of length
 $\le r\,|C^{\prime}_0\cap C^{\prime}_u|$.
This shows that  the pairs
  $(C, \rho_{\ast}({\cal E}^{\prime}|_{C^{\prime}_0}))$
  lie in a bounded family
 since both ${\cal F}/{\cal F}_{\torsionscriptsize}$ and ${\cal G}$ do.
On the other hand,
  the pairs $(C^{\prime}_0,{\cal E}^{\prime}|_{C^{\prime}_0})$
     lie in a bounded family
 if and only if
  the pairs $(C, \rho_{\ast}({\cal E}^{\prime}|_{C^{\prime}_0}))$
   lie in a bounded family.
This proves:
 
\bigskip

\begin{proposition}\label{bounded-Cp0Ep}
{\bf [boundedness of
            $\{(C^{\prime}_0, {\cal E}^{\prime}|_{C^{\prime}_0})\}$].}
  The family of isomorphisms classes of pairs
   $(C^{\prime }_0, {\cal E}^{\prime}|_{C^{\prime}_0})$
   is bounded.		
\end{proposition}

\bigskip

\begin{flushleft}
{\bf Boundedness of the family of isomorphisms classes of pairs
   $(C^{\prime (1)}_u, {\cal E}^{\prime}|_{C^{\prime (1)}_u})$}
\end{flushleft}
Let
  $C^{\prime(1)}_{u;ij}\simeq {\Bbb P}^1$ be the irreducible components
     of $C^{\prime(1)}_u$,
	   with $i$ labeling the connected components of $C^{\prime(1)}_u$  and
	           $j$ labeling the irreducible components in the $i$-th connected component
    			   of $C^{\prime(1)}_u$,    and 				
  ${\cal E}^{\prime}|_{C^{\prime(1)}_{u;ij}}
       \simeq \oplus_{k=1}^r{\cal O}_{{\Bbb P}^1}(a_{ijk})$.				   	
Recall Lemma~\ref{bounded-ctC0Cp1u}
    % Lemma [boundedness of combinatorial types of $C^{\prime}_0$ and $C^{\prime(1)}_u$]
 that there is a uniform bound on the number of combinatorial types of the ${\Bbb P}^1$-tree
   $C^{\prime(1)}_u$.
It follows that to show that
 the pairs $(C^{\prime(1)}_u,{\cal E}^{\prime}|_{C^{\prime(1)}_u})$
  lie in a bounded family,
we only need to show that the $a_{ijk}$'s that can appear in the above decomposition
   must lie in a bounded interval of ${\Bbb Z}$ that depends only on $(g;r,\chi;\beta,c)$.

\bigskip

\noindent
$(a)$ {\it Upper bound for $a_{ijk}$}
  
\medskip

\noindent
That $a_{ijk}$ has a uniform upper bound follows from the same argument as in
  $\;${\sl Theme `Boundedness of the family of isomorphisms classes of pairs
   $(C^{\prime (0)}_{u,+},
        {\cal E}^{\prime}|_{C^{\prime (0)}_{u,+}})$'}$\;$
  by comparing
    $l((\rho_{\ast}({\cal E}^{\prime}|_{C^{\prime(1)}_u}))
                                   _{\torsionscriptsize})$
  and  $l({\cal F}_{\torsionscriptsize})$.

\bigskip

\noindent
$(b)$ {\it Lower bound for $a_{ijk}$}
  
\medskip

\noindent
Let
  $$
     \iota_0\; :\;  C^{\prime}_0\;\hookrightarrow\; C\,, \hspace{2em}
	 \iota_u^{(0)}\; :\;  C^{\prime(0)}_u\;\hookrightarrow\; C\,, \hspace{2em}
     \iota_u^{(1)}\; :\;  C^{\prime(1)}_u\;\hookrightarrow\; C
  $$
 be the inclusion of subcurves.
Then
  the quotient homomorphisms of ${\cal O}_C$-modules
  $$
     {\cal E}^{\prime}\;\longrightarrow\;
	      {\cal E}^{\prime}|_{C^{\prime}_0}\,,   \hspace{2em}
     {\cal E}^{\prime}\;\longrightarrow\;
      	  {\cal E}^{\prime}|_{C^{\prime(0)}_u}\,,    \hspace{2em}
     {\cal E}^{\prime}\;\longrightarrow\;
          {\cal E}^{\prime}|_{C^{\prime(1)}_u} 		
  $$
  induce a  short exact sequence of ${\cal O}_C$-modules
 $$
   0\; \longrightarrow\;  {\cal E}^{\prime}\; \longrightarrow\;
       {\iota_0}_{\ast}({\cal E}^{\prime}|_{C^{\prime}_0})
	    \oplus   {\iota_u^{(0)}}_{\ast}
		                      ({\cal E}^{\prime}|_{C^{\prime(0)}_u})
	    \oplus   {\iota_u^{(1)}}_{\ast}
		                      ({\cal E}^{\prime}|_{C^{\prime(1)}_u})\; 	    		
        \longrightarrow\; {\cal G}^{\prime}\;\longrightarrow 0\,,							
 $$
 where ${\cal G}^{\prime}$ is $0$-dimensional of length bounded above
 by $\,r\,(|C^{\prime}_0\cap C^{\prime(0)}_u|
                   +|C^{\prime}_0\cap C^{\prime(1)}_u|
				   +|C^{\prime(0)}_u\cap C^{\prime(1)}_u|)$,
 which, in turn, is uniformly bounded above.
It follows that
 $$
   \chi({\cal E}^{\prime}|_{C^{\prime(1)}_u})\;
   =\; \chi({\cal E}^{\prime})\, +\, l({\cal G}^{\prime})\,
          -\, \chi({\cal E}^{\prime}|_{C^{\prime}_0})   \,
		  -\, \chi({\cal E}^{\prime}|_{C^{\prime(0)}_u})
 $$
 is uniformly bounded below
 since
   $\chi({\cal E}^{\prime})=\chi$ is fixed and
   all $l({\cal G}^{\prime})$,
        $-\, \chi({\cal E}^{\prime}|_{C^{\prime}_0})$, and
		$-\, \chi({\cal E}^{\prime}|_{C^{\prime(0)}_u})$
	  are uniformly bounded below
	 by the above argument, Proposition~\ref{bounded-Cp0u+},
	                                           Proposition~\ref{bounded-Cp0Ep}, and
											   Proposition~\ref{bounded-Cp0u0}.
	      %
	      % Proposition [boundedness of
          %                        $\{(C^{\prime(0)}_{u,+},
	      %                               {\cal E}^{\prime}|_{C^{\prime(0)}_{u,+}})\}$]	
          % Proposition [boundedness of
          %                       $\{(C^{\prime}_0, {\cal E}^{\prime}|_{C^{\prime}_0})\}$]
	      % Proposition [boundedness of
          %                        $\{(C^{\prime(0)}_{u,0},
	      %                               {\cal E}^{\prime}|_{C^{\prime(0)}_{u,0}})\}$]		
	
On the other hand,  one has the short exact sequence
 $$
   0\;\longrightarrow\;
       {\cal E}^{\prime}|_{C^{\prime(1)}_u}\;  \longrightarrow\;
        \oplus_{i,j}{\cal E}^{\prime}|_{C^{\prime(1)}_{u;ij}}\;
	   \longrightarrow\; {\cal G}^{\prime\prime}\;           \longrightarrow\; 0\,,	
 $$
  where ${\cal G}^{\prime\prime}$ is $0$-dimensional of length bounded above
  by $\,r\,|(C^{\prime(1)}_u)_{\singularscriptsize}|$,
  which, in turn, is uniformly bounded.
It follows that
 $$
   \sum_{i,j,k}(1+a_{ijk})\;
    =\; \sum_{i,j}
	          \chi(\oplus_{i,j}{\cal E}^{\prime}|_{C^{\prime(1)}_{u;ij}})\,
    =\; \chi({\cal E}^{\prime}|_{C^{\prime(1)}_u})\,			
		   +\, l({\cal G}^{\prime\prime})\;	  	
	\ge\;  	  \chi({\cal E}^{\prime}|_{C^{\prime(1)}_u})
 $$
 is uniformly bounded below.
Since
  the number of terms in the summation $\sum_{i,j,k}$ is bounded above by
    $r\,(H\cdot \beta)$  and
  the $a_{ijk}$'s are uniformly bounded above by Part (a),
 $a_{ijk}$ must be uniformly bounded below as well.
Together with Part (a), this proves that:

\bigskip

\begin{proposition}\label{bounded-Cp1uEp}
  {\bf [boundedness of
              ${(C^{\prime(1)}_u,{\cal E}^{\prime}|_{C^{\prime(1)}_u})}$]}
     The family of isomorphisms classes of pairs
        $(C^{\prime (1)}_u, {\cal E}^{\prime}|_{C^{\prime (1)}_u})$
        is bounded.
\end{proposition}

\bigskip
 
All together, this proves the boundedness of $\Image\Phi$.
It follows now from Proposition~\ref{bounded-mbfd} that
                             % Proposition [boundedness of morphisms from bounded family of domains]						 
 ${\frak M}_{Az^{\!f}\!(g;r,\chi)}^{\scriptsizeZss}(Y;\beta,c)$
   is also bounded.
This concludes the proof of Theorem~\ref{bounded-sZssm}.
                      % Theorem [boundedness of stack of $Z$-semistable morphisms of fixed type]

\bigskip

\subsection{The error charge of a Fourier-Mukai transform.}
		
Similar to the technique used in the work [L-W] of Jun Li and Baosen Wu,
 we define the notion of error charge that suits our problem in this subsection.
It will be used in Sec.~4.3
  to show that the bubbling-off procedure, which resolves irregularities of  a morphism
   created by filling-in the central fiber of a $1$-dimensional family of morphisms,
 must terminate in finitely many steps to recover a regular morphism within our category.
This proves then the completeness of the stack
 ${\frak M}_{Az^{\!f}\!(g;r,\chi)}^{\scriptsizeZss}(Y;\beta,c)$.

\bigskip
 
\begin{flushleft}
{\bf The error charge and its basic properties}
\end{flushleft}		
\begin{definition}\label{error-charge}
 {\bf [error charge].} {\rm		
 Let
    $C^{\prime}$ be a nodal curve,
    $\rho:C^{\prime}\rightarrow C\in \overline{\cal M}_g$ be the collapsing morphism,
	 and
	$\tilde{\cal F}^{\prime}$
    	be a $1$-dimensional coherent ${\cal O}_{C^{\prime}\times Y}$-module
	such that ${\pr_{C^{\prime}}}_{\ast}(\tilde{\cal F}^{\prime})$
      has rank $r$ on each irreducible component of $C^{\prime}$.			
 Then, from the existence of a flattening stratification for a coherent sheaf,
   there exist a finite set $D^{\prime}$ of points on $C^{\prime}$
    such that
    \begin{itemize}
	   \item[{\Large $\cdot$}]
	     $\tilde{\cal F}^{\prime}$ is flat, of relative length $r$,
		  over $U^{\prime}:= C^{\prime}-D^{\prime}$;
	
	   \item[{\Large $\cdot$}]
	     as coherent sheaf over $C^{\prime}$,
    		 $\tilde{\cal F}^{\prime}$ is not flat over points in $D^{\prime}$.
	\end{itemize}
 Define the {\it error charge of $\tilde{\cal F}^{\prime}$ over $C^{\prime}$,
    with respect to the central charge function $Z$},
   to be the following complex number
   $\Err^Z_{C^{\prime}}(\tilde{\cal F}^{\prime})$ :
   \begin{itemize}
     \item[{\Large $\cdot$}]
	  For $p^{\prime}\in C^{\prime}$,
	  let
	    $$
		   \tilde{\cal F}^{\prime}_{(p^{\prime})}\;
		    =\;  \{s^{\prime}\in \tilde{\cal F}^{\prime}\;|\;
			               \Ann(v)\subset {\cal I}_{p^{\prime}}^k\,,\;
						    \mbox{for some $k\ge 1$}
			            \}\,.
		$$
	  Define
       $$
         \tilde{\cal F}^{\prime}_{\torsionscriptsize/C^{\prime}}\;
	        :=\;  \sum_{p^{\prime}\in C^{\prime}}
	           \tilde{\cal F}^{\prime}_{(p^{\prime})}
       $$ 	
	  and
       $$
         \tilde{\cal F}^{\prime}_{\torsionfreescriptsize/C^{\prime}}\;
		    :=\;  \tilde{\cal F}^{\prime}/
			            \tilde{\cal F}^{\prime}_{\torsionscriptsize/C^{\prime}}\,. 		
       $$
  	  Note that
	    $\tilde{\cal F}^{\prime}_{(p^{\prime})}=0$
       		except for $p^{\prime}\in D^{\prime}$
			(hence $\sum_{p^{\prime}\in C^{\prime}}$ is well-defined)
	    and that
    	 $\tilde{\cal F}^{\prime}_{\torsionfreescriptsize/C^{\prime}}$
	    is now of relative dimension $0$ over $C^{\prime}$;
      points in $C^{\prime}$ over which
	    $\tilde{\cal F}^{\prime}_{\torsionfreescriptsize/C^{\prime}}$
	    is not flat is now a subset of $D^{\prime}\cap C^{\prime}_{\singularscriptsize}$.
		
	 \item[{\Large $\cdot$}]
	 (Cf.\ Definition~\ref{dtf})\\
	        % Definition [discrepancy to flatness]
     For $p^{\prime}\in C^{\prime}$,
	 define the {\it discrepancy to flatness
	 of $\tilde{\cal F}^{\prime}_{\torsionfreescriptsize/C^{\prime}}$
	 over $p^{\prime}$} to be
     $$
	  \begin{array}{rrl}
        \delta_{\flatscriptsize/C^{\prime}}
	     (\tilde{\cal F}^{\prime}_{\torsionfreescriptsize/C^{\prime}};p^{\prime})
        &  :=
        &	 l (\tilde{\cal F}^{\prime}_{\torsionfreescriptsize/C^{\prime}}
				       |_{p^{\prime}})\,
		            -\, r                                                                                                            \\[1.2ex]
	  & =
        & \delta_{\flatscriptsize}
		           ({\pr_{C^{\prime}}}_{\ast}
		                  (\tilde{\cal F}^{\prime}_{\torsionfreescriptsize/C^{\prime}}),
                             p^{\prime})\,.
      \end{array}								
     $$ 	
	 The last equality follows from the fact that
	    $\Supp(\tilde{\cal F}^{\prime}_{\torsionfreescriptsize/C^{\prime}})$
       is affine,  of relative dimension $0$	 over $C^{\prime}$.
	 Also, note that
	  $\,\delta_{\flatscriptsize/C^{\prime}}
	     (\tilde{\cal F}^{\prime}_{\torsionfreescriptsize/C^{\prime}};p^{\prime})
		=0\,$
		except for a subset of points
		in $D^{\prime}\cap C^{\prime}_{\singularscriptsize}$.
	 Thus, we may define the {\it  discrepancy to flatness
 	   of $\tilde{\cal F}^{\prime}_{\torsionfreescriptsize/C^{\prime}}$
	   over $C^{\prime}$} to be
	  $$
        \delta_{\flatscriptsize/C^{\prime}}
		   (\tilde{\cal F}^{\prime}_{\torsionfreescriptsize/C^{\prime}})\;
		:=\; \sum_{p^{\prime}\in C^{\prime}}
		        \delta_{\flatscriptsize/C^{\prime}}
	          (\tilde{\cal F}^{\prime}
			                                _{\torsionfreescriptsize/C^{\prime}};p^{\prime})\,.	
      $$	
		
     \item[\Large $\cdot$]	
	 Define the error charge $\Err^Z_{C^{\prime}}(\tilde{\cal F}^{\prime})$
       of $\tilde{\cal F}^{\prime}/C^{\prime}$ to be
      $$
        \Err^Z_{C^{\prime}}(\tilde{\cal F}^{\prime})\;
		:=\;   Z \left( (\rho\times \Id_Y)_{\ast}(
			                   \tilde{\cal F}^{\prime}_{\torsionscriptsize/C^{\prime}}
			                                                                    )      \right)\,
		       +\, \delta_{\flatscriptsize/C^{\prime}}
	                (\tilde{\cal F}^{\prime}_{\torsionfreescriptsize/C^{\prime}})\,.
      $$	
	 Note that
      $$
          \Err^Z_{C^{\prime}}(\tilde{\cal F}^{\prime})\;
	      =\;  \Err^Z_{C^{\prime}}
		               (\tilde{\cal F}^{\prime}_{\torsionscriptsize/C^{\prime}})
                +\, \Err^Z_{C^{\prime}}
	                   (\tilde{\cal F}^{\prime}_{\torsionfreescriptsize/C^{\prime}})\,.
      $$							
   \end{itemize}
}\end{definition}

\bigskip

\begin{lemma}\label{vecacfCp}
{\bf [vanishing of error charge as criterion of flatness$\,/C^{\prime}$].}
 Continuing the set-up in Definition~\ref{error-charge}.
                                       % Definition [error charge]
 Then,
    $\tilde{\cal F}^{\prime}$ is flat over $C^{\prime}$
    if and only if  $\Err^Z_{C^{\prime}}(\tilde{\cal F}^{\prime})=0\,$.
\end{lemma}

\begin{proof}
  Note that the subsheaf $\tilde{\cal F}^{\prime}_{(p^{\prime})}$
       of $\tilde{\cal F}^{\prime}$ in Definition~\ref{error-charge}
                                                                   % Definition [error charge]
   can be defined for all $p^{\prime}\in C^{\prime}$.
  With this, the lemma is an immediate consequence of the fact
    that $\tilde{\cal F}^{\prime}$ is flat over $C^{\prime}$
	   if and only if
	     $\tilde{\cal F}^{\prime}_{(p^{\prime})}=0$
		 and $l(\tilde{\cal F}^{\prime}|_{p^{\prime}})=r$
	  for all $p^{\prime}\in C^{\prime}$,
	that $\tilde{\cal F}^{\prime}_{(p^{\prime})}$
            	is of relative dimension $0$ over $C\times Y$   and hence
            $(\rho\times\Id_Y)_{\ast}
			       (\tilde{\cal F}^{\prime}_{(p^{\prime})})=0$
		       if and only if
			     $\tilde{\cal F}^{\prime}_{(p^{\prime})}=0$,
	  and
    that $Z(\,\bullet\,)=0$ if and only if $\bullet=0$.

\end{proof}

\bigskip

The following lemma follows from Lemma~\ref{tccef} ([L-Y3: Lemma~2.1.2] (D(10.1)):
                                                      % Lemma [twisted central charge: explicit form]

\bigskip

\begin{lemma}\label{definity-ec}
  {\bf [definity of $\Err^Z_{C^{\prime}}(\tilde{\cal F}^{\prime})$].}
  Continuing the set-up in Definition~\ref{error-charge}.
                                        % Definition [error charge]
  Then:
     \begin{itemize}
       \item[(1)]
	    As a function on $\CohCategory_1(C^{\prime}\times Y)$,
		 $\Err^Z_{C^{\prime}}$ takes its value
		 in a locally-finite rank-$2$ lattice in ${\Bbb C}$.

       \item[(2)]
        $-\,\Imaginary\Err^Z_{C^{\prime}}(\tilde{\cal F}^{\prime})\; \ge\; 0\,$.
		
	   \item[(3)]
        $-\,\Imaginary\Err^Z_{C^{\prime}}(\tilde{\cal F}^{\prime})\; =\; 0\,$
	    if and only
	                $\tilde{\cal F}^{\prime}_{(p^{\prime})}$ is $0$-dimensional
    	               for all $p^{\prime}\in C^{\prime}$.

       \item[(4)]					
        When $-\Imaginary\Err^Z_{C^{\prime}}
                   (\tilde{\cal F}^{\prime}_{(p^{\prime})})=0$,
               $\Err^Z_{C^{\prime}}(\tilde{\cal F}^{\prime})\in {\Bbb Z}_{\ge 0}$.
    \end{itemize}			
\end{lemma}

\bigskip

\begin{flushleft}
{\bf Decrease of the error charge under a bubbling-off of $C^{\prime}$}
\end{flushleft}
\begin{definition}\label{order-C}
 {\bf [order on ${\Bbb C}$].} {\rm
  Define an {\it order  $\prec$} on ${\Bbb C}$ as follows:
  For $z_1$, $z_2\in {\Bbb C}$, we say that
   {\it $z_1$ precedes $z_2$},
    in notation $z_1\prec z_2$ (or equivalently $z_2\succ z_1$),
   if either $-\Imaginary z_1 < -\Imaginary z_2$
     or $-\Imaginary z_1= -\Imaginary z_2$ and $\Real z_1 < \Real z_2$.
  When either $z_1\prec z_2$ or $z_1=z_2$ holds,
     one denotes $z_1\preccurlyeq z_2$ (or equivalently $z_2\succcurlyeq z_1$).
}\end{definition}

\bigskip

\begin{proposition}\label{decrease-ec}
 {\bf [decrease of $\Err^Z_{\bullet}(\,\,\bullet)$ under bubbling off].}
 Continuing the situation in Definition~\ref{error-charge}.
                                          % Definition [error charge]
 Let
   $C^{\prime\prime}$ be a nodal curve that admits  a collapsing morphism
   $$
     h\;  :\;  C^{\prime\prime}\longrightarrow\; C^{\prime}
   $$
   that contracts a connected ${\Bbb P}^1$-tree subsurve $C^{\prime\prime}_{u_h}$
   in $C^{\prime\prime}$.
 Note the commutative diagram
   $$
     \xymatrix{
      C^{\prime\prime}\ar[rr]^-{h}\ar[rd]_-{\rho^{\prime\prime}}
      && C^{\prime}\ar[ld]^{\rho^{\prime}}\\
     & C 		
     }
   $$	
    where $\rho^{\prime}$ and $\rho^{\prime\prime}$ are collapsing morphisms
    that stabilize the nodal curves.	
  Let $\tilde{\cal F}^{\prime\prime}\in \CohCategory_1(C^{\prime\prime})$
   be a $1$-dimensional coherent sheaf on $C^{\prime\prime}\times Y$
   with the following properties:
   \begin{itemize}
    \item[(0)]
     ${\pr_{C^{\prime\prime}}}_{\ast}(\tilde{\cal F}^{\prime\prime})$
	 is of rank $r$ on each irreducible component of $C^{\prime\prime}$.
	
    \item[(1)]
     $(h\times \Id_Y)_{\ast}(\tilde{\cal F}^{\prime\prime})
	    \simeq   \tilde{\cal F}^{\prime}$.
		
    \item[(2)]
    The natural sequence of homomorphisms of
     ${\cal O}_{C^{\prime\prime}\times Y}$-modules
     $(h\times\Id_Y)^{\ast}(\tilde{\cal F}^{\prime})
		       \rightarrow\tilde{\cal F}^{\prime\prime}\rightarrow 0$
      is exact.		
		
     \item[(3)]
      For at least one of the ${\Bbb P}^1$-components of the ${\Bbb P}^1$-tree subcurve
	   of $C^{\prime\prime}$ that is collapsed by $h$ to points in $C^{\prime}$,
       $(\tilde{\cal F}^{\prime\prime}|_{{\Bbb P}^1})
		     _{\torsionfreescriptsize/{\Bbb P}^1}$	
      is not the pull-back of a coherent sheaf on $C^{\prime}\times Y$ via $h\times \Id_Y$.		
   \end{itemize}	
 Then,
  $$
    \Err^Z_{C^{\prime\prime}}(\tilde{\cal F}^{\prime\prime})\;
	   \prec\;     \Err^Z_{C^{\prime}}(\tilde{\cal F}^{\prime})\,.
  $$
\end{proposition}

\bigskip

The proof takes the rest of this theme, which we now proceed.

Note that under the isomorphism
   $(h\times \Id_Y)_{\ast}(\tilde{\cal F}^{\prime\prime})
       \simeq \tilde{\cal F}^{\prime}$ from Property (1) of the statement,
   $$
      (h\times \Id_Y)_{\ast}
       (\tilde{\cal F}^{\prime\prime}_{\torsionscriptsize/C^{\prime\prime}})\;
        \subset\;  \tilde{\cal F}^{\prime}_{\torsionscriptsize/C^{\prime}}\,.
   $$
In particular,
     the curve class
    $\, [\tilde{\cal F}^{\prime}_{\torsionscriptsize/C^{\prime}}]\,
	    -\,  [(h\times\Id_Y)_{\ast}
	           (\tilde{\cal F}^{\prime\prime}_{\torsionscriptsize/C^{\prime\prime}}) ]\,$
	on $C^{\prime}\times Y$ is effective if it is not zero.		
Consider the built-in short exact sequence of ${\cal O}_{C^{\prime\prime}}$-modules
 $$
   0\; \longrightarrow\;
        \tilde{\cal F}^{\prime\prime}_{\torsionscriptsize/C^{\prime\prime}}\;
		  \longrightarrow\;
        \tilde{\cal F}^{\prime\prime}\; \longrightarrow\;
		\tilde{\cal F}^{\prime\prime}_{\torsionfreescriptsize/C^{\prime\prime}}\;
		 \longrightarrow\; 0\,.
 $$
Since
 $\tilde{\cal F}^{\prime\prime}_{\torsionscriptsize/C^{\prime\prime}}$
   has relative dimension $0$ over $C^{\prime}\times Y$ under $h\times\Id_Y$,
  $R^1\!(h\times\Id_Y)_{\ast}(
       \tilde{\cal F}^{\prime\prime}_{\torsionscriptsize/C^{\prime\prime}})=0$.
It follows that one has the following commutative diagram of
  ${\cal O}_{C^{\prime}\times Y}$-modules
  $$
   \xymatrix{
   0 \ar[r]
       & (h\times\Id_Y )_{\ast}
     	    (\tilde{\cal F}^{\prime\prime}_{\torsionscriptsize/C^{\prime\prime}})
			\ar[r]
			\ar@{^{(}->}[d]^-{\iota}
       & (h\times\Id_Y)_{\ast}(\tilde{\cal F}^{\prime\prime})
		   \ar[r]
		   \ar[d]^-{\simeq}
	   & (h\times\Id_Y)_{\ast}
	           (\tilde{\cal F}^{\prime\prime}_{\torsionfreescriptsize/C^{\prime\prime}})
		   \ar[r]
           \ar@{->>}[d]^-{\alpha}		   & 0\\
   0 \ar[r]
       & \tilde{\cal F}^{\prime}_{\torsionscriptsize/C^{\prime}}\ar[r]
       & \tilde{\cal F}^{\prime}\ar[r]
	   & \tilde{\cal F}^{\prime}_{\torsionfreescriptsize/C^{\prime}}\ar[r]      & 0
	}	
 $$
 where both horizontal sequences are exact. 	
This implies that
  $$
     \Ker(\alpha)\;\simeq\; \Coker(\iota)\,.
  $$
$\alpha$ induces  a homomorphism of ${\cal O}_{C^{\prime}}$-modules
 $$
   {\pr_{C^{\prime}}}_{\ast}(\alpha)\;:\;
        {\pr_{C^{\prime}}}_{\ast}
		   ( (h\times\Id_Y)_{\ast}
                    (\tilde{\cal F}^{\prime\prime}
					                  _{\torsionfreescriptsize/C^{\prime\prime}}))\;
      \longrightarrow\; 									  		
    {\pr_{C^{\prime}}}_{\ast}
	   (\tilde{\cal F}^{\prime}_{\torsionfreescriptsize/C^{\prime}})\,.
 $$
Since $  \pr_{C^{\prime}}\circ(h\times \Id_Y)= h\circ \pr_{C^{\prime\prime}}$,
  $\; {\pr_{C^{\prime}}}_{\ast}(\alpha)$, in turn,
  defines a homomorphism
 $$
    \alpha^{\prime}\;:\;
     h_{\ast}( {\pr_{C^{\prime\prime}}}_{\ast}
   (\tilde{\cal F}^{\prime\prime}_{\torsionfreescriptsize/C^{\prime\prime}}))\;
    \longrightarrow\;
  {\pr_{C^{\prime}}}_{\ast}
	    (\tilde{\cal F}^{\prime}_{\torsionfreescriptsize/C^{\prime}})\,.
 $$
%
%%%%%%%%%%%%%%%%%%%%%%%%%%%%%%%%%%%%
% Extend the above diagram of ${\cal O}_{C^{\prime}\times Y}$-modules to
% $$
%   \xymatrix{
%    &&&  0 \ar[d]    \\
%    & 0  \ar[d]  & 0  \ar[d]  & \Ker(\alpha)\ar[d]    \\
%    0 \ar[r]
%       & (h\times\Id_Y )_{\ast}
%     	    (\tilde{\cal F}^{\prime\prime}_{\torsionscriptsize/C^{\prime\prime}})
%			\ar[r]
%			\ar@{^{(}->}[d]^-{\iota}
%       & (h\times\Id_Y)_{\ast}(\tilde{\cal F}^{\prime\prime})
%		   \ar[r]
%		   \ar[d]^-{\simeq}
%	   & (h\times\Id_Y)_{\ast}
%	           (\tilde{\cal F}^{\prime\prime}_{\torsionfreescriptsize/C^{\prime\prime}})
%		   \ar[r]
%           \ar@{->>}[d]^-{\alpha}		   & 0\\
%   0 \ar[r]
%       & \tilde{\cal F}^{\prime}_{\torsionscriptsize/C^{\prime}}\ar[r] \ar[d]
%       & \tilde{\cal F}^{\prime}\ar[r] \ar[d]
%	   & \tilde{\cal F}^{\prime}_{\torsionfreescriptsize/C^{\prime}}\ar[r]  \ar[d]   & 0  \\
%   & \Coker(\iota)\ar[d]  & 0  & 0   \\	
%   & 0
%	}	
% $$
% where all the horizontal and vertical sequences are exact.
%%%%%%%%%%%%%%%%%%%%%%%%%%%%%%%%%%%%%
%
Note that $\Ker(\alpha)$ is either $1$-dimensional or $0$-dimensional.
For the convenience of further discussions, let
 $$
    C^{\prime\prime}\; =\;  C^{\prime\prime}_{0_h} \cup C^{\prime\prime}_{u_h}
 $$
 be the decomposition of $C^{\prime\prime}$
  as a union of the ${\Bbb P}^1$-tree subcurve $C^{\prime\prime}_{u_h}$
     that is contracted by $h$ and the remaining subcurve $C^{\prime\prime}_{0_h}$.

\bigskip

\noindent
{\it Case $(a):$   $\;\Ker(\alpha)$ is $1$-dimensional.}
 
\medskip

\noindent
This happens exactly when
  there exists a ${\Bbb P}^1$-component of $C^{\prime\prime}_{u_h}$
   such that $\Supp(\tilde{\cal F}^{\prime\prime}|_{{\Bbb P}^1})$
   has an irreducible component that is of relative dimension $1$
   over both ${\Bbb P}^1$ and $Y$.
In this case,
   $[\tilde{\cal F}^{\prime}_{\torsionscriptsize/C^{\prime}}]
  -  [(h\times\Id_Y)_{\ast}
             (\tilde{\cal F}^{\prime\prime}_{\torsionscriptsize/C^{\prime\prime}}) ]$
   is effective.
Thus,
   $$
      -\Imaginary\Err^Z_{C^{\prime\prime}}(\tilde{\cal F}^{\prime\prime})\;
    <\;  -\,\Imaginary\Err^Z_{C^{\prime}}(\tilde{\cal F}^{\prime})
   $$
  and, hence,
  $\Err^Z_{C^{\prime\prime}}(\tilde{\cal F}^{\prime\prime})
      \prec \Err^Z_{C^{\prime}}(\tilde{\cal F}^{\prime})$,
   as claimed.

\bigskip

\noindent
{\it Case $(b):$ $\;\Ker(\alpha)$ is $0$-dimensional.}

\medskip

\noindent
I.e., there exists no ${\Bbb P}^1$-component of $C^{\prime\prime}_{u_h}$
   such that $\Supp(\tilde{\cal F}^{\prime\prime}|_{{\Bbb P}^1})$
   has an irreducible component that is of relative dimension $1$
   over both ${\Bbb P}^1$ and $Y$.
Observe first the following lemma:

\bigskip

\begin{lemma}\label{positivity-CppFpp}
 {\bf [positivity of
     ${\pr_{C^{\prime\prime}}}_{\ast}
         (\tilde{\cal F}^{\prime\prime}
		                                           _{\torsionfreescriptsize/C^{\prime\prime}})$
	  on $C^{\prime\prime}_{u_h}$].}
 Let
   $$
    {\cal G}^{\prime\prime}\;
	   := \;
	   	   \left({\pr_{C^{\prime\prime}}}_{\ast}
                          (\tilde{\cal F}^{\prime\prime}
		                                           _{\torsionfreescriptsize/C^{\prime\prime}}))
												   |_{C^{\prime\prime}_{u_h}}
           \right)_{\torsionfreescriptsize}\,.												
   $$												
 Then, when $\Ker(\alpha)$ is $0$-dimensional,
   Properties (2) and (3) in the statement of Proposition~\ref{decrease-ec} imply that
                   % Proposition [decrease of $\Err^Z_{\bullet}(\,\,\bullet)$ under bubbling off]
   ${\cal G}^{\prime\prime}$	is a positive torsion-free sheaf
    on the ${\Bbb P}^1$-tree $C^{\prime\prime}_{u_h}$.
\end{lemma}
  
\begin{proof}
 To show that ${\cal G}^{\prime\prime}$ is positive is to show that
  for every ${\Bbb P}^1$-component ${\Bbb P}^1$ of $C^{\prime\prime}_{u_h}$,\\
    $({\cal G}^{\prime\prime}|_{{\Bbb P}^1})_{\torsionfreescriptsize}$
     is nonnegative and
  for the ${\Bbb P}^1$-components of $C^{\prime\prime}_{u_h}$ in Property (3),\\
   $({\cal G}^{\prime\prime}|_{{\Bbb P}^1})_{\torsionfreescriptsize}$
    is positive.
 In this case,
   $$
   ({\cal G}^{\prime\prime}|_{{\Bbb P}^1})_{\torsionfreescriptsize}\;
	 =\; \pr_{C^{\prime\prime}} \left(
	    (\tilde{\cal F}^{\prime\prime}|_{{\Bbb P}^1})
		                               _{\torsionfreescriptsize/C^{\prime\prime}}
                                                                  \right)
   $$
   is locally free   and
   $(\tilde{\cal F}^{\prime\prime}|_{{\Bbb P}^1})
		                               _{\torsionfreescriptsize/C^{\prime\prime}}
						\in \CohCategory_1({\Bbb P}^1\times Y)$
    defines a special morphism as studied in Sec.~2.2.
 The same argument that gives Lemma~\ref{ggEpdE}
    % Lemma [global generation of ${\cal E}$ and positivity of $\degree({\cal E})$ ]
   proves then
    $({\cal G}^{\prime\prime}|_{{\Bbb P}^1})_{\torsionfreescriptsize}$
	is positive.
 This proves the lemma.
  
\end{proof}

\bigskip
      
In the current case,
     $R^1\!{\pr_{C^{\prime}}}_{\ast}(\Ker(\alpha))=0$;
  together with left exactness of push-forward, one has thus:	
  \begin{itemize}
   \item[{\Large $\cdot$}]  {\it
     When $\Ker(\alpha)$ is $0$-dimensional,
	 the homomorphisms
     ${\pr_{C^{\prime}}}_{\ast}(\alpha)$ and, hence,
	 $$ 				
	    \alpha^{\prime}\;:\;
          h_{\ast}( {\pr_{C^{\prime\prime}}}_{\ast}
         (\tilde{\cal F}^{\prime\prime}
		                                           _{\torsionfreescriptsize/C^{\prime\prime}}))\;
         \longrightarrow\;
        {\pr_{C^{\prime}}}_{\ast}
	    (\tilde{\cal F}^{\prime}_{\torsionfreescriptsize/C^{\prime}})
     $$
    are surjective.
  Furthermore,
    $\Ker(\alpha^{\prime})= {\pr_{C^{\prime}}}_{\ast}(\Ker(\alpha))$.
	}
  \end{itemize}
Since
  ${\pr_{C^{\prime}}}_{\ast}
	     (\tilde{\cal F}^{\prime}_{\torsionfreescriptsize/C^{\prime}})$
    is a torsion-free ${\cal O}_{C^{\prime}}$-module
      and
 both
   $h_{\ast}( {\pr_{C^{\prime\prime}}}_{\ast}
         (\tilde{\cal F}^{\prime\prime}
		                                           _{\torsionfreescriptsize/C^{\prime\prime}}))$
       and
   ${\pr_{C^{\prime}}}_{\ast}
	    (\tilde{\cal F}^{\prime}_{\torsionfreescriptsize/C^{\prime}})$
	 are of rank $r$,
  $$
    \Ker(\alpha^{\prime})\;=\;
	  \left( h_{\ast}( {\pr_{C^{\prime\prime}}}_{\ast}
         (\tilde{\cal F}^{\prime\prime}
		                                           _{\torsionfreescriptsize/C^{\prime\prime}}))
     \right)_{\torsionscriptsize}\,.
  $$

\bigskip

\begin{lemma}\label{difference-ec}
 {\bf [difference of error charges].}
 When $\Ker(\alpha)$ is $0$-dimensional,
 (with the notation from Lemma~\ref{positivity-CppFpp})
                          % Lemma [positivity of
                          %                 ${\pr_{C^{\prime\prime}}}_{\ast}
                          %                      (\tilde{\cal F}^{\prime\prime}
                          %		                                          _{\torsionfreescriptsize/C^{\prime\prime}})$
	                      %                 on $C^{\prime\prime}_{u_h}$]
  $$
    \Err^Z_{C^{\prime\prime}}(\tilde{\cal F}^{\prime\prime})\,
	 -\,  \Err^Z_{C^{\prime}}(\tilde{\cal F}^{\prime})\;
	 =\; -\,\chi(\Ker(\alpha))\,-\, \degree({\cal G}^{\prime\prime})
	     \;  \in\;  {\Bbb Z}_{<0}\,.
  $$
\end{lemma}

\begin{proof}
 When $\Ker(\alpha)$ is $0$-dimensional,
  $[ (\rho^{\prime\prime}\times \Id_Y)_{\ast}
      (\tilde{\cal F}^{\prime\prime}_{\torsionscriptsize/C^{\prime\prime}}) ]
	  = [ (\rho^{\prime}\times \Id_Y)_{\ast}
      (\tilde{\cal F}^{\prime}_{\torsionscriptsize/C^{\prime}}) ]$
	 in $A_1(C\times Y)$.
 Thus,
   $$
     \begin{array}{l}
        \Err^Z_{C^{\prime\prime}}(\tilde{\cal F}^{\prime\prime})\,
	           -\,  \Err^Z_{C^{\prime}}(\tilde{\cal F}^{\prime}) \\[1.2ex]
	     \hspace{2em}
		   =\; \chi((\rho^{\prime\prime}\times \Id_Y)_{\ast}
                  (\tilde{\cal F}^{\prime\prime}_{\torsionscriptsize/C^{\prime\prime}}))\,
		          -\, \chi((\rho^{\prime}\times \Id_Y)_{\ast}
               (\tilde{\cal F}^{\prime}_{\torsionscriptsize/C^{\prime}}))
			   			                                                                                                                                \\[1.2ex]
         \hspace{3.4em}
		  +\, \delta_{\flatscriptsize/C^{\prime\prime}}
            (\tilde{\cal F}^{\prime\prime}_{\torsionfreescriptsize/C^{\prime\prime}})\,
		  -\, \delta_{\flatscriptsize/C^{\prime}}
           (\tilde{\cal F}^{\prime}_{\torsionfreescriptsize/C^{\prime}})
		      		   \hspace{3em}.
    \end{array}			 			
  $$
 Now,
  $$
   \begin{array}{l}
  	  \chi((\rho^{\prime\prime}\times \Id_Y)_{\ast}
                  (\tilde{\cal F}^{\prime\prime}_{\torsionscriptsize/C^{\prime\prime}}))\,
		          -\, \chi((\rho^{\prime}\times \Id_Y)_{\ast}
               (\tilde{\cal F}^{\prime}_{\torsionscriptsize/C^{\prime}}))
			                                                                                                                       \\[1.2ex]
      \hspace{2em}																																						
	  =\; \chi(\tilde{\cal F}^{\prime\prime}_{\torsionscriptsize/C^{\prime\prime}})\,
      	       -\, \chi(\tilde{\cal F}^{\prime}
		                                                   _{\torsionscriptsize/C^{\prime}})\\[1.2ex]
      \hspace{2em}														
	  =\;   \chi((h\times\Id_Y)_{\ast}
			    (\tilde{\cal F}^{\prime\prime}_{\torsionscriptsize/C^{\prime\prime}}))\,
		   -\, \chi(\tilde{\cal F}^{\prime}_{\torsionscriptsize/C^{\prime}})\,
	  	                                                                                                                           \\[1.2ex]
      \hspace{2em}
      =\;  -\, \chi(\Coker(\iota))\; =\; -\, \chi(\Ker(\alpha))\;
	  =\;  -\, \chi(\Ker(\alpha^{\prime}))   \\[1.2ex]
	  \hspace{2em}
	  =\; -\, \chi\left(
	      	      \left( h_{\ast}( {\pr_{C^{\prime\prime}}}_{\ast}
                             (\tilde{\cal F}^{\prime\prime}
		                                           _{\torsionfreescriptsize/C^{\prime\prime}}))
             \right)_{\torsionscriptsize}
	               \right)\;\;\;\;\le\;0\,,
    \end{array}																																						
   $$
 while
   $$
    \begin{array}{l}
      \delta_{\flatscriptsize/C^{\prime\prime}}
            (\tilde{\cal F}^{\prime\prime}_{\torsionfreescriptsize/C^{\prime\prime}})\,
	    -\, \delta_{\flatscriptsize/C^{\prime}}
           (\tilde{\cal F}^{\prime}_{\torsionfreescriptsize/C^{\prime}})
                         		                                                                   \\[1.2ex]
      \hspace{2em}
      =\; \delta_{\flatscriptsize}(   {\pr_{C^{\prime\prime}}}_{\ast}
               (\tilde{\cal F}^{\prime\prime}
			                                 _{\torsionfreescriptsize/C^{\prime\prime}})	  )\,
		   -\,   \delta_{\flatscriptsize}(   {\pr_{C^{\prime}}}_{\ast}
                (\tilde{\cal F}^{\prime}
				               _{\torsionfreescriptsize/C^{\prime}})	  )\\[1.2ex]
      \hspace{2em}
      =\;   -\, \degree({\cal G}^{\prime\prime})\;\;\;\; <\; 0\,.
	\end{array}
   $$
   Here, the last line follows Proposition~\ref{ddwbfTt}  and
                                               Lemma~\ref{positivity-CppFpp}.
       %
       % Proposition [decrease of $\delta_{\flatscriptsize}(\,\bullet\,)$ when bubbling off
	   %                       ${\Bbb P}^1$-tree with positive sheaf]
       % Lemma       [positivity of
       %                       ${\pr_{C^{\prime\prime}}}_{\ast}
       %                           (\tilde{\cal F}^{\prime\prime}
       %		                                                _{\torsionfreescriptsize/C^{\prime\prime}})$
	   %                       on $C^{\prime\prime}_{u_h}$]    	
	   %
   This proves the lemma.	
    
\end{proof}

\bigskip

It follows that,
when $\Ker(\alpha)$ is $0$-dimensional,
 $\Err^Z_{C^{\prime\prime}}(\tilde{\cal F}^{\prime\prime})
     \prec \Err^Z_{C^{\prime}}(\tilde{\cal F}^{\prime})$ as claimed as well.
   
This proves Proposition~\ref{decrease-ec}.
                 % Proposition [decrease of $\Err^Z_{\bullet}(\,\,\bullet)$ under bubbling off]
	
% \noindent\hspace{40.76em}$\square$

\bigskip

\subsection{Completeness of
        ${\frak M}_{Az^{\!f}\!(g;r,\chi)}^{\scriptsizeZss}(Y;\beta,c)$}
		
We now proceed to prove the  completeness of
           ${\frak M}_{Az^{\!f}\!(g;r,\chi)}^{\scriptsizeZss}(Y; \beta, c)$
		  by stepwise reductions of the error charge of the Fourier-Mukai transforms involved.
		  		
With
  the order $\prec$ on ${\Bbb C}$ in Definition~\ref{order-C}  and
                                                                   % Definition [order on ${\Bbb C}$]
  the properties of the error charge $\Err^Z_{\bullet}(\,\bullet\,)$
  from Lemma~\ref{definity-ec} and Lemma~\ref{vecacfCp},
      % Lemma [definity of $\Err^Z_{C^{\prime}}(\tilde{\cal F}^{\prime})$]
      % Lemma [vanishing of error charge as criterion of flatness$\,/C^{\prime}$]
one has:
  \begin{itemize}
   \item[{\Large $\cdot$}]
     $\Err^Z_{C^{\prime}}(\tilde{\cal F}^{\prime})\succcurlyeq 0$
       for all $\tilde{\cal F}^{\prime}\in \CohCategory_1(C^{\prime}\times Y)$.
	
   \item[{\Large $\cdot$}]	
    The equality  ``$=$" holds
	if and only if $\tilde{\cal F}^{\prime}$ is flat over $C^{\prime}$.	
  \end{itemize}	
 
Let
  $T$ be an affine smooth curve with parameter $t$ and a base-point,
    denoted by $0$ with its ideal sheaf denoted by $m_0=(t)$,   and
  $U:= T-\{0\}$.
Consider a $U$-family of $Z$-semistable morphisms from Azumaya nodal curves
  with a fundamental module to $Y$ of type $(g;r,\chi;\beta,c)$, given by
  $$
   \tilde{\cal E}^{\prime}_U\in \CohCategory (C^{\prime}_U\times Y)\,,
  $$
where
 \begin{itemize}
   \item[{\Large $\cdot$}]
     $C^{\prime}_U$ is a $U$-family of nodal curves with
	  a built-in collapsing morphism
	 $\rho^{\prime}_U: C^{\prime}_U/U\rightarrow C_U/U$,
      where $C_U/U$ is a $U$-family of stable curves of genus $g$;

   \item[{\Large $\cdot$}]	
    $\tilde{\cal E}^{\prime}_U$
     	is flat, of relative dimension $0$ and relative length $r$,
	 over $C^{\prime}_U$;
	
   \item[{\Large $\cdot$}]
    $\tilde{\cal F}_U
	   :=(\rho^{\prime}_U\times\Id_Y)_{\ast}(\tilde{\cal E}^{\prime}_U)
	   \in\CohCategory(C_U\times Y)$
	  is a $U$-family of $Z$-semistable Fourier-Mukai transforms from stable curves of genus $g$
     to $Y$.	
 \end{itemize}
The data is also equipped with a built-in exact sequence of
 ${\cal O}_{C^{\prime}_U\times Y}$-modules
 $$
   (\rho^{\prime}_U\times\Id_Y)^{\ast}(\tilde{\cal F}_U)\;
     \longrightarrow\;  \tilde{\cal E}^{\prime}_U\;
	 \longrightarrow\; 0\,.
 $$

In the discussion below, some procedures may require a finite base change to realize.
To avoid overflow of notations, we set the convention that the new base (with the pulled-back family)
 from a base change will be still denoted by $T$.

%\vspace{7em}
\bigskip

\begin{flushleft}
{\bf Step $\boldsymbol{(a):}$ Filling in a Fourier-Mukai transform over $0\in T$}
\end{flushleft}
Possibly after passing to a base change, one may fill the $U$-family $C^{\prime}_U$
 of nodal curves to a $T$-family of nodal curves
 with a collapsing morphism $\rho^{\prime}_T:C^{\prime}_T\rightarrow C_T$ over $T$
 that extends the collapsing morphism
 $\rho^{\prime}_U: C^{\prime}_U\rightarrow  C_U$ over $U$.
For the $T$-family $C_T$ of stable curves,
  it follows from [L-Y3: Sec.~3] (D(10.1)) that
 one may complete the $U$-family $\tilde{\cal F}_U$  to a $T$-family
 $$
    \tilde{\cal F}_T\in \CohCategory(C_T\times Y)
 $$
 of $Z$-semistable Fourier-Mukai transform to stable curves of genus $g$ to $Y$,
 using the technique of elementary modifications by Stacy Langton [La].

\bigskip

\begin{flushleft}
{\bf Step $\boldsymbol{(b):}$ Filling in a quotient sheaf over $0\in T$}
\end{flushleft}
It follows from the properness of the Quot-scheme
 $\Quot_{C^{\prime}_T\times Y)/T}
      (  (\rho^{\prime}_T\times \Id_Y)^{\ast}(\tilde{\cal F}_T) ,\,\bullet\,)$
 of quotient sheaves of $(\rho^{\prime}_T\times\Id_Y)^{\ast}(\tilde{\cal F}_T)$
   over $T$
 that the $U$-family of quotient sheaves
   $(\rho^{\prime}_U\times\Id_Y)^{\ast}(\tilde{\cal F}_U)
                                              \rightarrow  \tilde{\cal E}^{\prime}_U \rightarrow 0$
  extends uniquely to a $T$-family of quotient sheaves
    $$
     (\rho^{\prime}_T\times\Id_Y)^{\ast}(\tilde{\cal F}_T)\;
                  \longrightarrow\;   \tilde{\cal E}^{\prime}_T\;   \longrightarrow\;   0\,,
    $$
   (with $\tilde{\cal E}^{\prime}_T$ flat over $T$), 	
 whose restriction to over $0\in T$ is the quotient sheaf
    $$
     (\rho^{\prime}_0\times\Id_Y)^{\ast}(\tilde{\cal F}_0)\;
                      \longrightarrow\;   \tilde{\cal E}^{\prime}_0\;   \longrightarrow\;   0\,.	
    $$

\bigskip

\begin{lemma}\label{vanishing-R1rpIdYEp0}
 {\bf [vanishing of
               $R^1\!(\rho^{\prime}_0\times\Id_Y)_{\ast}
	                (\tilde{\cal E}^{\prime}_0)$ on $C_0\times Y$
			      and
               $R^1\!(\rho^{\prime}_T\times\Id_Y)_{\ast}
	                (\tilde{\cal E}^{\prime}_T)$ on $C_T\times Y$].}
 $$
   R^1\!(\rho^{\prime}_0\times\Id_Y)_{\ast}
	                                                        (\tilde{\cal E}^{\prime}_0)\; =\; 0\;\;
    \mbox{on $C_0\times Y$}														
	 \hspace{2em}\mbox{and}\hspace{2em}
    R^1\!(\rho^{\prime}_T\times\Id_Y)_{\ast}
	                                                        (\tilde{\cal E}^{\prime}_T)\; =\; 0\;\;
    \mbox{on $C_T\times Y$}\,.
 $$															
\end{lemma}	

\begin{proof}
 These are the consequences of:
  \begin{itemize}
   \item[{\Large $\cdot$}]
    the exact sequences
      $(\rho^{\prime}_0\times\Id_Y)^{\ast}(\tilde{\cal F}_0)
                                \rightarrow  \tilde{\cal E}^{\prime}_0  \rightarrow  0$
	and
    $(\rho^{\prime}_T\times\Id_Y)^{\ast}(\tilde{\cal F}_T)
                               \rightarrow   \tilde{\cal E}^{\prime}_T   \rightarrow   0$
    respectively,
	
   \item[{\Large $\cdot$}]
    $C^{\prime}_0\times Y$ (resp.\ $C^{\prime}_T\times Y$)
     is of relative dimension $\le 1$	over $C_0\times Y$ (resp.\ $C_T\times Y$),
	
   \item[{\Large $\cdot$}]
    $R^1\!(\rho^{\prime}_0\times\Id_Y)_{\ast}
	                                  ({\cal O}_{C^{\prime}_0\times Y})=0$
    (resp.\   $R^1\!(\rho^{\prime}_T\times\Id_Y)_{\ast}
	                                        ({\cal O}_{C^{\prime}_T\times Y})=0$).
  \end{itemize}
 The details follow the same proof of Lemma~\ref{vR1sm}.
     % Lemma [vanishing of $R^1\!(\rho\times \Id_Y)_{\ast}(\,\mbox{graph}\,)$
     %                for semistable morphism]
                             
\end{proof}

\bigskip

Consider now the natural homomorphisms $\tilde{\alpha}_0$ and $\tilde{\alpha}_T$
 from the compositions
  $$
    \begin{array}{c}
      \tilde{\cal F}_0\; \longrightarrow\;
       (\rho^{\prime}_0\times \Id_Y)_{\ast}
	      (\rho^{\prime}_0\times \Id_Y)^{\ast}(\tilde{\cal F}_0)\;
	       \longrightarrow\;
		  (\rho^{\prime}_0\times \Id_Y)_{\ast}(\tilde{\cal E}^{\prime}_0)
		                                                                                                 \\[1.2ex]
      \mbox{and}\hspace{2em}																										 
      \tilde{\cal F}_T\; \longrightarrow\;
       (\rho^{\prime}_T\times \Id_Y)_{\ast}
	      (\rho^{\prime}_T\times \Id_Y)^{\ast}(\tilde{\cal F}_T)\;
	   \longrightarrow\;
	  (\rho^{\prime}_T\times \Id_Y)_{\ast}(\tilde{\cal E}^{\prime}_T)\,.
             \hspace{3em}		
	\end{array}
  $$
	
\bigskip

\begin{lemma}\label{coincidence-pf-C0Y-CTY}
 {\bf [coincidence of push-forward on $C_0\times Y$ and $C_T\times Y$].}
   The natural homomorphisms
     $$
	    \tilde{\alpha}_0\::\;
		   \tilde{\cal F}_0\; \longrightarrow\;
	      (\rho^{\prime}_0\times\Id_Y)_{\ast}(\tilde{\cal E}^{\prime}_0)\;\;
		 \mbox{on $C_0\times Y$}		
		  \hspace{1em}\mbox{and}\hspace{1em}
	    \tilde{\alpha}_T\;:\;
	     \tilde{\cal F}_T\; \longrightarrow\;
	      (\rho^{\prime}_T\times\Id_Y)_{\ast}(\tilde{\cal E}^{\prime}_T)\;\;
         \mbox{on $C_T\times Y$}		
	 $$
	 are isomorphisms.
\end{lemma}	

\begin{proof}
 We only need to show that
  $\tilde{\alpha}_T:
	     \tilde{\cal F}_T\longrightarrow
	      (\rho^{\prime}_T\times\Id_Y)_{\ast}(\tilde{\cal E}^{\prime}_T)$
  is an isomorphism
  over an open dense complement of a codimension-$2$ subset of $\Supp(\tilde{\cal F}_T)$.
 The lemma then follows from Lemma~\ref{vanishing-R1rpIdYEp0}  and
                                                    Lemma~\ref{rpf}.
	                    % Lemma  [vanishing of
                        %                   $R^1\!(\rho^{\prime}_0\times\Id_Y)_{\ast}
	                    %                         (\tilde{\cal E}^{\prime}_0)$  on $C_0\times Y$ and
                        %                    $R^1\!(\rho^{\prime}_T\times\Id_Y)_{\ast}
	                    %                         (\tilde{\cal E}^{\prime}_T)$ on $C_T\times Y$].}
						%
	                    % Lemma [relation on push-forward as a closed condition]
						
 First note that
  since
    $\tilde{\cal F}_T$ is flat over $T$ while
     $\Ker(\tilde{\alpha}_T)\subset \tilde{\cal F}_T$ is supported only over $0\in T$,
  it must be that 	
    $\Ker(\tilde{\alpha}_T)=0$ and $\tilde{\alpha}_T$ is a monomorphism.
 Note also that
   $$
   ((\rho^{\prime}_T\times\Id_Y)^{\ast}(\tilde{\cal F}_T))
                                                    _{\flatscriptsize/T}\;
	 =\; (\rho^{\prime}_T\times\Id_Y)^{\ast}(\tilde{\cal F}_T)
	           \left/
			      \left\{\, \parbox{14em}{\small
				    torsion sections of
                    $(\rho^{\prime}_T\times\Id_Y)^{\ast}(\tilde{\cal F}_T)$\\
				    supported over an infinitesimal\\ neighborhood of $0\in T$} \,
				  \right\}
			   \right.\, .
   $$	
 Since
     both $\tilde{\cal F}_T$ and
    $(\rho^{\prime}_T\times\Id_Y)_{\ast}(\tilde{\cal E})^{\prime}_T$
  are flat over $T$,
 one has the exact sequence
  $$
    ((\rho^{\prime}_T\times\Id_Y)^{\ast}(\tilde{\cal F}_T))
                                                    _{\flatscriptsize/T}\;
      \stackrel{\tilde{\beta}^{\prime}_T}{\longrightarrow}\; 	
	  \tilde{\cal E}^{\prime}_T\; \longrightarrow\; 0
  $$
  and
  $\tilde{\alpha}_T$ can be regarded as the composition
  $$
     \tilde{\cal F}_T\; \hookrightarrow\;
       (\rho^{\prime}_T\times \Id_Y)_{\ast}
	      (( (\rho^{\prime}_T\times \Id_Y)^{\ast}(\tilde{\cal F}_T)
		         )_{\flatscriptsize/T})\;
	   \longrightarrow\; (\rho^{\prime}_T\times \Id_Y)_{\ast}
	                                                                                 (\tilde{\cal E}^{\prime}_T)\,.
  $$
 
 Let $\tilde{e}_T\in
       (\rho^{\prime}_T\times \Id_Y)_{\ast}(\tilde{\cal E}^{\prime}_T)$
  be a local section of
  $(\rho^{\prime}_T\times \Id_Y)_{\ast}(\tilde{\cal E}^{\prime}_T)$
  that is regular and nowhere-zero over a nonempty open set of $C_0\times Y$.
 Then, there is a local section
  $\tilde{e}^{\prime}_T \in \tilde{\cal E}^{\prime}_T$,
   regular and nowhere-zero over a nonempty open subset of  $C^{\prime}_0\times Y$,
  such that
  $(\rho^{\prime}_T\times\Id_Y)_{\ast}(\tilde{e}^{\prime}_T)
     =\tilde{e}_T$.
 In turn,
  there is a local section
     $\tilde{f}^{\prime}_T \in
      ((\rho^{\prime}_T\times\Id_Y)^{\ast}(\tilde{\cal F}_T))
                                                    _{\flatscriptsize/T}$,
    regular and nowhere-zero over a nonempty open subset of $C^{\prime}_0\times Y$,
  such that
    $\tilde{\beta}^{\prime}_T (\tilde{f}^{\prime}_T)
	    = \tilde{e}^{\prime}_T$.
		
 Suppose that $\tilde{e}_T\notin \tilde{\alpha}_T(\tilde{\cal F}_T)$.
 Since
   $\tilde{\alpha}_U: \tilde{\cal F}_U \rightarrow
    (\rho^{\prime}_U\times\Id_Y)_{\ast}(\tilde{\cal E}^{\prime}_U)$
   is an isomorphism, 	
 let $\tilde{e}^{\prime}_U$ be the restriction of $\tilde{e}^{\prime}_T$
  to over $U\subset T$ (and similarly for other local section $(\,\cdot\,)_U$ ),
 then there exists an $\tilde{f}_U\in \tilde{\cal F}_U$ such that
  $\tilde{\alpha}_U(\tilde{f}_U)= \tilde{e}_U$.
 From naturality of the construction,
  $\tilde{\beta}^{\prime}_U(\tilde{f}^{\prime}_U)=\tilde{e}_U$
   also holds.
 Thus,
    with $\tilde{\cal F}_T$ as a subsheaf of
     $(\rho^{\prime}_T\times \Id_Y)_{\ast}
	      (( (\rho^{\prime}_T\times \Id_Y)^{\ast}(\tilde{\cal F}_T)
		         )_{\flatscriptsize/T})$,
   $\,\tilde{f}_U$ and
    $(\rho^{\prime}_T\times\Id_Y)_{\ast}(\tilde{f}^{\prime}_T)|_U$
    coincide as local sections of
    $(\rho^{\prime}_T\times \Id_Y)_{\ast}
	      (( (\rho^{\prime}_T\times \Id_Y)^{\ast}(\tilde{\cal F}_T)
		         )_{\flatscriptsize/T})$   				
  and
   $\tilde{f}^{\prime}_T$ gives rise to
   the extension of $\tilde{f}_U\in \tilde{\cal F}_U$  to
    $\,\tilde{f}_T :=
	   	(\rho^{\prime}_T\times\Id_Y)_{\ast}(\tilde{f}^{\prime}_T)
			 \in \tilde{\cal F}_T\,$
	with $\tilde{\alpha}_T(\tilde{f}_T)=\tilde{e}_T$.
   
  This proves that
     $\tilde{\alpha}_T:
	     \tilde{\cal F}_T\longrightarrow
	      (\rho^{\prime}_T\times\Id_Y)_{\ast}(\tilde{\cal E}^{\prime}_T)$
    is an isomorphism
   over an open dense complement of a codimension-$2$ subset of $\Supp(\tilde{\cal F}_T)$
   and, hence, the lemma.

\end{proof}
 
\bigskip	

If $\Err^Z_{\flatscriptsize/C^{\prime}_0}(\tilde{\cal E}^{\prime}_0)=0$,
 then go to the last Step (d).
 %%%%%%%%%%%%%
 % then after contracting the part of unstable ${\Bbb P}^1$-tree
 %    to which the restriction  of $\varphi_{\tilde{\cal E}^{\prime}_0}$
 %	is componentwise constant, 	
 % we are done.
 %%%%%%%%%%%%%
Otherwise, move on to the next Step (c) to continue the reduction.

\bigskip

\begin{flushleft}
{\bf Step $\boldsymbol{(c):}$ Turning on the reduction
		-- Bubbling off to absorb/reduce irregularities\\
		    $\mbox{\hspace{4.2em}}$
		   of the morphism over $0\in T$}
\end{flushleft}
We first reduce the irregularities of $\varphi_{\tilde{\cal E}^{\prime}_0}$
  that come from $1$-dimensional subsheaf of
  $(\tilde{\cal E}^{\prime}_0)_{\torsionscriptsize/C^{\prime}_0}$.
Once such irregularities are all resolved
  to obtain a new $\varphi_{\tilde{\cal E}^{\prime}_0}$
 we then reduce the irregularities of the new $\varphi_{\tilde{\cal E}^{\prime}_0}$
  that comes from points on $C^{\prime}_0$
  over which $\tilde{\cal E}^{\prime}_0$ is not flat.

\bigskip

\begin{flushleft}
{\bf Step $\boldsymbol{(c.1):}$
 If $\; -\Imaginary\Err^Z_{C^{\prime}_0}(\tilde{\cal E}^{\prime}_0)>0$}
\end{flushleft}
Then, there exists a point $p^{\prime}\in C^{\prime}_0$ such that
 $(\tilde{\cal E}^{\prime}_0)_{(p^{\prime})}$is $1$-dimensional.
Up to a base change of $T$, let
 $$
   \xymatrix{
   C^{\prime\prime}_T   \ar[rr]^-{h_T}\ar[rd]
     && C^{\prime}_T \ar[ld] \\
   & T	
   }
 $$
 be a collapsing morphism of $T$-family of nodal curves
  from blowing up an appropriate subscheme of $C^{\prime}_T$
   that contains $p^{\prime}$, now as a point in $C^{\prime}_T$,
  such that
   \begin{itemize}
    \item[{\Large $\cdot$}]
    $h_U$ is an isomorphism;
	
   \item[{\Large $\cdot$}]		
   the ${\Bbb P}^1$-tree subcurve $C^{\prime\prime}_{u_{h_0}}$
    of $C^{\prime\prime}_0$ that gets contracted by $h_0$ to $p^{\prime}$ is connected.
  \end{itemize}	
 
Consider the new $T$-family  of $1$-dimensional coherent sheaves
  $\tilde{\cal E}^{\prime\prime}_T
     \in  \CohCategory((C^{\prime\prime}_T\times Y)/T)$
  in the exact sequence over $T$
  $$
	  (h_T\times\Id_Y)^{\ast}(\tilde{\cal E}^{\prime}_T)\;
       \stackrel{\tilde{\beta}^{\prime\prime}_T}{\longrightarrow}\;
	   \tilde{\cal E}^{\prime\prime}_T\;   \longrightarrow\; 0
  $$
  from the Quot-scheme$/T$ completion of the exact sequence over $U$
  $$
	(h_U\times\Id_Y)^{\ast}(\tilde{\cal E}^{\prime}_U)\;
	  \stackrel{I\!d}{\longrightarrow}\;
	  (h_U\times\Id_Y)^{\ast}(\tilde{\cal E}^{\prime}_U)\;
	   \longrightarrow\; 0\,.
  $$
Note that
  $\Ker(\tilde{\beta}^{\prime\prime}_T)
      = ((h_T\times\Id_Y)^{\ast}(\tilde{\cal E}^{\prime}_T))
                                    _{\torsionscriptsize/T }$
	 is supported over an infinitesimal neighborhood of $0\in T$. 	
Then,
 similar argument in the proof of  Lemma~\ref{vanishing-R1rpIdYEp0}
   in Step (b) implies that
                                     % Lemma  [vanishing of
                                     %                 $R^1\!(\rho^{\prime}_0\times\Id_Y)_{\ast}
             	                     %                    (\tilde{\cal E}^{\prime}_0)$  on $C_0\times Y$
							         %                   and
                                     %                 $R^1\!(\rho^{\prime}_T\times\Id_Y)_{\ast}
	                                 %                       (\tilde{\cal E}^{\prime}_T)$ on $C_T\times Y$]
													
\bigskip

\begin{lemma}\label{vanishing-R1h0IdY-Cpp0Y-c1}
      {\bf [vanishing of
                 $R^1\!(h_0\times\Id_Y)_{\ast}
	              (\tilde{\cal E}^{\prime\prime}_0)$ on $C^{\prime}_0\times Y$  and
                 $R^1\!(h_T\times\Id_Y)_{\ast}
	              (\tilde{\cal E}^{\prime\prime}_T)$ on $C^{\prime}_T\times Y$].}
 $$
    R^1\!(h_0\times\Id_Y)_{\ast}
	                                                        (\tilde{\cal E}^{\prime\prime}_0)\; =\; 0\;\;
                  \mbox{on $C^{\prime}_0\times Y$}															
	 \hspace{2em}\mbox{and}\hspace{2em}
    R^1\!(h_T\times\Id_Y)_{\ast}
	                                                        (\tilde{\cal E}^{\prime\prime}_T)\; =\; 0\;\;
				  \mbox{on $C^{\prime}_T\times Y$}\,.
 $$															
\end{lemma}	

\bigskip

\noindent
And similar argument in the proof of Lemma~\ref{coincidence-pf-C0Y-CTY}
    in Step (b) implies that
               % Lemma [coincidence of push-forward on $C_0\times Y$ and $C_T\times Y$]
	
\bigskip

\begin{lemma}\label{coincidence-pf-Cp0Y-CpTY-c1}
{\bf [coincidence of push-forward on $C^{\prime}_0\times Y$ and $C^{\prime}_T\times Y$].}
   The natural homomorphisms
     $$
	     \tilde{\cal E}^{\prime}_0\; \longrightarrow\;
	      (h_0\times\Id_Y)_{\ast}(\tilde{\cal E}^{\prime\prime}_0)\;\;
		               \mbox{on $C^{\prime}_0\times Y$}
		  \hspace{2em}\mbox{and}\hspace{2em}
	     \tilde{\cal E}^{\prime}_T\; \longrightarrow\;
	      (h_T\times\Id_Y)_{\ast}(\tilde{\cal E}^{\prime\prime}_T)\;\;
                       \mbox{on $C^{\prime}_T\times Y$}		
	 $$
	 from the compositions
	 $$
    \begin{array}{c}
      \tilde{\cal E}^{\prime}_0\; \longrightarrow\;
       (h_0\times \Id_Y)_{\ast}
	      (h_0\times \Id_Y)^{\ast}(\tilde{\cal E}^{\prime}_0)\;
	       \longrightarrow\;
		  (h_0\times \Id_Y)_{\ast}(\tilde{\cal E}^{\prime\prime}_0)
		                                                                                                 \\[1.2ex]
      \mbox{and}\hspace{2em}																										 
      \tilde{\cal E}^{\prime}_T\; \longrightarrow\;
       (h_T\times \Id_Y)_{\ast}
	      (h_T\times \Id_Y)^{\ast}(\tilde{\cal E}^{\prime}_T)\;
	   \longrightarrow\;
	  (h_T\times \Id_Y)_{\ast}(\tilde{\cal E}^{\prime\prime}_T)\,.
             \hspace{3em}		
	\end{array}
  $$
  are isomorphisms.
\end{lemma}	
	
\bigskip

By the nature of the blow-up in the construction, one has in addition the following lemma:

\bigskip

\begin{lemma}\label{diagram-sj}
{\bf [diagram of surjections].}
 The natural diagram of morphisms from the construction
 $$
  \xymatrix{
   (\Supp(\tilde{\cal E}^{\prime\prime}_T))_{\redscriptsize}
         \ar[rr]^-{h_T\times I\!d_Y}
		 \ar[d]_-{pr_{C^{\prime\prime}_T}}
	   && (\Supp(\tilde{\cal E}^{\prime}_T))_{\redscriptsize}
	             \ar[d]^-{pr_{C^{\prime}_T}}\\
	 C^{\prime\prime}_T  \ar[rr]^-{h_T}
	   &&   C^{\prime}_T
   }
 $$
  has all the arrows surjections.
\end{lemma}
  
\bigskip

\noindent
It follows that
   there is a ${\Bbb P}^1$-component  of $C^{\prime\prime}_{u_{h_0}}$
   such that
     ${\pr_Y}_{\ast}((\tilde{\cal E}^{\prime\prime}_0|_{{\Bbb P}^1})$
	 is $1$-dimensional   and, hence,
  we are in `{\it Case $(a):$   $\Ker(\alpha)$ is $1$-dimensional}$\,$'
   of the previous theme.
 It follows from the discussion there that
   $$
     0\; \le \;    -\,\Imaginary \Err^Z_{C^{\prime\prime}_0}
	                                                     (\tilde{\cal E}^{\prime\prime}_0)\;
	 <\;  -\,\Imaginary\Err^Z_{C^{\prime}_0}(\tilde{\cal E}^{\prime}_0)\,.
   $$
  
If $\Err^Z_{C^{\prime\prime}_0}(\tilde{\cal E}^{\prime\prime}_0)=0$,
 then go to the last Step (d).
Else,
  if $\Err^Z_{C^{\prime\prime}_0}(\tilde{\cal E}^{\prime\prime}_0)\ne 0$
   but
      $- \Imaginary\Err^Z_{C^{\prime\prime}_0}
                                           (\tilde{\cal E}^{\prime\prime}_0)=0$,
 then redenote $(C^{\prime\prime}_T, \tilde{\cal E}^{\prime\prime}_T)$
        as $(C^{\prime}_T, \tilde{\cal E}^{\prime}_T)$
         and go to the next Step (c.2).
Finally,
 if $- \Imaginary\Err^Z_{C^{\prime\prime}_0}
                                           (\tilde{\cal E}^{\prime\prime}_0)\ne 0$,
   redenote $(C^{\prime\prime}_T, \tilde{\cal E}^{\prime\prime}_T)$
   as $(C^{\prime}_T, \tilde{\cal E}^{\prime}_T)$
   and repeat the current Step (c.1).

\bigskip

\begin{flushleft}
{\bf Step $\boldsymbol{(c.2):}$
 If $\; -\Imaginary\Err^Z_{C^{\prime}_0}(\tilde{\cal E}^{\prime}_0)=0$}
\end{flushleft}
$\tilde{\cal E}^{\prime}_0$ is now of relative dimension $0$ over $C^{\prime}_0$.
Let
  $p^{\prime}\in C^{\prime}_0$ be such that
    $\tilde{\cal E}^{\prime}_0$ is not flat over $p^{\prime}$.
As in Step (c.1), up to a base change of $T$, let
 $$
   \xymatrix{
   C^{\prime\prime}_T   \ar[rr]^-{h_T}\ar[rd]
     && C^{\prime}_T \ar[ld] \\
   & T	
   }
 $$
 be a collapsing morphism of $T$-family of nodal curves
  from blowing up an appropriate subscheme of $C^{\prime}_T$
   that contains $p^{\prime}$, now as a point in $C^{\prime}_T$,
  such that
   \begin{itemize}
    \item[{\Large $\cdot$}]
    $h_U$ is an isomorphism;
	
   \item[{\Large $\cdot$}]		
   the ${\Bbb P}^1$-tree subcurve $C^{\prime\prime}_{u_{h_0}}$
    of $C^{\prime\prime}_0$ that gets contracted by $h_0$ to $p^{\prime}$ is connected.
  \end{itemize}	
 
 Consider the new $T$-family  of $1$-dimensional coherent sheaves
  $\tilde{\cal E}^{\prime\prime}_T
     \in  \CohCategory((C^{\prime\prime}_T\times Y)/T)$
  in the exact sequence over $T$
  $$
   (h_T\times\Id_Y)^{\ast}(\tilde{\cal E}^{\prime}_T)\;
       \stackrel{\tilde{\beta}^{\prime\prime}_T}{\longrightarrow}\;
	   \tilde{\cal E}^{\prime\prime}_T\;   \longrightarrow\; 0
  $$
  from the Quot-scheme$/T$ completion of the exact sequence over $U$
  $$
   (h_U\times\Id_Y)^{\ast}(\tilde{\cal E}^{\prime}_U)\;
       \stackrel{I\!d}{\longrightarrow}\;
	 (h_U\times\Id_Y)^{\ast}(\tilde{\cal E}^{\prime}_U)\;
	   \longrightarrow\; 0\,.
  $$
Again,
$\Ker(\tilde{\beta}^{\prime\prime}_T)
   = ((h_T\times\Id_Y)^{\ast}(\tilde{\cal E}^{\prime}_T))
                                    _{\torsionscriptsize/T }$
	 is supported over an infinitesimal neighborhood of $0\in T$    and
one has the following two lemmas  by the same reasoning for the corresponding lemmas
  in Step (b) and Step (c.1):
													
\bigskip

\begin{lemma}\label{vanishing-R1h0IdY-Cp0Y-CpTY-c2}
      {\bf [vanishing of
                 $R^1\!(h_0\times\Id_Y)_{\ast}
	              (\tilde{\cal E}^{\prime\prime}_0)$ on $C^{\prime}_0\times Y$  and
                 $R^1\!(h_T\times\Id_Y)_{\ast}
	              (\tilde{\cal E}^{\prime\prime}_T)$ on $C^{\prime}_T\times Y$].}
 $$
    R^1\!(h_0\times\Id_Y)_{\ast}
	                                                        (\tilde{\cal E}^{\prime\prime}_0)\; =\; 0\;\;
                  \mbox{on $C^{\prime}_0\times Y$}															
	 \hspace{2em}\mbox{and}\hspace{2em}
    R^1\!(h_T\times\Id_Y)_{\ast}
	                                                        (\tilde{\cal E}^{\prime\prime}_T)\; =\; 0\;\;
				  \mbox{on $C^{\prime}_T\times Y$}\,.
 $$															
\end{lemma}	

\bigskip
 
\begin{lemma}\label{coincidence-pf-Cp0Y-CpTY-c2}
{\bf [coincidence of push-forward on $C^{\prime}_0\times Y$ and $C^{\prime}_T\times Y$].}
   The natural homomorphisms
     $$
	     \tilde{\cal E}^{\prime}_0\; \longrightarrow\;
	      (h_0\times\Id_Y)_{\ast}(\tilde{\cal E}^{\prime\prime}_0)\;\;
		               \mbox{on $C^{\prime}_0\times Y$}
		  \hspace{2em}\mbox{and}\hspace{2em}
	     \tilde{\cal E}^{\prime}_T\; \longrightarrow\;
	      (h_T\times\Id_Y)_{\ast}(\tilde{\cal E}^{\prime\prime}_T)\;\;
                       \mbox{on $C^{\prime}_T\times Y$}		
	 $$
	 from the compositions
	 $$
    \begin{array}{c}
      \tilde{\cal E}^{\prime}_0\; \longrightarrow\;
       (h_0\times \Id_Y)_{\ast}
	      (h_0\times \Id_Y)^{\ast}(\tilde{\cal E}^{\prime}_0)\;
	       \longrightarrow\;
		  (h_0\times \Id_Y)_{\ast}(\tilde{\cal E}^{\prime\prime}_0)
		                                                                                                 \\[1.2ex]
      \mbox{and}\hspace{2em}																										 
      \tilde{\cal E}^{\prime}_T\; \longrightarrow\;
       (h_T\times \Id_Y)_{\ast}
	      (h_T\times \Id_Y)^{\ast}(\tilde{\cal E}^{\prime}_T)\;
	   \longrightarrow\;
	  (h_T\times \Id_Y)_{\ast}(\tilde{\cal E}^{\prime\prime}_T)\,.
             \hspace{3em}		
	\end{array}
  $$
  are isomorphisms.
\end{lemma}	
	
\bigskip
 
Furthermore,  one has the following positivity property:
  
\bigskip

\begin{lemma}\label{positive-pfCppuh0}
{\bf [positivity of push-forward to $C^{\prime\prime}_{u_{h_0}}$].}\\
 The torsion-free sheaf
  $\,({\pr_{C^{\prime\prime}_0}}_{\ast}(\tilde{\cal E}^{\prime\prime}_0)
           |_{C^{\prime\prime}_{u_{h_0}}})_{\torsionfreescriptsize}$
  on the ${\Bbb P}^1$-tree $C^{\prime\prime}_{u_{h_0}}$ is positive.
\end{lemma}

\begin{proof}
 Since ${\pr_{C^{\prime}_T}}_{\ast}(\tilde{\cal E}^{\prime}_T)$
     is coherent,
  there exist
      an affine neighborhood $U^{\prime}_T$  of $p^{\prime}$ in $C^{\prime}_T$
      and a $k>0$
    such that
	 $({\pr_{C^{\prime}_T}}_{\ast}(\tilde{\cal E}^{\prime}_T))
	               |_{U^{\prime}_T}$ is realized as a quotient sheaf
	 $$
	   {\cal O}_{U^{\prime}_T}^{\,\oplus k}\;
	       \longrightarrow\;
	     ({\pr_{C^{\prime}_T}}_{\ast}(\tilde{\cal E}^{\prime}_T))
	               |_{U^{\prime}_T}\;
           \longrightarrow\; 0\,.				   				
	 $$
 The composition of homomorphisms of ${\cal O }_{U^{\prime}_T\times Y}$-modules
   $$
     \pr_{U^{\prime}_T}^{\ast}
	    ({\cal O}_{U^{\prime}_T}^{\,\oplus k})\;
     \longrightarrow\;
     \pr_{U^{\prime}_T}^{\ast}	
	    ( ({\pr_{C^{\prime}_T}}_{\ast}(\tilde{\cal E}^{\prime}_T))
	               |_{U^{\prime}_T})\;
     \longrightarrow\;
	  \tilde{\cal E}^{\prime}_T|_{U^{\prime}_T}	
   $$
  gives a quotient sequence of ${\cal O}_{U^{\prime}_T\times Y}$-modules
   $$
	  {\cal O}_{U^{\prime}_T\times Y}^{\,\oplus k}\;    \longrightarrow\;
	  \tilde{\cal E}^{\prime}_T|_{U^{\prime}_T}\;        \longrightarrow\; 0\,.	
   $$
  Since $\tilde{\cal E}^{\prime}_T$
      is flat over $U^{\prime}_T-\{p^{\prime}\}$, of relative length $r$,
   the above quotient sequence induces a rational map
    $$
	  \xymatrix{
       f^{\prime}_T\;:\;  U^{\prime}_T\;\;
	       \ar@{.>}[r]
		  & \;\;\Quot_Y({\cal O}_Y^{\,\oplus k},r)
	   }	
	$$
	that is regular on $U^{\prime}_T-\{p^{\prime}\}$
  	and has a point of indeterminancy $p^{\prime}$.
  Let
    $U^{\prime\prime}_T:= h_T^{-1}(U^{\prime}_T)\subset C^{\prime\prime}_T$.	
  Then, $U^{\prime\prime}_T$ contains $C^{\prime\prime}_{u_{h_0}}$  and
    the exact sequence
    $$
       {\cal O}_{U^{\prime\prime}\times Y}^{\,\oplus k}\; \longrightarrow\;
         \tilde{\cal E}^{\prime\prime}_T|_{U^{\prime\prime}_T}\;
		 \longrightarrow\; 0	
    $$	
	from the composition of homomorphisms of
	${\cal O}_{U^{\prime\prime}_T\times Y}$-modules
   $$
      (h_T|_{U^{\prime\prime}_T}\times Y)^{\ast}
	      ({\cal O}_{U^{\prime}_T\times Y}^{\,\oplus k})\;  \longrightarrow\;
	  (h_T|_{U^{\prime\prime}_T}\times Y)^{\ast}
          (\tilde{\cal E}^{\prime}|_{U^{\prime}_T})          \;  \longrightarrow\;
         \tilde{\cal E}^{\prime\prime}_T|_{U^{\prime\prime}_T}		  		               	
   $$
    defines rational map
   $$
     \xymatrix{
      f^{\prime\prime}_T\; :\;  U^{\prime\prime}_T\;\;
	   \ar@{.>}[r]
	   & \;\;\Quot_Y({\cal O}_Y^{\,\oplus k},r)
	  }	
   $$
   that is regular on
    $(h_T|_{U^{\prime\prime}_T})^{-1}(U^{\prime}_T-\{p^{\prime}\})
      	= U^{\prime\prime}_T- C^{\prime\prime}_{u_{h_0}}$.
  Furthermore,
  since
     $U^{\prime\prime}_T$ is smooth on the complement of a codimension-$2$ subset
	   and
	 $\Quot_Y({\cal O}_Y^{\,\oplus k},r)$ is proper,
   the domain of definition for $f^{\prime\prime}_T$ extends over an open dense subset
	$V^{\prime\prime}_0\subset C^{\prime\prime}_{u_{h_0}}$.	
 The nature of the blow-up construction implies that
   $f^{\prime\prime}_0(V^{\prime\prime}_0)$ must be $1$-dimensional.
 It follows now by the same argument as that for Lemma~\ref{ggEpdE} that
                     % Lemma [global generation of ${\cal E}$ and positivity of $\degree({\cal E})$ ]
  the torsion-free sheaf
  $\,({\pr_{C^{\prime\prime}_0}}_{\ast}(\tilde{\cal E}^{\prime\prime}_0)
           |_{C^{\prime\prime}_{u_{h_0}}})_{\torsionfreescriptsize}$
  on the ${\Bbb P}^1$-tree $C^{\prime\prime}_{u_{h_0}}$ is positive.
         
\end{proof}

\bigskip

\noindent
It follows from Lemma~\ref{difference-ec} that, as nonnegative integers,
                    % Lemma [difference of error charges]
  $$
     0\; \le \;  \Err^Z_{C^{\prime\prime}_0}
	                                                     (\tilde{\cal E}^{\prime\prime}_0)\;
	 <\;  \Err^Z_{C^{\prime}_0}(\tilde{\cal E}^{\prime}_0)\,.
   $$
	
If $\Err^Z_{\flatscriptsize/C^{\prime\prime}_0}
                                           (\tilde{\cal E}^{\prime\prime}_0)=0$,
 then go to the next and last Step (d).
 %%%%%%%%%%%%%%%%%%%%%%%%%%%%%%%%%%%
 % after contracting the part of unstable ${\Bbb P}^1$-tree
 %    to which the restriction  of $\varphi_{\tilde{\cal E}^{\prime\prime}_0}$
 % 	is componentwise constant, 	
 % we are done.
 %%%%%%%%%%%%%%%%%%%%%%%%%%%%%%%%%%%%
Otherwise,
   redenote $(C^{\prime\prime}_T, \tilde{\cal E}^{\prime\prime}_T)$
     as $(C^{\prime}_T, \tilde{\cal E}^{\prime}_T)$
   and repeat the current Step (c.2).

\bigskip

\begin{flushleft}
{\bf Step $\boldsymbol{(d):}$  Termination of the reduction\\
          $\mbox{\hspace{4.2em}}$
          -- Recovery of a regular morphism in our category over $t\in T$}
\end{flushleft}
Since
 $$
    0\; \preccurlyeq\;
     \Err^Z_{C^{\prime\prime}_0}(\tilde{\cal E}^{\prime\prime}_0)\;
	 \prec\;    \Err^Z_{C^{\prime}_0}(\tilde{\cal E}^{\prime}_0)
 $$
  by Lemma~\ref{definity-ec} and Proposition~\ref{decrease-ec},   and
       % Lemma         [definity of $\Err^Z_{C^{\prime}}(\tilde{\cal F}^{\prime})$]
       % Proposition  [decrease of $\Err^Z_{\bullet}(\,\,\bullet)$ under bubbling off]
in Step (c)
we reduce first $-\Imaginary\Err^Z_{\bullet}(\,\bullet\,) $ until it becomes $0$ and
  then reduce the now-nonnegative-integer $\Err^Z_{\bullet}(\,\bullet\,)$ until it is $0$,
Step (c) must terminate after finitely many repetitions to give a
 $\tilde{\cal E}^{\prime\prime}_0\in \Coh_1(C^{\prime\prime}_0\times Y)$
 with $\Err^Z_{C^{\prime\prime}_0}(\tilde{\cal E}^{\prime\prime}_0)=0$
 since these error charges  lie in a locally finite rank-$2$ lattice in ${\Bbb C}$.

\bigskip

\begin{definition}\label{admissible-tree}
{\bf [admissible ${\Bbb P}^1$-tree].} {\rm
 Let $C^{\prime\prime}$ be a nodal curve.
 A ${\Bbb P}^1$-tree subcurve $C^{\prime\prime}_{(1)}$  of $C^{\prime\prime}$
  is called {\it admissible}  if  $C^{\prime\prime}_{(1)}$ satisfies the following condition:
  \begin{itemize}
   \item[{\Large $\cdot$}]
    Let $C^{\prime\prime}
	          =C^{\prime\prime}_{(0)}\cup C^{\prime\prime}_{(1)}$
	  be the decomposition of $C^{\prime\prime}$ into the union
	   of $C^{\prime\prime}_{(1)}$
	   and the complementary closed subcurve $C^{\prime\prime}_{(0)}$.
    Then, each connected component of $C^{\prime\prime}_{(1)}$
	   intersects with $C^{\prime\prime}_{(0)}$ at only either one or two (smooth) points.
  \end{itemize}
  (Note that by definition, connected components of $C^{\prime\prime}_{(1)}$
      are disjoint from each other in $C^{\prime\prime}$.)
}\end{definition}
 
\bigskip
 
\begin{lemma}\label{factorization-cm}
 {\bf [factorization of collapsing morphism].}
 Let
   $$
     \xymatrix{
       C^{\prime\prime}_T\ar[rr]^-{\rho^{\prime\prime}_T} \ar[rd]   && C_T \ar[ld]   \\
	   & T
	 }
   $$
   be a $T$-morphism between two (flat) $T$-families of nodal curves
    of the same (arithmetic)genus
   such that $\rho^{\prime\prime}_U$     is an isomorphism.
  (Here, recall that $0\in T$ and $U:=T-\{0\}$.)
  Note that $\rho^{\prime\prime}_0$ must be a collapsing morphism that contracts
   an admissible ${\Bbb P}^1$-tree subcurve
    $C^{\prime\prime}_{u_{\rho^{\prime\prime}_0}}$  of $C^{\prime\prime}_0$.
  Let $C^{\prime\prime}_{0; (1)}$	 be an admissible ${\Bbb P}^1$-tree subcurve of
   $C^{\prime\prime}_0$ that is contained in
   $C^{\prime\prime}_{u_{\rho^{\prime\prime}_0}}$.
  Then, up to a base change on $T$ with the pulled-back family,
    there exist collapsing morphisms between $T$-family of nodal curves of the same genus
	  $$
	    \xymatrix{
	       C^{\prime\prime}_T \ar[rr]^-{h_T}\ar[rd]
    		 && \underline{C}^{\prime\prime}_T \ar[ld]  \\
		   & T
		 }
		 \hspace{3em}\mbox{and}\hspace{3em}
        \xymatrix{
	       \underline{C}^{\prime\prime}_T
		      \ar[rr]^-{\underline{\rho}^{\prime\prime}_T}\ar[rd]   && C_T \ar[ld] \\
		     & T
		 } 		
	  $$
     where $h_T$ contracts exactly $C^{\prime\prime}_{0;(1)}$,
     such that $\rho^{\prime\prime}_T=   \underline{\rho}^{\prime\prime}_T\circ h_T$.
\end{lemma}	
 
\begin{proof}
 Let $H_T$ be an ample Cartier divisor on $C_T$
  that is relative very ample on $C_T/T$ and supported in the relative smooth locus of
   $(C_T
          -\rho^{\prime\prime}_T
		       (C^{\prime\prime}_{u_{\rho^{\prime\prime}_0}}))
		  /T$.
 Then the pull-back Cartier divisor
   $H^{\prime\prime}_T := {\rho^{\prime\prime}_T}^{-1}(H_T)$
   is supported on the relative smooth locus of
   $(C^{\prime\prime}_T-C^{\prime\prime}_{u_{\rho^{\prime\prime}_0}})/T$
    and
  the tautological quotient sheaf
   ${\cal O}_{C^{\prime\prime}_T}\otimes_{\Bbb C}
        H^0(C^{\prime\prime}_T,
	              {\cal O}_{C^{\prime\prime}_T}(H^{\prime\prime}_T))
       \rightarrow
    {\cal O}_{C^{\prime\prime}_T}(H^{\prime\prime}_T)$
   from evaluations of global sections
  determines a morphism
  $f_T: C^{\prime\prime}_T
     \rightarrow
     \Proj(H^0(C^{\prime\prime}_T,
	       {\cal O}_{C^{\prime\prime}_T}(H^{\prime\prime}_T)))$
  whose image $\Image(f_T)$ is isomorphic to $C_T$ and
  the morphism $f_T: C^{\prime\prime}_T\rightarrow \Image(f_T)$
  recovers $\rho^{\prime\prime}_T$.
   
 Similarly, up to a base change on $T$,
  there exists an effective Cartier divisor $\hat{H}^{\prime\prime}_T$
   on $C^{\prime\prime}_T$, supported in the relative smooth locus of
   $(C^{\prime\prime}_T-C^{\prime\prime}_{0;(1)})/T$.
 Then $m\hat{H}^{\prime\prime}_T$, $m$ large enough,
  determines a collapsing morphism $h_T$
  that contracts exactly $C^{\prime\prime}_{0;(1)}$.
 In turn, up to a base change on $T$, there exists an effective Cartier divisor
  $\underline{H}^{\prime\prime}_T$ on $\underline{C}^{\prime\prime}_T$,
   supported in the relative smooth locus of
   $(\underline{C}^{\prime\prime}_T
         -h_T(C^{\prime\prime}_{u_{\rho^{\prime\prime}_0}}))/T$,
  and $\underline{m}\, \underline{H}^{\prime\prime}_T$, $\underline{m}$ large enough,
   determines a collapsing morphism $\underline{\rho}^{\prime\prime}_T$
   that contracts exactly
   $h_T(C^{\prime\prime}_{u_{\rho^{\prime\prime}_0}})$.
 The lemma follows.
 
\end{proof}
  
\bigskip

Resume our proof of the completeness  of
  ${\frak M}_{Az^{\!f}\!(g;r,\chi)}^{\scriptsizeZss}(Y; \beta, c)$.
By construction and with some redenotation, we now have
  \begin{itemize}
   \item[{\Large $\cdot$}]
    a $T$-family of nodal curves
      $C^{\prime\prime}_T/T$ with a collapsing $T$-morphism
      $\rho^{\prime\prime}_T: C^{\prime\prime}_T\rightarrow C_T$
      such that
	  \begin{itemize}
	   \item[{\Large $\cdot$}]
	    $C^{\prime\prime}_U/U \simeq C^{\prime}_U/U$ specified in the beginning;
	
	   \item[{\Large $\cdot$}]
	    under the above isomorphism,
		$\rho^{\prime\prime}_U = \rho^{\prime}_U$.	
	  \end{itemize}
   
   \item[{\Large $\cdot$}]
    a  coherent sheaf $\tilde{\cal E}^{\prime\prime}_T$
	   on $C^{\prime\prime}_T\times Y$
     that is flat over $C^{\prime\prime}_T$, of relative length $r$, 	   	
	with the following properties:
	\begin{itemize}
     \item[(0)]	
	  under the $U$-isomorphism $C^{\prime\prime}_U \simeq C^{\prime}_U$ above,
      $\tilde{\cal E }^{\prime\prime}_U\simeq \tilde{\cal E}^{\prime}_U$.	
	
	 \item[(1)]
	  $(\rho^{\prime\prime}_T\times\Id_Y)_{\ast}
	       (\tilde{\cal E}^{\prime\prime}_T)=:\tilde{\cal F}_T$
	   gives a (flat) $T$-family of $Z$-semistable Fourier-Mukai transforms
      from fibers of $C_T/T$ to $Y$.

    \item[(2)]
     The natural sequence of homomorphisms
	  $\;(\rho^{\prime\prime}_T\times\Id_Y)^{\ast}(\tilde{\cal F}_T)
	      \rightarrow \tilde{\cal E}^{\prime\prime}_T\rightarrow 0\;$
	  is exact.	      	
    \end{itemize}	
 \end{itemize}
Let
 $C^{\prime\prime}_{u_{\rho^{\prime\prime}_0}}$
    be the ${\Bbb P}^1$-tree subcurve of $C^{\prime\prime}_0$
    that is collapsed by $\rho^{\prime\prime}_0$.
	
\bigskip

\begin{lemma}\label{collapsing-aPti}
{\bf  [collapsing admissible ${\Bbb P}^1$-tree in
            $C^{\prime\prime}_{u_{\rho^{\prime\prime}_0}}$].}
 Let
  \begin{itemize}
    \item[{\Large $\cdot$}]
     $C^{\prime\prime}_{0;(1)}
                  \subset C^{\prime\prime}_{u_{\rho^{\prime\prime}_0}}$
     be a connected admissible ${\Bbb P}^1$-tree subcurve of $C^{\prime\prime}_0$
     that is contained in $C^{\prime\prime}_{u_{\rho^{\prime\prime}_0}}$
     such that the restriction of $\varphi_{\tilde{\cal E}^{\prime\prime}_0}$
     to $C^{\prime\prime}_{0;(1)}$ is constant,  and
	
   \item[{\Large $\cdot$}]
    $$
	   \xymatrix{
	       C^{\prime\prime}_T \ar[rr]^-{h_T}\ar[rd]
    		 && \underline{C}^{\prime\prime}_T \ar[ld]  \\
		   & T
		 }
		 \hspace{3em}\mbox{and}\hspace{3em}
        \xymatrix{
	       \underline{C}^{\prime\prime}_T
		      \ar[rr]^-{\underline{\rho}^{\prime\prime}_T}\ar[rd]   && C_T \ar[ld] \\
		     & T
		 }\,, 		
	  $$
     where $h_T$ contracts exactly $C^{\prime\prime}_{0;(1)}$
     and $\rho^{\prime\prime}_T=   \underline{\rho}^{\prime\prime}_T\circ h_T$,
	 be the induced factorization of collapsing morphisms from Lemma~\ref{factorization-cm}.
                                                                        % Lemma [factorization of collapsing morphism]
  \end{itemize}
 Define  
  $$
      \underline{\tilde{\cal E}}^{\prime\prime}_T\;
      :=\;   (h_T\times \Id_Y)_{\ast}(\tilde{\cal E}^{\prime\prime}_T)\,. 
  $$
 Then, 
  $\; \tilde{\cal F}^{\prime\prime}_T
	 = (\underline{\rho}^{\prime\prime}_T\times \Id_Y)_{\ast}
	                (\underline{\tilde{\cal E}}^{\prime\prime}_T)\;$
     and 
  the natural sequence
	 $\;(\underline{\rho}^{\prime\prime}_T\times\Id_Y)^{\ast}(\tilde{\cal F}_T)
	      \rightarrow \underline{\tilde{\cal E}}^{\prime\prime}_T\rightarrow 0\;$
  of homomorphisms is exact.	      	  
\end{lemma} 

\begin{proof}
 This follows from the fact that 
   $H^0(Z;{\cal O}_Z)={\Bbb C}$ for any connected projective variety $Z$     and that
   $\Supp(\tilde{\cal E}^{\prime\prime}_T)$ is affine over $C^{\prime\prime}_T$. 
  
\end{proof} 

\bigskip
 
Thus, 
  after contracting step by step all the admissible ${\Bbb P}^1$-tree subcurves
       in $C^{\prime\prime}_0$ 
     that are contained in $C^{\prime\prime}_{u_{\rho^{\prime\prime}_0}}$
         and to whose components 
         the restriction of $\varphi_{\tilde{\cal E}^{\prime\prime}_0}$ are constant
  and redenotation,
   $\tilde{\cal E}^{\prime\prime}_T$ now has in addition the following property:
  \begin{itemize}
   \item[]
	\begin{itemize}
	  \item[(3)]
       For each ${\Bbb P}^1$-component (denoted by ${\Bbb P}^1$) 
	     of the ${\Bbb P}^1$-tree subcurve of $C^{\prime\prime}_0$ 
	    that is collapsed by $\rho^{\prime\prime}_0$ to points in $C_0$,
       if $\varphi_{\tilde{\cal E}^{\prime\prime}_0}|_{{\Bbb P}^1}$	
    	   is a constant morphism, 
       then ${\Bbb P}^1$ has at least three special points 	   
    \end{itemize}
  \end{itemize}
 In other words,  we obtains 
    a $T$-family of $Z$-semistable morphisms from Azumaya nodal curves 
	with a fundamental module to $Y$ of type $(g;r,\chi;\beta,c)$
  that, up to a base change, extends the original $U$-family of $Z$-semistable morphisms
   of type $(g;r,\chi;\beta,c)$.
 
\bigskip

This proves the completeness of 
  ${\frak M}_{Az^{\!f}\!(g;r,\chi)}^{\scriptsizeZss}(Y; \beta, c)$.
 
Together with Sec.~4.1, this proves Theorem~4.0
    % Theorem [${\frak M}_{Az^{\!f}\!(g;r,\chi)}^{\scriptsizeZss}(Y;\beta,c)$ compact]
 on the compactness of
 ${\frak M}_{Az^{\!f}\!(g;r,\chi)}^{\scriptsizeZss}(Y; \beta, c)$.

%%%%%%%%%%%%%%%%%%%%%%%%%%%
%
% \bigskip
% 
% \section{Stability condition revisited}
% 
% \bigskip
% 
% \noindent $\bullet$
%  Issue on choices of a relative positive degree class
%    on ${\cal C}_{\overline{\cal M}_g}/\overline{\cal M}_g$.
%  
% \bigskip
% 
% \noindent $\bullet$
%  ???????????????????.
%  
% \bigskip 
%  
% 
%%%%%%%%%%%%%%%%%%%%%%%%%%%%%
 
%%%%%%%%%%%%%%%%%%%%%%%%%%%%%%%%%%% 
%
% \bigskip
% 
% \section{A distinguished compact sub-moduli stack }
% 
% \bigskip
% 
% \noindent $\bullet$
% Substack whose nodal curves has the property that stabilization involves contraction 
%  of ${\Bbb P}^1$-chains only. 
% 
% \bigskip
% 
% \noindent $\bullet$
% Completeness, and hence compactness, of this substack. 
%   
% \bigskip
% 
% \noindent $\bullet$
%  ???????????????????.
%  
% \bigskip 
%
%%%%%%%%%%%%%%%%%%%%%%%%%%%%%%%%%%%%%%% 

\newpage
\baselineskip 13pt
%references
{\footnotesize

\vspace{6em}

\noindent
chienhao.liu@gmail.com, chienliu@math.harvard.edu; \\
yau@math.harvard.edu

}%endfootnotesize

\end{document}